\documentclass[11pt]{article}

\usepackage{amsfonts}
\usepackage{amsthm}
\usepackage{amscd}
\usepackage{amssymb}
\usepackage{graphics}
\usepackage{amsmath}
\usepackage[english]{babel}
\usepackage[colorlinks=true,linkcolor=NavyBlue,urlcolor=RoyalBlue,citecolor=PineGreen,linktocpage=true]{hyperref}
\usepackage{bbm}
\usepackage{yfonts}
\usepackage{mathrsfs}
\usepackage{upgreek}
\usepackage{color}
\usepackage{epsfig}
\usepackage{graphicx}
\usepackage{enumerate}
\usepackage[usenames,dvipsnames]{xcolor}
\usepackage{url}
\usepackage{mathabx}
\usepackage{stmaryrd}
\usepackage{relsize}

\usepackage{geometry}
\numberwithin{equation}{section}

\newtheorem{theorem}{Theorem}[section]
\newtheorem{lemma}[theorem]{Lemma}
\newtheorem{corollary}[theorem]{Corollary}
\newtheorem{proposition}[theorem]{Proposition}

\newtheorem{claim}[theorem]{Claim}

\theoremstyle{definition}
\newtheorem{definition}[theorem]{Definition}

\theoremstyle{remark}
\newtheorem{remark}[theorem]{Remark}

\newcommand{\BF}{{\mathbb{F}}}

\newcommand{\BN}{{\mathbb{N}}}

\newcommand{\BQ}{{\mathbb{Q}}}
\newcommand{\BR}{{\mathbb{R}}}

\newcommand{\BZ}{{\mathbb{Z}}}

\newcommand{\FG}{{\mathfrak{G}}}

\newcommand{\FX}{{\mathfrak{X}}}

\newcommand{\Fa}{{\mathfrak{a}}}
\newcommand{\Fb}{{\mathfrak{b}}}
\newcommand{\Fc}{{\mathfrak{c}}}

\newcommand{\CA}{{\mathcal{A}}}

\newcommand{\CH}{{\mathcal{H}}}

\newcommand{\CM}{{\mathcal{M}}}
\newcommand{\CN}{{\mathcal{N}}}
\newcommand{\CO}{{\mathcal{O}}}

\newcommand{\CS}{{\mathcal{S}}}
\newcommand{\CT}{{\mathcal{T}}}

\newcommand{\ST}{{\mathscr{T}}}

\newcommand{\scrC}{{\mathscr{C}}}
\newcommand{\scrV}{{\mathscr{V}}}

\newcommand{\scrG}{{\mathscr{G}}}
\newcommand{\scrE}{{\mathscr{E}}}

\newcommand{\sm}{{\mathsf{m}}}
\newcommand{\sg}{{\mathsf{g}}}

\newcommand{\ind}{{\mathbbm{1}}}
\newcommand{\dist}{{\mathrm{dist}}}
\newcommand{\lk}{{\mathrm{lk}}}

\newcommand{\ls}{{\llbracket}}
\newcommand{\rs}{{\rrbracket}}

\newcommand{\qu}{{/\hspace{-2.9pt}/}}

\newcommand{\om}{{\omega}}
\newcommand{\Om}{{\Omega}}

\newcommand{\si}{{\sigma}}
\newcommand{\Si}{{\Sigma}}
\newcommand{\lam}{{\lambda}}

\title{On groups and simplicial complexes}

\author{Alexander Lubotzky\footnote{Research supported in part by the BSF, ERC, ISF and NSF.}$~^{,\dag}$,\,\, Zur Luria\footnote{Supported by Dr.~Max R\"ossler, the Walter Haefner Foundation and the ETH Foundation} \,\,and\,\, Ron Rosenthal}

\begin{document}
	\maketitle

\begin{abstract}
	The theory of $k$-regular graphs is closely related to group theory. Every $k$-regular, bipartite graph is a Schreier graph with respect to some group $G$, a set of generators $S$ (depending only on $k$) and a subgroup $H$. The goal of this paper is to begin to develop such a framework for $k$-regular simplicial complexes of general dimension $d$. Our approach does not directly generalize the concept of a Schreier graph, but still presents an extensive family of $k$-regular simplicial complexes as quotients of one universal object: the $k$-regular $d$-dimensional arboreal complex, which is itself a simplicial complex originating in one specific group depending only on $d$ and $k$. Along the way we answer a question from \cite{PR12} on the spectral gap of higher dimensional Laplacians and prove a high dimensional analogue of Leighton's graph covering theorem. This approach also suggests a random model for $k$-regular $d$-dimensional multicomplexes.
\end{abstract}


\section{Introduction}

Groups play an important role in the study of graphs, especially those with some symmetry, such as Cayley graphs. However, even general $k$-regular graphs, i.e., graphs in which the symmetry is restricted only to the local neighborhoods of vertices, are intimately connected with group theory. Let $G$ be a group with a set of generators $S$. Recall that the Schreier graph with respect to a subgroup $H\leq G$, denoted $\mathrm{Scr}(G/H;S)$, is the quotient of the Cayley graph $\mathrm{Cay}(G;S)$ by the action of $H$ (see Appendix \ref{sec:appendix}).

Regular graphs are usually Schreier graphs, but not always (see Propositions \ref{prop:bipartite_graphs} and \ref{prop:not_a_Schrier_graph}). In particular, we have the following proposition:

\begin{proposition}\label{prop:the_graph_case}
	Every connected, bipartite $k$-regular graph $X$ is a Schreier graph, namely, there exists a group $G$ with a subset of generators $S$ and a subgroup $H\leq G$, such that $X\cong \textrm{Sch}(G/H;S)$.
\end{proposition}

The proof of Proposition \ref{prop:the_graph_case} is easy: Let $V$ be the set of vertices of $X$. Since $X$ is connected, $k$-regular and bipartite, one can write the edges of $X$ as the union $k$ disjoint perfect matchings (see Appendix \ref{sec:appendix}). This gives rise to $k$ permutations $S=\{s_1,\ldots,s_k\}$ (in fact $k$ involutions) in $\CS_V$ (the symmetric group of $V$). Let $G_0$ be the subgroup of $\CS_V$ generated by $S$. The connectedness  of $X$ implies that $G_0$ acts transitively on $V$, so $V$ can be identified with $G_0/H_0$, where $H_0=\textrm{Stab}(v_0)$ is the stabilizer group for some fixed $v_0\in V$. Furthermore, one can check that $X\cong \textrm{Sch}(G_0/H_0;S)$. 

The main goal of this work is to set up a similar framework for $d$-dimensional $k$-regular simplicial complexes (for arbitrary $d,k\in\BN$), namely, to present simplicial complexes of this type as suitable quotients of some group. A natural naive way to do it would be to start with the notion of a Cayley complex of a group. This is the clique complex of the Cayley graph, i.e., a set of $(j+1)$ vertices of the Cayley graph form a $j$-cell if and only if any two of its members are connected by an edge in the Cayley graph. The Schreier complex, will be then the clique complex of the Schreier graph. However, this method is very restrictive as it gives only clique complexes, i.e., those complexes which are completely determined by their graph structure (the $1$-skeleton). Moreover, these complexes are often non-regular in the usual sense of regularity of complexes, and in general, it is not easy to determine their dimension. 

Let us set now a few definitions and then give our different approach to the above goal: For $n\in\BN=\{1,2,\ldots\}$ we use the notation $[n]=\{1,2,\ldots,n\}$ and $\ls n\rs = \{0,1,\ldots,n\}$. Let $X$ be a simplicial complex with vertex set $V$. This means that $X$ is a non-empty collection of finite subsets of $V$, called \emph{cells}, which is closed under inclusion, i.e., if $\tau\in X$ and $\sigma\subseteq\tau$, then $\sigma\in X$. The \emph{dimension of a cell} $\sigma$ is $|\sigma|-1$, and $X^{j}$ denotes the set of $j$-cells (cells of dimension $j$) for $j\geq -1$. Without loss of generality, we always assume that $X^0=V$. The dimension of $X$, which we denote by $d$, is the maximal dimension of a cell in it. We will always assume that $d$ is finite and use the abbreviation $d$-complex for a simplicial complex of dimension $d$. We say that $X$ is \emph{pure} if every cell in $X$ is contained in at least one $d$-cell. Unless stated explicitly, any simplicial complex appearing in this paper is assumed to be pure. For a $(j+1)$-cell $\tau=\{\tau_0,\ldots,\tau_{j+1}\}$, its \emph{boundary} $\partial\tau$ is defined to be the set of $j$-cells $\{\tau\backslash\{\tau_i\}\}_{i=0}^{j+1}$. The \emph{degree of a $j$-cell} $\sigma$ in $X$, denoted $\deg(\sigma)\equiv\deg_X(\si)$, is defined to be the number of $(j+1)$-cells $\tau$ which contain $\sigma$ in their boundary. The complex $X$ is called $k$-regular (or more precisely upper $k$-regular) if $\deg_X(\si)=k$ for every $\si\in X^{d-1}$. 

Going back to graphs, i.e., $d=1$, the last definition recovers the notion of $k$-regular graphs. Proposition \ref{prop:the_graph_case} showed that such bipartite graphs are Schreier graphs. In fact, the proof of Proposition \ref{prop:the_graph_case} shows a bit more: The elements $s\in S$ are all of order $2$. Therefore $G_0$ is a quotient of the infinite group $T(k)=\langle \beta_1,\ldots,\beta_k ~:~ \beta_i^2=e,\,i=1,\ldots,k\rangle$, the free product of $k$ copies of the cyclic group of order $2$. Let $\pi:T(k)\to G$ be the unique epimorphism sending $\beta_i$ to $s_i$ for $1\leq i\leq k$. By pushing $H_0$ backward to $T(k)$, thus obtaining the subgroup $H=\pi^{-1}(H_0)$, we see that $X$ is actually isomorphic to $\textrm{Sch}(T(k)/H;B)$, where $B=\{\beta_1,\ldots,\beta_j\}$. Thus, $T(k)$ is a universal object in the sense that all bipartite, $k$-regular connected graphs are Schreier graphs of it and are thus quotients of the universal Cayley graph $\ST_k:=\mathrm{Cay}(T(k);B)$. Note that $\ST_k$ is simply the $k$-regular tree.

We would like to generalize this picture to higher dimensions, but as mentioned before, doing so will lead only to Cayley complexes and Schreier complexes which are clique complexes and are not necessarily $k$-regular. We will therefore take a different approach: Let $L(k)$ be the line graph of $\ST_k$, namely, the graph whose vertices are the edges of $\ST_k$ and two vertices of $L(k)$ are connected by an edge, if as edges of $\ST_k$, they share a common vertex. Denoting by $C_k$ the cyclic group of order $k$, one can verify that $L(k)$ is a $2(k-1)$-regular graph and is, in fact, isomorphic to the Cayley graph $\textrm{Cay}(G_{1,k};S)$, where $G_{1,k} = K_0 * K_1= \langle \alpha_0,\alpha_1 ~:~ \alpha_0^k=\alpha_1^k=e\rangle$ is the free product of two copies ($K_0$ and $K_1$) of $C_k$ and the set of generators is $S=\{\alpha_0^i,\alpha_1^i ~:~ i=1,\ldots,k-1\}$. The line graph of every connected, bipartite $k$-regular graph is therefore a quotient of this graph. 

Starting with the group $G_{1,k}$, one can recover the $k$-regular tree as follows: The vertices will be $(K_0\backslash G_{1,k}) \cup (K_1 \backslash G_{1,k})$ and the edges correspond to elements of $G_{1,k}$, where $g\in G_{1,k}$ gives rise to an edge connecting $K_0g$ and $K_1g$. As $|K_ig|=k$ for every $g\in G_{1,k}$, this is a bipartite, $k$-regular graph. The fact that $G_{1,k}$ is a free product of $K_0$ and $K_1$ implies that this is the $k$-regular tree (this will be a special case of Corollary \ref{cor:bijection_of_G_and_T}). 

We generalize this picture to arbitrary dimension $d$ as follows: Let $G_{d,k}$ be the free product of $(d+1)$ copies of the cyclic group of order $k$, namely $G_{d,k}= K_0 * K_1\ldots *K_d$, where $K_i\cong C_k$. From $G_{d,k}$ we construct a $d$-dimensional simplicial complex $T_{d,k}$ as follows: Define the $0$-cells of $T_{d,k}$ to be $\{K_{\widehat{i}}g ~:~ g\in G_{d,k},\, i\in \ls d\rs\}$, where for $i\in \ls d\rs$ we define $K_{\widehat{i}} = K_1* \ldots *K_{i-1}*K_{i+1}*\ldots K_d$, and set $T_{d,k}$ to be the pure $d$-complex\footnote{Note that in a pure $d$-complex one only need to specify the $0$-cells and $d$-cells in order to recover the whole structure.} whose $d$-cells are $\{\{K_{\widehat{0}}g,K_{\widehat{1}}g,\ldots,K_{\widehat{d}}g\} ~:~ g\in G_{d,k}\}$. It turns out that $T_{d,k}$ is an arboreal complex in the sense of \cite{PR12}. It is the unique universal object of the category of $k$-regular simplicial complexes of dimension $d$ (see Proposition \ref{prop:universal_property} for a precise statement). Moreover, for every $d$-lower path connected (see Section \ref{sec:multicomplexes}), $(d+1)$-partite, $k$-regular simplicial complex $X$ there is a surjective simplicial map $\pi:T_{d,k}\to X$.

The group $G_{d,k}$ acts from the right on the right cosets of $K_{\widehat{i}}$, i.e., the vertices of $T_{d,k}$, and this action gives rise to a simplicial action of $G_{d,k}$ on $T_{d,k}$. If $H$ is a subgroup of $G_{d,k}$, then we may consider the quotient $T_{d,k}/H$. As it turns out, the quotient is not always a simplicial complex in the strict sense, but is rather a multicomplex (see Section \ref{sec:multicomplexes} for a precise definition). Thus, it is natural to extend the category we are working with to the category $\scrC_{d,k}$ of $k$-regular multicomplexes of dimension $d$. 

However, there is another delicate point here which cannot be seen in dimension 1. Before explaining it, we need the following definition:
\begin{definition}[Line graph] \label{def:line_graph}
	Let $X$ be a $d$-complex. The line graph of $X$ (also known as the dual graph of $X$), denoted $\scrG(X)=(\scrV(X),\scrE(X))$, is defined by $\scrV(X)=X^d$ and $\scrE(X) =\{\{\tau,\tau'\}\in \scrV(X)\times \scrV(X) ~:~ \tau\cap\tau'\in X^{d-1}\}$. We denote by $\dist_X=\dist : \scrV(X)\times \scrV(X)\to \BN\cup\{0\}$ the graph distance on the line graph. 
\end{definition}

It can happen that the line graphs of two non-isomorphic complexes are identical. For example, let $X$ be a $d$-dimensional simplicial complex with $d\geq 2$, and choose two vertices $v_1,v_2\in X^0$ which do not have a common neighbor in the $1$-skeleton. If we identify $v_1$ and $v_2$, we obtain a new simplicial complex $Y$, together with a surjective simplicial map $\varphi:X\to Y$ which induces an isomorphism between the line graphs $\scrG(X)$ and $\scrG(Y)$. Furthermore, one can verify that the link of $Y$ at $v_1=v_2$ is not connected. 

We show in Subsection \ref{subsec:quotient_is_link_connected}, that the quotient $T_{d,k}/ H$, as above is always link-connected (see Subsection \ref{subsec:link_connected} for definition) and obtain a one to one correspondence between the link-connected objects $\scrC_{d,k}^{lc}$ in $\scrC_{d,k}$ and subgroups of $G_{d,k}$. Along the way, we show that every $k$-regular multicomplex $Y$ in $\scrC_{d,k}$ has a unique minimal \emph{(branch) cover} $X\in \scrC_{d,k}^{lc}$ with $\scrG(X)\cong \scrG(Y)$. 

Another application of the main theorem is an high-dimensional analogue of Leighton's graph covering theorem. In our context it says that every two finite objects in the category $\scrC_{d,k}$ have a common finite (branch) covering in the category. Interestingly, we do not know how to prove this combinatorial statement without appealing to our group theoretic machinery. 

In Section \ref{sec:examples} we present some examples. One of the examples we discuss there, shows that for $q$ a prime power, the Bruhat-Tits buildings $\widetilde{A}_d$ over a local field $F$ of residue class $q$ is a quotient of $T_{d,q+1}$. Limiting ourselves to $d=2$, and comparing the spectrum of $T_{2,q+1}$ which was calculated in \cite{PR12,Ro14} and the one of $\widetilde{A}_2$, which was described in \cite{GP14}, we deduce a negative answer to a question asked in \cite{PR12} about the spectral gap of high-dimensional Laplacians. 

The basic idea of this paper is quite simple, but the precise formulation needs quite a lot of notation, definitions and preparation. This is done in Sections \ref{sec:multicomplexes}-\ref{sec:group_interpretation_of_T_d_k}, while the correspondence is proved in Sections \ref{sec:from_subgroup_to_elements_of_the_category} and \ref{sec:from_simplicial_complexes_to_subgroups_and_back}. In Section \ref{sec:further_relations} we discuss further relations between properties of subgroups and their associated multicomplexes. In Section \ref{sec:examples} we present various examples: we describe the complexes associated with some natural subgroups of $G_{d,k}$ and the subgroups associated with some interesting complexes. 

Our approach enables one in principle to build systematically all finite, partite $k$-regular multicomplexes. First one may generate the link-connected ones as the quotients of $T_{d,k}$ by a subgroup, and then, all of them by identifications of cells as above (see also Section \ref{sec:random_multicomplexes}). In particular, we get a random model of such complexes (see Section \ref{sec:random_multicomplexes}). A drawback of our method is that in many cases (in fact in ``most'' cases) we get multicomplexes and not complexes. Every such multicomplex gives rise to a simplicial complex (by ignoring the multiplicity of the cells), but it is not so easy to decide whether the original object is already a simplicial complex or merely a multicomplex (see Subsection \ref{subsec:the_intersection_property} for more on this issue). We plan to come back to this random model in the future.


\section{Preliminaries}\label{sec:preliminaries}

In this short section, we collect some additional definitions and notation from the theory of abstract simplicial complexes which are used throughout the paper. 

Given a $d$-complex $X$ and $-1\leq j\leq d$, the \emph{$j$-th skeleton} of $X$, denoted $X^{(j)}$, is the set of cells in $X$ of dimension at most $j$, that is $X^{(j)} := \bigcup_{i=-1}^{j} X^i$. We say that a $d$-complex $X$ has a \emph{complete skeleton} if $X^j = \binom{V}{j+1}$ for every $j<d$.  

For a cell $\si\in X$, define its coboundary, denoted $\delta(\si)\equiv\delta_X(\si)$, to be $\delta(\si)=\{\tau\in X ~:~ \si\subset \tau,\, |\tau\setminus \si|=1\}$, which in particular satisfies $\deg_X(\si)=|\delta_X(\si)|$. 

For $1\leq j\leq d$, we say that $X$ is \emph{$j$-lower path connected} if for every $\si,\si'\in X^j$ there exists a sequence $\si=\si^0,\si^1,\ldots,\si^m=\si'$ of $j$-cells in $X$ such that $\si^{i-1}\cap \si^i \in X^{j-1}$ for every $1\leq i\leq m$. 
 
Let $X$ and $Y$ be a pair of $d$-complexes. We say that $\varphi:X\to Y$ is a simplicial map, if $\varphi:X^0\to Y^0$ is a map, extended to the remaining cells by $\varphi(\{\si_0,\ldots,\si_j\})=\{\varphi(\si_0),\ldots,\varphi(\si_j)\}$ such that $\varphi(\si)\in Y^j$ for every $0\leq j\leq d$ and $\si\in X^j$.

Given $\rho\in X$, the link of $\rho$ is a $(d-|\rho|)$-dimensional complex defined by 
\begin{equation}
	\lk_X(\rho) = \{\si\in X ~:~ \rho\amalg\si \in X\}
\end{equation}
where we use the notation $\rho \amalg\si $ when the union is disjoint, i.e. $\rho\cap \si=\emptyset$. 

\begin{definition}[Nerve complex]\label{def:nerve_complex}
	Let $\CA=(\CA_i)_{i\in I}$ be a family of nonempty sets. The nerve complex of $\CA$, denoted $\CN(\CA)$, is the simplicial complex with vertex set $I$, such that $\si\in \CN(\CA)$ for $\si \subseteq I$ if and only if $\bigcap_{i\in \si}\CA_i \neq \emptyset$. 
\end{definition}


\section{Multicomplexes and the category $\scrC_{d,k}$}\label{sec:multicomplexes}

In this section we introduce a category of certain combinatorial objects, which is the main topic of this paper. 

\subsection{Multicomplexes}\label{subsec:multicomplexes}

We start by describing the notion of a multicomplex, see also \cite{Kue15}. For a set $A$ and a (multiplicity) function $\sm:A\to \BN$, define $A_\sm = \{(a,r) ~:~ a\in A,\, r\in [\sm(a)]\}$. Similarly, for $a\in A$, denote $A_\sm(a)=\{(a,r) ~:~ r\in [\sm(a)]\}$. Let $V$ be a countable set and $\widetilde{X}=(X,\sm,\sg)$ a triplet, where 
\begin{itemize}
	\item $X$ is a simplicial complex with vertex set $V$,
	\item $\sm:X\to \BN$ is a function (called the multiplicity function) satisfying $\sm(\si)=1$ for $\si\in X^{(0)}$,
	\item $\sg : \{((\tau,r),\si)\in X_\sm \times X ~:~ \si\in\partial \tau\}\to X_\sm$ is a map (called the gluing map) satisfying $\sg((\tau,r),\si)\in A_\sm(\si)$ for every $(\tau,r)\in X_\sm$ and $\si\in \partial \tau$. So, $\sg$ tells us which copy of $\si$ is in the boundary of the $r$-th copy of $\tau$. 
\end{itemize}  

Elements of $X_\sm$ are called multicells, and are often denoted by $\Fa,\Fb,\Fc,\ldots$. We denote by $\iota:X_\sm\to X$ the forgetful map, namely, $\iota((\tau,r))=\tau$ for every $(\tau,r)\in X_\sm$.

As in the case of simplicial complexes, we define the dimension of a multicell $\Fa\in X_\sm$ to be $\dim_{\widetilde{X}}(\Fa):=\dim_X(\iota(\Fa))=|\iota(\Fa)|-1$. The set of multicells of dimension $j$ (abbreviated $j$-multicells) is denoted by $X_\sm^j=\{\Fa\in X_\sm ~:~ \dim(\Fa)=j\}$. Since the multiplicity of $0$-cells in any  multicomplex is one by definition, we tacitly identify $X^0$ and $X_\sm^0$ using the identification $v \leftrightarrow (v,1)$. For $-1\leq j\leq d$, we denote by $X_\sm^{(j)}=\bigcup_{i=-1}^j X_\sm^i$ the $j$-skeleton of $X_\sm$ (see Figure \ref{fig:multi_complex_1} for an illustration of a multicomplex).

\begin{figure}[h]
	\begin{center}
	\includegraphics[scale=0.55]{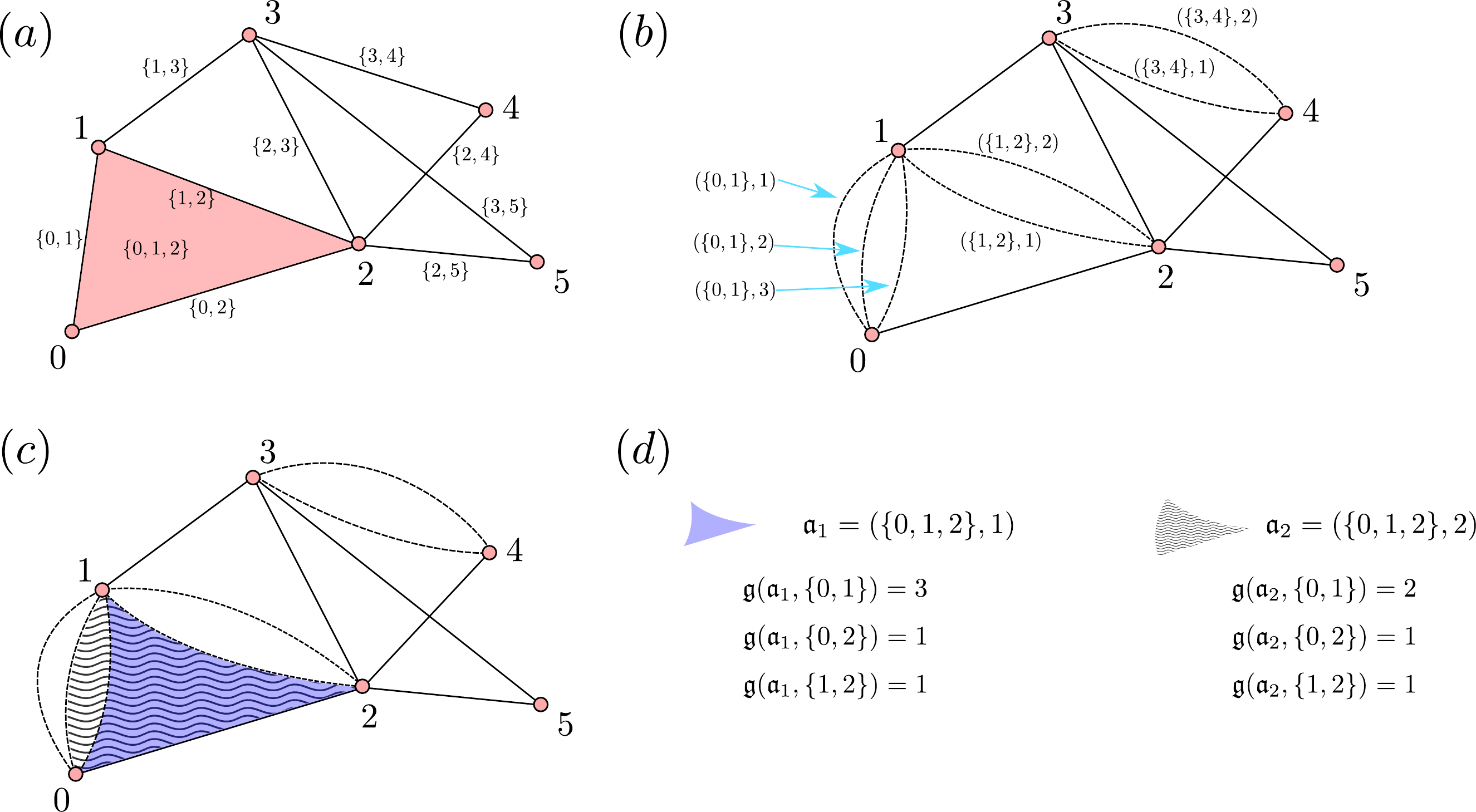}
	\caption{An example of a triple $(X,\sm,\sg)$. $(a)$ Illustration of the complex $X$ (note that unlike most complexes discussed in this paper, $X$ is not pure). $(b)$ The multiplicity of the edges is described. $(c)$ The multicomplex $X$ contains a unique triangle which has multiplicity $2$ in the multicomplex. The two $2$-multicells are illustrated by a wavy and a filled triangle. $(d)$ The gluing function of the two $2$-multicells is described. \label{fig:multi_complex_1}}
	\end{center}
\end{figure}

We define the multiboundary of a multicell $\Fa=(\tau,r)$ by
\begin{equation}
	\partial^m \Fa=\begin{cases}
		\{\sg(\Fa,\si) ~:~ \si\in \partial \iota(\Fa)\} &\text{ if }  \dim(\Fa)\geq 0\\
		\emptyset & \text{ if }   \dim(\Fa)=-1
		\end{cases},
\end{equation} 
namely, the set of multicells of dimension $\dim(\Fa)-1$ which are glued to $\Fa$. 

Using the multiboundary, one can define the set of multicells contained in a given multicell $\Fa\in X_\sm$ as follows: Declare $\Fa$ to be contained in itself, and for $j=\dim(\Fa)-1$ define $\Fb\in X^{j}_m$ to be contained in $\Fa$ if and only if $\Fb\in \partial^m \Fa$. Proceeding inductively from $j=\dim(\Fa)-2$ to $j=-1$, declare $\Fc\in X^j_m$ to be contained in $\Fa$ if there exists $\Fb\in X^{j+1}_m$ contained in $\Fa$ such that $\Fc\in \partial^m \Fb$. Note that this defines a partial order on the family of multicells which we denote by $\preceq$. 

\begin{definition}[Multicomplex] 
	Let $V$ be a countable set and $\widetilde{X}=(X,\sm,\sg)$ be a triplet as above. We say that $\widetilde{X}$ is a multicomplex if it satisfies the following consistency property: for every $\Fa=(\tau,r)\in X_\sm$ and every pair $\Fb=(\si,s)$ and $\Fb'=(\si',s')$ contained in $\Fa$ such that $\dim(\Fb)=\dim(\Fb')$ and $\rho:=\si\cap \si'\in X^{\dim(\Fb)-1}$ it holds that $\sg(\Fb,\rho) = \sg(\Fb',\rho)$. 
	
	 We say that $\widetilde{X}$ is a $d$-dimensional multicomplex ($d$-multicomplex) if the associated complex $X$ is $d$-dimensional. 
\end{definition}

Since the $0$-cells of a multicomplex always have multiplicity one, the consistency property is always satisfied for triples $(X,\sm,\sg)$, where $X$ is a simplicial complex of dimension $d\leq 2$. An illustration of the consistency property in the $3$-dimensional case can be found in Figure \ref{fig:multi_complex_2}.

\begin{figure}[h]
	\begin{center}
	\includegraphics[scale=0.5]{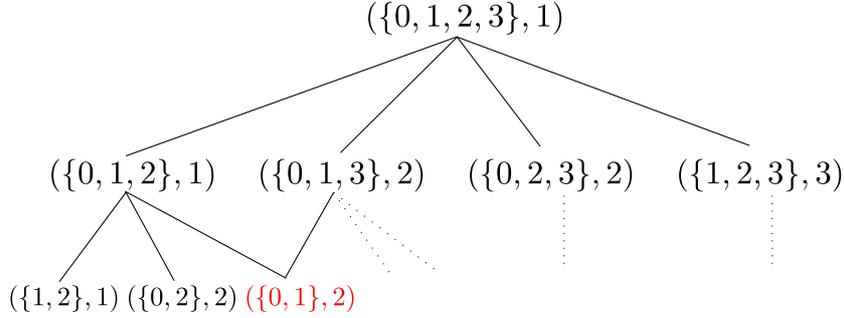}
	\caption{An example of the consistency property in dimension $3$. Assume that the complex $X$ is the full complex on $4$ vertices and that $(X,\sm,\sg)$ is a triplet which defines a multicomplex. Assuming that $\sg((\{0,1,2,3\},1),\{0,1,2\})=(\{0,1,2\},1)$ and $\sg((\{0,1,2,3\},1),\{0,1,3\})=(\{0,1,3\},2)$, the consistency property gurantees that on the joint edge of $\{0,1,2\}$ and $\{0,1,3\}$, that is $\{0,1\}$, we must have $\sg((\{0,1,2\},1),\{0,1\})=\sg((\{0,1,3\},2),\{0,1\})$. \label{fig:multi_complex_2}}
	\end{center}
\end{figure}

\begin{remark}\label{rem:multi_in_0_cells}
	Note that the requirement not to have multiplicity in the $0$-cells of a multicomplex $\widetilde{X}=(X,\sm,\sg)$, i.e. $\sm(v)=1$ for every $v\in X^0$, is only made in order to fix a description of the multicomplex. Indeed, any triplet $\widetilde{X}=(X,\sm,\sg)$, in which one also allows multiplicity in the $0$-cells can be transformed into a multicomplex by declaring the set of $0$-cells to be $X_\sm^0$ and making the appropriate changes in the cell names. One can verify that this changes the multiplicity of each $0$-cell to $1$ while preserving the structure of the multicomplex, namely, the partial order of containment. 
\end{remark}

One can verify that if $\widetilde{X}$ is a multicomplex, then for every $\Fa\in X_\sm$ there exists a unique bijection $f:\{\si\in X ~:~ \sigma\subseteq \iota(\Fa)\}\to \{\Fb\in X_\sm ~:~ \Fb\preceq \Fa\}$ such that $\iota\circ f$ is the identity, that is, for every $\si\subset \iota(\Fa)$ there exists a unique multicell $\Fb\preceq \Fa$ with $\iota(\Fb)=\si$.

A $d$-multicomplex is called pure if each of its multicells is contained in at least one $d$-multicell. Throughout the paper we assume that all multicomplexes are pure. 

Similarly to the case of simplicial complexes, for $\Fa\in X_\sm$, we set $\delta_{\widetilde{X}}(\Fa)\equiv\delta(\Fa) =\{\Fb\in X_\sm ~:~\Fa\in \partial^m\Fb\}$, and define the degree of a cell $\Fa\in X_\sm^j$ by $\deg_{\widetilde{X}}(\Fa)\equiv\deg(\Fa):=|\delta_{\widetilde{X}}(\Fa)|$. 
We say that a $d$-multicomplex is $k$-regular (or more precisely upper $k$-regular), if the degree of any $(d-1)$-multicell in $X_\sm$ is $k$. 

Two multicells of the same dimension are called neighbors if they contain a common codimension $1$ multicell\footnote{Note that according to the above definition, each multicell is a neighbor if itself.}. A sequence of multicells of the same dimension $\tau(0),\tau(1),\ldots,\tau(m)$ is called a path if $\tau(r-1)$ and $\tau(r)$ are neighbors for every $1\leq r\leq m$. For $1\leq j\leq d$, we say that $\widetilde{X}$ is $j$-lower path connected if for every pair of $j$-multicells $\Fa,\Fa'\in X_\sm^j$, there exists a path from $\Fa$ to $\Fa'$. 

Let $\widetilde{X}=(X,\sm,\sg)$ and $\widetilde{Y}=(Y,\sm',\sg')$ be a pair of $d$-multicomplexes. We say that $\widetilde{\varphi}$ is a simplicial multimap from $\widetilde{X}$ to $\widetilde{Y}$, if it is a map from $X_\sm$ to $Y_{\sm'}$ such that the following conditions hold:
\begin{itemize}
	\item there exists a simplicial map $\varphi:X\to Y$ such that $\iota\circ \widetilde{\varphi}(\Fa) = \varphi(\iota(\Fa))$ for every $\Fa\in X_\sm$, that is, $\widetilde{\varphi}$ extends a simplicial map $\varphi$ by sending each multicell associated with a cell $\si$ to a multicell associated with the cell $\varphi(\si)$. 
	\item $\sg'(\widetilde{\varphi}(\Fa),\varphi(\si))=\widetilde{\varphi}(\sg(\Fa,\si))$ for every $\Fa\in X_\sm$ and $\si\in \partial \iota(\Fa)$, that is, $\widetilde{\varphi}$ preserves the gluing structure of $\widetilde{X}$ by gluing $\widetilde{\varphi}(\Fa)$ to the copy of $\varphi(\si)$ given by $\widetilde{\varphi}(\sg(\Fa,\si))$, for $\si\in \partial \iota(\Fa)$. 
\end{itemize}
The simplicial map $\varphi$ associated with the simplicial multimap $\widetilde{\varphi}$ is called the base map of $\widetilde{\varphi}$. Note that one can recover the base map of a simplicial multimap $\widetilde{\varphi}$, by sending $\si\in X$ to $\iota\circ \widetilde{\varphi}((\si,1))$.


\subsection{Link-connected multicomplexes}\label{subsec:link_connected}

In this subsection we wish to identify a special set of multicomplexes which we call link-connected. We start by defining the link of a multicell in a multicomplex. In order to give a simple description of the links we describe them using the indexing of the multicells in the original multicomplex. In particular, this might lead to the existence of $0$-cells with multiplicity. The description can be transformed into a ``formal'' multicomplex, i.e., removing the multiplicity of the $0$-cells as explained in Remark \ref{rem:multi_in_0_cells} and reindexing the multicells.  

\begin{definition}[link of a multicell]
	Let $\widetilde{X}=(X,\sm,\sg)$ be a multicomplex, let $\Fa\in X_m$ be a multicell,  and denote by $\rho=\iota(\Fa)$ the corresponding cell in $X$. The link of $\Fa$ is a $(d-|\rho|)$-multicomplex, denoted by $\mathrm{lk}_{\widetilde{X}}(\Fa)=(\mathrm{lk}_{X}(\Fa),\sm_\Fa,\sg_\Fa)$, where 
	\begin{equation}
		\mathrm{lk}_{X}(\Fa) = \{\tau \in X ~:~ \exists \Fb\in X_\sm \text{ containing } \Fa \text{ such that } \iota(\Fb)= \rho\amalg \tau\}.
	\end{equation}
	The multiplicity of $\tau\in \lk_X(\Fa)$ is 
	\begin{equation}
		\sm_\Fa(\tau) = |\{\Fb\in X_\sm ~:~ \Fb \text{ contains } \Fa \text{ and } \iota(\Fb)=\rho\amalg \tau\}|.
	\end{equation}
	Instead of using the set $\lk_X(\Fa)_{\sm_\Fa}$ to denote the multicells, we use the natural indexing induced from the original multicomplex, that is, we denote the multicells associated with $\tau\in \lk_X(\Fa)$ by 
	\begin{equation}
		M_\Fa(\tau):=\big\{(\tau,i) ~:~ (\tau\amalg \rho,i) \text{ contains } \Fa \text{ in }\widetilde{X}\big\},
	\end{equation}
	so that $\sm_\Fa(\tau)=|M_\Fa(\tau)|$. 

Finally, the gluing function $\sg_\Fa$ is defined as follows: given a multicell $(\tau,i)$ in $\lk_{\widetilde{X}}(\Fa)$ and $\si\in \partial \tau$ let $\sg_\Fa((\tau,i),\si) = (\si,j)$, where $j$ is the unique index such that $\sg((\tau\amalg \rho,i),\si\amalg \rho) = (\si\amalg \rho,j)$.
\end{definition}

Note that if $\widetilde{X}$ has no multiplicity, namely, it is a standard simplicial complex, we recover the standard definition of the link of a cell.

\begin{definition}[link-connected multicomplex]
	Let $\widetilde{X}=(X,\sm,\sg)$ be a $d$-multicomplex. We say that $\widetilde{X}$ is \emph{link-connected} if for every $\Fa\in X^{(d-2)}_\sm$ the link $\lk_{\widetilde{X}}(\Fa)$ is connected, that is, its one skeleton is a connected multigraph.
\end{definition}

\begin{figure}[h]
	\begin{center}
	\includegraphics[scale=1.6]{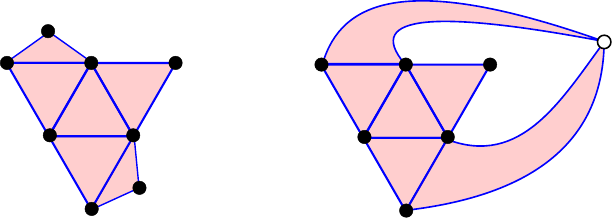}
	\caption{On the left, a link-connected complex. On the right, a complex which is not link-connected (as can be seen by observing the link of the white $0$-cell). \label{fig:not_link_connected}}
	\end{center}
\end{figure}

\begin{proposition}\label{prop:equivalent_defn_of_lc}
	Let $\widetilde{X}=(X,\sm,\sg)$ be a pure $d$-multicomplex. Then, the following are equivalent:
	\begin{enumerate}[(1)]
		\item $\widetilde{X}$ is link-connected.
		\item For every $-1\leq j\leq d-2$ and $\Fa\in X^j_\sm$, the link $\lk_{\widetilde{X}}(\Fa)$ (which is a $(d-j-1)$-dimensional multicomplex) is $(d-j-1)$-lower path connected.
	\end{enumerate}
\end{proposition}

\begin{proof}
	$(1)\Rightarrow (2)$. Fix $-1\leq j\leq d-2$ and $\Fa\in X^j_\sm$. We prove the claim by induction on $0\leq i\leq d-j-2$ of the following statement: 
\begin{quote}	
	For every pair of $(d-j-1)$-multicells $\Fb,\Fb'$ in $\lk_{\widetilde{X}}(\Fa)$, there exists $l\in \BN\cup\{0\}$ and a sequence $\Fb=\Fb_0,\Fb_1,\ldots,\Fb_l=\Fb'$ of $(d-j-1)$-multicells in $\lk_{\widetilde{X}}(\Fa)$, such that for every $1\leq r\leq l$, the multicells $\Fb_r$ and $\Fb_{r-1}$ contain a common $i$-multicell. 
\end{quote}	

	For $i=0$, this follows from the assumption that $\widetilde{X}$ is link connected and pure. Indeed, let $\Fb,\Fb'$ be as above, and fix a pair of vertices $v\in \Fb$ and $v'\in \Fb'$. Since $\widetilde{X}$ is link-connected, one can can find a sequence of vertices $v=v_0,v_1,\ldots,v_l=v'$ such that each consecutive pair of vertices belong to some common multiedge in $\lk_{\widetilde{X}}(\Fa)$. For $1\leq i\leq l$, let $\Fb_i$ be any $(d-j-1)$-multicell containing the multiedge connecting $v_i$ and $v_{i+1}$ along the path (which must exist since the link is pure), we conclude that $\Fb,\Fb_1,\ldots,\Fb_l,\Fb'$ is the required path of $(d-j-1)$-multicells. 
	
	Next, assume the statement holds for some $0\leq i<d-j-2$ and let $\Fb,\Fb'$ be a pair of $(d-j-1)$-multicells in $\lk_{\widetilde{X}}(\Fa)$. By assumption, one can find $l\in\BN\cup\{0\}$ and sequence of $(d-j-1)$-multicells $\Fb=\Fb_0,\ldots,\Fb_l=\Fb'$ in $\lk_{\widetilde{X}}(\Fa)$, such that, for every $1\leq r\leq l$, the multicells $\Fb_{r-1}$ and $\Fb_{r}$ contain a common $i$-multicell. Hence, in order to complete the induction step, we only need to show that for every $1\leq r\leq l$, one can find a path of $(d-j-1)$-multicells connecting $\Fb_{r-1}$ and $\Fb_r$, such that any consecutive pair of the multicells share a common $(i+1)$-multicell. 
	To this end, fix $1\leq r\leq l$, and denote by $\Fc=(\rho,s)$ a common $i$-multicell contained in $\Fb_{r-1}$ and $\Fb_r$. Using the definition of the link $\lk_{\widetilde{X}}(\Fa)$, we obtain that $\widetilde{\Fc}:=(\rho\amalg \iota(\Fa),s)$ is a $(j+i+1)$-multicell in $\widetilde{X}$. Since $(a)$ $\lk_{\widetilde{X}}(\widetilde{\Fc}) = \lk_{\lk_{\widetilde{X}}(\Fa)}(\Fc)$, $(b)$ by assumption, $\lk_{\widetilde{X}}(\widetilde{\Fc})$ is connected, and $(c)$ $\Fb_r$ and $\Fb_{r-1}$ both have corresponding $0$-cells $v$ and $v'$ in the link $\lk_{\widetilde{X}}(\widetilde{\Fc})$; it follows that one can find a path $v=v_0,\ldots,v_t=v'$ in $\lk_{\widetilde{X}}(\widetilde{\Fc})$, connecting $v$ and $v'$. Using the correspondence from the link to its original complex, the path $v=v_0,\ldots,v_t=v'$, can be lifted from $\lk_{\widetilde{X}}(\widetilde{\Fc}) = \lk_{\lk_{\widetilde{X}}(\Fa)}(\Fc)$ back to $\lk_{\widetilde{X}}(\Fa)$, yielding a sequence of  $(d-j-1)$-multicells from $\Fb$ to $\Fb'$, such that each pair of consecutive $(d-j-1)$-multicells along the sequence share a common $(i+1)$-multicell, thus completing the proof of the induction step.
	
	$(2)\Rightarrow (1)$. Let $\Fa\in X_\sm$, and let $v,v'$ be two $0$-cells in $\lk_{\widetilde{X}}(\Fa)$. Since $\widetilde{X}$ is pure, one can find two $(d-j-1)$-multicells $\Fb,\Fb'$ in $\lk_{\widetilde{X}}(\Fa)$ such that $v$ is contained in $\Fb$ and $v'$ is contained in $\Fb'$. Using $(2)$, one can find $l\in\BN\cup\{0\}$ and a sequence $\Fb=\Fb_0,\Fb_1,\ldots,\Fb_l=\Fb'$ of $(d-j-1)$-multicells in $\lk_{\widetilde{X}}(\Fa)$ such that $\si_i:=\iota(\Fb_i)\cap\iota(\Fb_{i-1})\in X^{d-j-2}$ and $\sg(\Fb_i,\si_i)=\sg(\Fb_{i-1},\si_i)$ for every $1\leq i\leq l$. Defining $v_0=v$, $v_l=v'$, and $v_i$ for $1\leq i\leq l-1$ to be any vertex in $\si_i$, we obtain a sequence of vertices, such that each consecutive pair of vertices belong to some common multiedge. Hence $\lk_{\widetilde{X}}(\Fa)$ is connected. 
\end{proof}


\subsection{Colorable multicomplexes} \label{subsec:colorable_multicomplexes}

Next, we turn to discuss the notion of coloring of a $d$-multicomplex.

\begin{definition}[Colorable multicomplexes]
	A $d$-complex $X$ is called colorable if there exists a coloring of its $0$-cells by $(d+1)$ colors, $\gamma:X^0\to \ls d\rs$, such that the $0$-cells contained in any $d$-cell have distinct colors. A $d$-multicomplex $\widetilde{X}=(X,\sm,\sg)$ is said to be \emph{colorable} if the associated $d$-complex $X$ is colorable. 
\end{definition}

If $\widetilde{X}$ is a colorable $d$-multicomplex and $\gamma:X^0\to \ls d\rs$ is a coloring, one can extend the coloring to all the cells and multicells of $\widetilde{X}$ as follows. Let $1\leq j\leq d$. We color the $j$-cells of $\widetilde{X}$ using $\binom{d+1}{j+1}$ colors by defining the color of $\rho \in X^j$  to be $\{\gamma(v) ~:~ v\in \rho\}$ and coloring the $j$-multicells by $\gamma(\Fa) = \gamma(\iota(\Fa))$. Note that this is well defined, i.e., all $j$-cells are colored by exactly $(j+1)$ colors from $\ls d\rs$, since the multicomplex is pure. With a slight abuse of notation, we use $\gamma$ to denote the coloring of all cells and multicells of $X$ and $\widetilde{X}$. Since $0$-cells of any $d$-cell are colored with distinct colors, for every $\tau\in X^d$, the map $\gamma$ induces a bijection between $\{\rho\in X ~:~ \rho\subseteq \tau\}$ and subsets of $\ls d\rs$. Similarly, for every $\Fa\in X_\sm^d$, the map $\gamma$ induces a bijection between $\{\Fb\in X_\sm ~:~ \Fb\preceq \Fa\}$ and subsets of $\ls d\rs$.


\subsection{Ordering of a multicomplex}

Recall that $C_k$ is the cyclic group of order $k$, and for a set $B$ denote by $\CS_B$ the permutation group of $B$.

\begin{definition}
	Let $k,d\geq 1$, and assume that $\widetilde{X}=(X,\sm,\sg)$ is a $d$-multicomplex all of whose $(d-1)$-multicells have degree at most $k$. We call $\om=(\om_{\Fb})_{\Fb\in X^{d-1}_\sm}$ a \emph{$k$-ordering}, if for every $\Fb\in X^{d-1}_\sm$, the element $\om_{\Fb}:C_k\to \CS_{\delta_{\widetilde{X}}(\Fb)}$ is a transitive homomorphism. In other words, for every $\Fb\in X_\sm^{d-1}$, the map $\om_\Fb$ is a group homomorphism satisfying the property: for every $\Fa,\Fa'\in \delta(\Fb)$, there exists $\beta\in C_k$ such that $\om_{\Fb}(\beta).\Fa=\Fa'$, where $\om_\Fb(\beta).\Fa$ is the action of the permutation $\om_{\Fb}(\beta)$ on the $d$-multicell $\Fa$. 
\end{definition}

\begin{remark}$~$\vspace{-3pt}
\begin{enumerate}[(1)]
	\item	One can also consider the case $k=\infty$, constructed in the same way with the group $\BZ$ instead of $C_k$. In this case one obtains a family of multicomplexes without any restriction on the degrees of the $(d-1)$-multicells. In fact, all the results that follow can be generalized to the case $k=\infty$ by making the appropriate changes in the definitions.
	\item For a colorable $d$-multicomplex with coloring $\gamma$, one can also consider the more general situation in which the $(d-1)$-multicells of the same color are $k$-ordered where $k$ depends on the color. As before, our discussion can be extended to cover this case as well by suitable changes in the definitions. 
	\item We choose to work with orderings that are based on the cyclic group for  convenience. In fact, one can set $G$ to be any group of order $k$ and work with orderings of the form $\om=(\om_\Fb)_{\Fb\in X_\sm^{d-1}}$, where $\om_{\Fb}: G \to \CS_{\delta_{\widetilde{X}}(\Fb)}$ is a transitive homomorphism. 
\end{enumerate}
\end{remark}

\begin{claim}\label{clm:degree_divides_k}
	If $\widetilde{X}=(X,\sm,\sg)$ is a $d$-multicomplex and $\om$ is a $k$-ordering of $\widetilde{X}$, then $\deg_{\widetilde{X}}(\Fb)$ divides $k$ for every $\Fb\in X_m^{d-1}$. 
\end{claim}

\begin{proof}
	Fix $\Fb\in X_\sm^{d-1}$. Since $\om_\Fb:C_k \to \CS_{\delta(\Fb)}$ is a transitive homomorphism, it follows that $\deg_{\widetilde{X}}(\Fb) = [C_k : N]$, where $N\leq C_k$ is the stabilizer subgroup of a fixed $d$-multicell containing $\Fb$. Hence, $\deg_{\widetilde{X}}(\Fb)$ divides $k$. 
\end{proof}

Throughout the remaining of this paper, whenever a confusion may not occur, we refer to a $k$-ordering simply as an ordering, in which case the appropriate $k$ should be clear from the context. 


\subsection{The category $\scrC_{d,k}$.} 
	We define $\scrC_{d,k}$ to be the category of quartets $(\widetilde{X},\gamma,\om,\Fa_0)$ where,
	\begin{itemize}
		\item $\widetilde{X}=(X,\sm,\sg)$ is a colorable, pure $d$-dimensional multicomplex which is $d$-lower path connected satisfying $\max_{\Fb\in X^{d-1}_\sm}\deg_{\widetilde{X}}(\Fb)\leq k$.
		\item $\gamma$ is a coloring of $X$.
		\item $\om$ is a $k$-ordering of $\widetilde{X}$.
		\item $\Fa_0\in X_m^d$ is the root of the multicomplex, i.e., a fixed $d$-multicell.
	\end{itemize}	  
	
Given a pair of objects $(\widetilde{X},\gamma,\om,\Fa_0)$ and $(\widetilde{Y},\widehat{\gamma},\widehat{\om},\widehat{\Fa}_0)$ in $\scrC_{d,k}$, we say that $\widetilde{\varphi}$ is a morphism from $(\widetilde{X},\gamma,\om,\Fa_0)$ to $(\widetilde{Y},\widehat{\gamma},\widehat{\om},\widehat{\Fa}_0)$ if $\widetilde{\varphi} :\widetilde{X}\to \widetilde{Y}$ is a simplicial multimap which preserves the root, coloring and ordering, namely:
\begin{itemize}
	\item $\widetilde{\varphi}(\Fa_0)=\widehat{\Fa}_0$,
	\item $\widehat{\gamma}\circ \widetilde{\varphi} = \gamma$
	\item $\widetilde{\varphi}(\om_{\Fb}(\beta).\Fa) = \widehat{\om}_{\widetilde{\varphi}(\Fb)}(\beta).\widetilde{\varphi}(\Fa)$ for every $\Fb\in X_m^{d-1}$, $\Fa\in \delta_{\widetilde{X}}(\Fb)$ and $\beta\in C_k$.
\end{itemize}

For future use we denote by $\scrC_{d,k}^{lc}$ the set of objects $(\widetilde{X},\gamma,\om,\Fa_0)\in \scrC_{d,k}$ such that $\widetilde{X}$ is link connected. For example, the complex on the left in Figure \ref{fig:not_link_connected} (together with a choice of coloring, ordering and a root) belongs to $\scrC_{2,2}^{lc}$, while the complex on the right in Figure \ref{fig:not_link_connected} belongs to $\scrC_{2,2}\setminus \scrC_{2,2}^{lc}$. 


\section{Group action on elements of $\scrC_{d,k}$}\label{sec:group_action}

In this section we describe a left action of a specific group, denoted $G_{d,k}$, on the $d$-multicells of an object in the category $\scrC_{d,k}$.

\subsection{The group $G_{d,k}$.} 
For natural numbers $d\geq 1$ and $k\geq 1$, recall that $C_k$ is the cyclic group of order $k$, and define 
\begin{equation}
	G_{d,k} = \overset{d}{\underset{i=0}{\Asterisk}}C_k = \Big\langle\alpha_0,\ldots,\alpha_d\, \Big|\, \alpha_i^k=e \text{ for }i=0,\,\ldots d\Big\rangle
\end{equation}
to be the free product of $(d+1)$ copies of $C_k$. 

Every element of $G_{d,k}$ can be written as a word of the form $\alpha_{j_m}^{l_m} \ldots \alpha_{j_2}^{l_2} \alpha_{j_1}^{l_1}$ for some $m\geq 0$, $l_1,\ldots,l_m\in \BZ$ and $j_1,\ldots,j_m\in \ls d\rs$. The length of the word $\alpha_{j_m}^{l_m} \ldots \alpha_{j_2}^{l_2} \alpha_{j_1}^{l_1}$ is defined to be $m$. We say that $\alpha_{j_m}^{l_m} \ldots \alpha_{j_2}^{l_2} \alpha_{j_1}^{l_1}$ is \emph{reduced} if  $l_1,\ldots,l_m\in \ls k-1\rs \setminus \{0\}$ and $j_i \neq j_{i+1}$ for every $1\leq i\leq m-1$. It is well known that every element of $G_{d,k}$ is represented by a unique reduced word. In particular, the identity is represented by the empty word. 

Any word $\alpha_{j_m}^{l_m} \ldots \alpha_{j_2}^{l_2} \alpha_{j_1}^{l_1}$, representing an element $g\in G$, can be transformed into any other word (and in particular into the reduced one) representing the same element $g$ using the following transformations. 
\begin{equation}\label{eq:basics_steps}
\begin{aligned}
	&(a) \text{ Replacing } l_i \text{ by } l'_i \text{ for some }1\leq i\leq m \text{ and } l'_i \text{ such that } l_i=l'_i \text{ mod }k.\\
	&(b) \text{ If } j_i = j_{i+1} \text{ for some } 1\leq i\leq m, \text{ replacing }  \alpha_{j_{i+1}}^{l_{i+1}}\alpha_{j_i}^{l_i} \text{ by } \alpha_{j_i}^{l_i+l_{i+1}}.\\
	&(c) \text{ Replacing } \alpha_{j_i}^{l_i} \text{ by } \alpha_{j'_i}^{l'_i}\alpha_{j''_i}^{l''_i} \text{ with } j'_i=j''_i=j_i  \text{ and } l'_i + l''_i = l_i. \\
	&(d) \text{ Deleting }\alpha_i^{l_i} \text{ if }l_i=0 \text{ mod }k, \text{ or adding }\alpha_{i_*}^{0} \text{ between }\alpha_i^{l_i} \text{ and }\alpha_{i+1}^{l_{i+1}}.
\end{aligned}
\end{equation}


\subsection{The left action of $G_{d,k}$ on $d$-multicells}

Let $(\widetilde{X},\gamma,\om,\Fa_0)\in\scrC_{d,k}$ with $\widetilde{X}=(X,\sm,\sg)$. We define an action of $G_{d,k}$ on the $d$-multicells of $\widetilde{X}$ using the coloring $\gamma$ and the ordering $\om$.

For $\Fa\in X_\sm^d$, $i\in \ls d\rs$ and $l\in\BZ$, define $\alpha_i^l.\Fa$ as follows. Let $\Fb\in X_\sm^{d-1}$ be the unique $(d-1)$-multicell in $\partial^m \Fa$ of color $\ls d\rs\setminus \{i\}$. Then, set $\alpha_i^l.\Fa = \om_{\Fb}(\alpha_i^l).\Fa$. Next, given any $g\in G$ such that $g=\alpha_{i_m}^{l_m}\ldots\alpha_{i_2}^{l_2}\alpha_{i_1}^{l_1}$ for some $m\geq 0$, $i_1,\ldots,i_m\in \ls d\rs$ and $l_1,\ldots,l_m\in\BZ$, set $g.\Fa = \alpha_{i_m}^{l_m}.(\ldots (\alpha_{i_2}^{l_2}.(\alpha_{i_1}^{l_1}.\Fa))\ldots)$.
	
	First, we show that this mapping is well-defined, i.e., that any pair of words representing the same group element $g$ acts on the $d$-multicells in the same way.
	Once this is shown, we immediately conclude that the mapping defines an action of $G_{d,k}$ on $X_\sm^d$. Let $g=\alpha_{i_m}^{l_m}\ldots\alpha_{i_2}^{l_2}\alpha_{i_1}^{l_1}$ and $\Fa\in X_\sm^d$. Since we can move from any word representing $g$ to any other via the elementary steps described in \eqref{eq:basics_steps}, it is enough to show that any elementary step does not change the value of $g.\Fa$. When applying an elementary change of type $(a)$ we have $\alpha_{j_i}^{l_i}=\alpha_{j_i}^{l'_i}$ and in particular $\om_{\Fb}(\alpha_{j_i}^{l_i})=\om_{\Fb}(\alpha_{j_i}^{l'_i})$ for every $\Fb\in X_\sm^{d-1}$. Since the permutations are the same, so is the action they induce.  As for an elementary change of type $(b)$ or $(c)$, replacing $\alpha_{j_{i+1}}^{l_{i+1}}\alpha_{j_i}^{l_i}$ with $j_i=j_{i+1}$ by $\alpha_{j_i}^{l_i+l_{i+1}}$ or vice versa yields the same action as we now explain. Let $\Fa\in X_\sm^d$, and let $\Fb\in \partial^m\Fa$ be the $(d-1)$-multicell of color $\ls d\rs \setminus \{j_i\}$ in $\Fa$'s boundary. On the one hand, the action of $\alpha_{j_i}^{l_i+l_{i+1}}$ gives the $d$-multicell $\alpha_{j_i}^{l_i+l_{i+1}}.\Fa = \om_{\Fb}(\alpha_{j_i}^{l_i+l_{i+1}}).\Fa$. On the other hand, since $\alpha_{j_i}^{l_i}.\Fa = \om_{\Fb}(\alpha_{j_i}^{l_i}).\Fa$, it follows that the unique $(d-1)$-multicell contained in $\alpha_{j_i}^{l_i}.\Fa$ of color $\ls d\rs\setminus \{j_i\}$ is $\Fb$. Consequently $\alpha_{j_i}^{l_{i+1}}.(\alpha_{j_i}^{l_i}.\Fa) = \om_{\Fb}(\alpha_{j_i}^{l_{i+1}}).( \om_{\Fb}(\alpha_{j_i}^{l_{i}}).\Fa)$, and hence the resulting $d$-multicell is $\alpha_{j_i}^{l_{i+1}+l_i}.\Fa=\om_\Fb(\alpha_{j_i}^{l_i+l_{i+1}}).\Fa$, using the fact that $\om_{\Fb}$ is a homomorphism from $C_k$ to $\CS_{\delta(\Fb)}$. Finally, an elementary change of type $(d)$ does not change the resulting cell as the action defined by $\alpha_i^0$ is trivial. 

\begin{claim}\label{clm:action_of_G_d_k_from_the_left_is_transitive}
	For every  $(\widetilde{X},\gamma,\om,\Fa_0)\in\scrC_{d,k}$, the action of $G_{d,k}$ on its $d$-multicells is transitive.
\end{claim}

\begin{proof}
	Let $(\widetilde{X},\gamma,\om,\Fa_0)\in\scrC_{d,k}$ with $\widetilde{X}=(X,\sm,\sg)$ and let $\Fa,\Fa'\in X_\sm^d$. Since $\widetilde{X}$ is $d$-lower path connected, there exists a sequence of $d$-multicells $\Fa=\Fa(0),\Fa(1),\ldots,\Fa(m)=\Fa'$ such that $\Fb_j:=\sg(\Fa(j),\si(j))=\sg(\Fa(j-1),\si(j))$ for every $1\leq j\leq m$. Define $i_j\in \ls d\rs$ to be the unique color such that $\gamma(\Fb_j) = \widehat{i_j}$ for $1\leq j\leq m$. Furthermore, let $l_j\in \ls k-1\rs$ be the unique number such that $\om_{\Fb_j}(\alpha_{i_j}^{l_j}).\Fa(j-1)=\Fa(j)$ (which must exist since $\om_{\Fb_j}$ is transitive). It now follows from the definition of the action that $\Fa' = \alpha_{i_{m-1}}^{l_{m-1}}\ldots\alpha_{i_2}^{l_2}\alpha_{i_1}^{l_1}.\Fa$, which shows that the action is transitive. 
\end{proof}

\begin{remark}
	Note that the action of $G_{d,k}$ is only defined on the $d$-multicells, and is not a simplicial action. In particular, there is no action on the $0$-cells and it does not preserve the neighboring relation of multicells.

\end{remark}


\section{The universal element of $\scrC_{d,k}$}\label{sec:the_universal_element}

We say that $(\widetilde{X},\gamma,\om,\Fa_0)\in \scrC_{d,k}$ satisfies the universal property (of $\scrC_{d,k}$) if for any $(\widetilde{Y},\widehat{\gamma},\widehat{\om},\widehat{\Fa}_0)\in \scrC_{d,k}$ there exists a unique morphism $\varphi:(\widetilde{X},\gamma,\om,\Fa_0)\to(\widetilde{Y},\widehat{\gamma},\widehat{\om},\widehat{\Fa}_0)$.

It follows from the definition of the universal object that, if it exists, it is necessarily unique up to a bijective morphism. Indeed, assume that $(\widetilde{X},\gamma,\om,\Fa_0)\in \scrC_{d,k}$ and $(\widetilde{Y},\widehat{\gamma},\widehat{\om},\widehat{\Fa}_0)\in\scrC_{d,k}$ satisfy the universal property. Then there exist unique morphisms $\varphi:(\widetilde{X},\gamma,\om,\Fa_0)\to (\widetilde{Y},\widehat{\gamma},\widehat{\om},\widehat{\Fa}_0)$ and $\psi:(\widetilde{Y},\widehat{\gamma},\widehat{\om},\widehat{\Fa}_0)\to (\widetilde{X},\gamma,\om,\Fa_0)$. Consequently $\psi\circ\varphi:(\widetilde{X},\gamma,\om,\Fa_0)\to (\widetilde{X},\gamma,\om,\Fa_0)$ and $\varphi\circ \psi:(\widetilde{Y},\widehat{\gamma},\widehat{\om},\widehat{\Fa}_0)\to (\widetilde{Y},\widehat{\gamma},\widehat{\om},\widehat{\Fa}_0)$ are morphisms from $(\widetilde{X},\gamma,\om,\Fa_0)$ and $(\widetilde{Y},\widehat{\gamma},\widehat{\om},\widehat{\Fa}_0)$ to themselves respectively. Since such morphisms are unique by the universal property, and since the identity maps from $\widetilde{X}$ and $\widetilde{Y}$ to themselves are morphisms as well, it follows that $\psi\circ \phi=\textrm{id}_{\widetilde{X}}$ and $\phi\circ\psi = \textrm{id}_{\widetilde{Y}}$. In particular, we obtain that $\psi$ and $\phi$ are morphisms which are also bijections between the $0$-cells of $\widetilde{X}$ and $\widetilde{Y}$, which proves that $(\widetilde{X},\gamma,\om,\Fa_0)$ and $(\widetilde{Y},\widehat{\gamma},\widehat{\om},\widehat{\Fa}_0)$ are isomorphic.


\subsection{Arboreal complexes}\label{subsec:arboreal_complexes}

Our next goal is to give an explicit construction for the universal element (which in particular proves its existence). We start by recalling the definition of arboreal complexes from \cite{PR12}. 

\begin{definition}[Arboreal complexes \cite{PR12}]\label{def:arboreal_complex}	
We say that a $d$-complex is \emph{arboreal }if it is obtained by the following procedure:  Start with a $d$-cell $\CT$, and attach to each of its $(d-1)$-cells new $d$-cells, using a new vertex for each of the new $d$-cells. Continue by induction, adding new $d$-cells to each of the $(d-1)$-cells which were added in the last step, using a new vertex for each of the new $d$-cells. As is the case for graphs (the case $d=1$), for every $d,k\geq 1$,  this defines a unique \emph{$k$-regular $d$-dimensional arboreal complex}, denoted $T_{d,k}$   (see Figure \ref{fig:universal_arboreal_complex} for an illustration). 
\end{definition}

\begin{remark}
	Note that $T_{1,k}$ is the $k$-regular tree. But, the case $d\geq 2$ brings a new phenomenon, although the degree of each $(d-1)$-cell in $T_{d,k}$ is $k$ by definition, the degree of each $j$-cell for $j<d-1$ is infinite. 
\end{remark}

\begin{figure}[h]
	\begin{center}
	\includegraphics[scale=1.6]{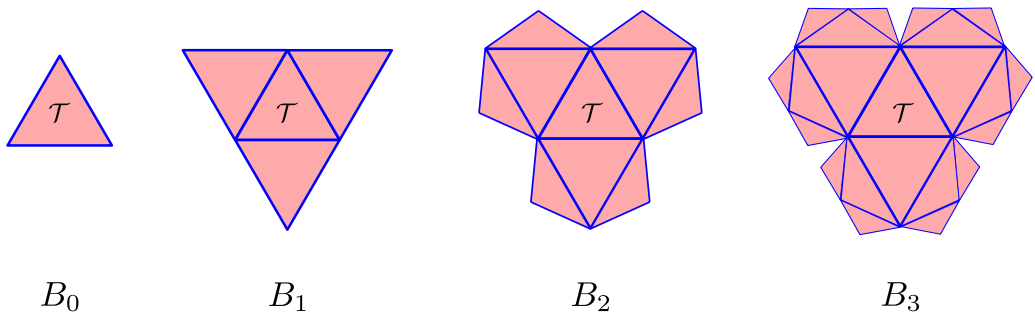}
	\caption{The first four	steps of the construction of the universal arboreal complex $T_{2,2}$. \label{fig:universal_arboreal_complex}}
	\end{center}
\end{figure}

The complexes $T_{d,k}$ were studied in \cite{PR12,Ro14}, and in particular the spectrum and spectral measure of their Laplacians were calculated. One can verify that the complex $T_{d,k}$ is transitive at all levels, that is, for any $0\leq j\leq d$ and any pair of $j$-cells $\rho,\rho'$, there is a simplicial automorphism of $T_{d,k}$ taking $\rho$ to $\rho'$. 

For $n\geq 0$, we denote by $B_n$ the ball of radius $n$ around $\CT$, that is, the subcomplex of $T_{d,k}$ containing all $d$-cells (and $j$-cells contained in them) which are attached to $T_{d,k}$ in the first $(n+1)$ steps of the construction (see Figure \ref{fig:universal_arboreal_complex}). 

The construction of $T_{d,k}$ allows us to introduce a coloring of its $0$-cells using $(d+1)$ colors (see Figure \ref{fig:universal_arboreal_complex_coloring}). Indeed, denote by $(\CT_i)_{i=0}^d$ the $0$-cells of $\CT$ and color them using a different color for each $0$-cell. Next, assume that the $d$-cells in $B_n$ were colored in such a way that the colors of the $0$-cells in each $d$-cell are distinct. If $\tau\in B_{n+1}\setminus B_n$ is a $d$-cell that contains the $(d-1)$-cell $\si\in B_n$, then by induction we know that the colors of the $0$-cells of $\si$ are distinct, and thus, there exists a unique color $j\in \ls d\rs$ which is not used to color the $0$-cells of $\si$. Define the color of $\tau\setminus \si$ to be $j$. Since the uncolored $0$-cells of the $d$-cells which are added in the $(n+1)$-step are distinct this is a well defined coloring, and, by construction, the colors of the $0$-cells of each of the $d$-cells in $B_{n+1}$ are distinct. Continuing by induction, we obtain a coloring of $T_{d,k}$ and in particular that $T_{d,k}$ is colorable. For future use, we fix a coloring of the $0$-cells $\Gamma:T_{d,k}^0\to \ls d\rs$, which, as explained in Subsection \ref{subsec:colorable_multicomplexes}, can also be used to color all the cells in $T_{d,k}$. Note that the coloring $\Gamma$ is completely determined by its values on the $0$-cells $(\CT_i)_{i=0}^d$. 

\begin{figure}[h]
	\begin{center}
	\includegraphics[scale=1.6]{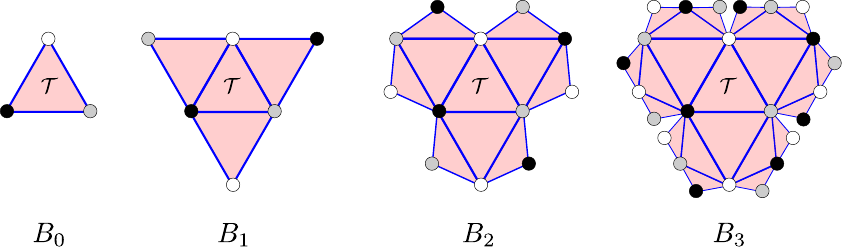}
	\caption{The first four steps of the construction of a coloring for the universal arboreal complex $T_{2,2}$. The colors black, white and gray are used instead of $0,1,2$.\label{fig:universal_arboreal_complex_coloring}}
	\end{center}
\end{figure}

	One may wish to characterize the complex $T_{d,k}$ as the unique $k$-regular $d$-complex in which for any pair of $d$-cells $\tau,\tau'$ in $T_{d,k}$ there exists a unique non-backtracking path from $\tau$ to $\tau'$. However, this does not determine $T_{d,k}$ uniquely, as any complex obtained from $T_{d,k}$ by identification of lower dimensional cells of the same color yields another complex with the above property. Nevertheless, using the results of this paper, we can prove the following uniqueness result:

\begin{proposition}\label{prop:uniqueness_of_T_d_k}	
	 $T_{d,k}$ is the unique link-connected, $k$-regular $d$-complex with the property that for any pair of $d$-cells $\tau,\tau'$ in $T_{d,k}$ there exists a unique non-backtracking path from $\tau$ to $\tau'$.
\end{proposition}

The proof of Proposition \ref{prop:uniqueness_of_T_d_k} is postponed to Section \ref{sec:from_simplicial_complexes_to_subgroups_and_back}. Let $\Om=(\Om_\si)_{\si\in (T_{d,k})^{d-1}}$ be a fixed choice of $k$-ordering for $T_{d,k}$. 

\begin{proposition}[Universal property of $T_{d,k}$]\label{prop:universal_property}
	The quartet $(T_{d,k},\Gamma,\Om,\CT)\in \scrC_{d,k}$ satisfies the universal property. 
\end{proposition}

\begin{proof}
	Let $(\widetilde{Y},\widehat{\gamma},\widehat{\om},\widehat{\Fa}_0)\in \scrC_{d,k}$ with $\widetilde{Y}=(Y,\sm,\sg)$. Using the ball structure of $T_{d,k}$ we prove by induction that the existence of a unique morphism $\widetilde{\varphi}:(T_{d,k},\Gamma,\Om,\CT)\to (\widetilde{Y},\widehat{\gamma},\widehat{\om},\widehat{\Fa}_0)$. More precisely, we prove the following statement by induction on $n$:
	\begin{quote}
		There exists a unique morphism $\widetilde{\varphi}:B_n\to \widetilde{Y}$ such that: $(i)$ $\widetilde{\varphi}$ is a simplicial multimap from $B_n$ to $\widetilde{Y}$, $(ii)$ $\widehat{\gamma} \circ\widetilde{\varphi} = \Gamma$ on $B_n$, $(iii)$ for every $\si\in B_n$ satisfying $\deg_{B_n}(\si)=k$ we have $\widehat{\om}_{\widetilde{\varphi}(\si)}(\beta) . \widetilde{\varphi}(\tau)= \widetilde{\varphi} (\Om_\si(\beta).\tau)$ for any $\tau\in \delta_{T_{d,k}}(\si)$ and $\beta\in C_k$, and $(iv)$ $\widetilde{\varphi}(\CT)=\widehat{\Fa}_0$. 
	\end{quote}
	
	For $n=0$, recall that $B_0$ is the complex composed of a unique $d$-cell $\CT$. Since condition $(iv)$ forces us to have $\widetilde{\varphi}(\CT)=\Fa_0$ and since by condition $(ii)$ the colors must be preserved, there exists a unique simplicial multimap satisfying the induction assumption, namely, the map sending the unique cell of color $J$ contained in $\CT$ to the unique multicell of color $J$ contained in $\widehat{\Fa}_0$.	
	
	Next, assume the induction assumption holds for $n$. We show that $\widetilde{\varphi}$ can be extended to $B_{n+1}$ in a unique way so that $(i)-(iv)$ are satisfied. Indeed, given any $d$-cell $\tau\in B_{n+1}\setminus B_n$, there exists a unique $(d-1)$-cell $\si\in B_n$ such that $\si \subset \tau$. Furthermore, there exists a unique $\tau'\in B_n^d$ containing $\si$. Since $\widetilde{\varphi}$ was already defined on $B_n$, and in particular on $\si$ and $\tau'$, we can use them together with the ordering $\widehat{\om}$ in order to define the map $\widetilde{\varphi}$ for all $d$-cells containing $\si$. Indeed, using the transitivity of the ordering $\Om_\si$ one can find (a unique) $\beta\in C_k$ such that $\Om_\si(\beta).\tau'=\tau$. We then define $\widetilde{\varphi}(\tau)=\widetilde{\varphi}(\Om_\si(\beta).\tau) :=\widehat{\om}_{\widetilde{\varphi}(\si)}(\beta).\widetilde{\varphi}(\tau')$. Finally, we extend the definition of $\widetilde{\varphi}$ to the new lower-dimensional cells in $B_{n+1}$ in the unique possible way which preserves the coloring, i.e., for $\tau\in B_{n+1}\setminus B_n$ and $\si\subset \tau$, define $\widetilde{\varphi}(\si)$ to be the unique multicell contained in $\widetilde{\varphi}(\tau)$ such that $\widehat{\gamma}(\iota\widetilde{\varphi}(\si))=\Gamma(\si)$. The function $\varphi$, which is defined now on $B_{n+1}$, satisfies $(i)-(iv)$ by definition and is the unique such function. This completes the induction step and hence the proof.  
\end{proof}

\begin{remark}
	One can think of $T_{d,k}$ as a hyperbolic building whose Weyl group is the free product of $(d+1)$ copies of the cyclic group of order $2$. We thank Shai Evra and Amitay Kamber for this observation. 
\end{remark}


\subsection{Links of $T_{d,k}$}

\begin{proposition}\label{prop:links_in_T^d_k}
	For every $-1\leq j\leq d-2$ and $\rho\in T_{d,k}^j$, the link $\lk_{T_{d,k}}(\rho)$ is isomorphic to $T_{d-|\rho|,k}$. 
\end{proposition}

\begin{proof}
	Since $\lk_X(\emptyset)=X$ for any complex, we have $\lk_{T_{d,k}}(\emptyset)=T_{d,k}$ which completes the proof for $j=-1$. Furthermore, since for any $\rho,\rho'\in T_{d,k}$ such that $\rho\cap \rho' =\emptyset$ and $\rho\cup\rho'\in T_{d,k}$ it holds that $\lk_{\lk_{T_{d,k}}(\rho)}(\rho') = \lk_{T_{d,k}}(\rho\cup\rho')$, it is sufficient to prove the lemma for $\rho\in T_{d,k}^0$ as the remaining cases follow by induction. 
	
Assume next that $v\in T_{d,k}^0$ and fix a $(d-1)$-dimensional cell $\widehat{\si}\in \lk_{T_{d,k}}(v)$ (note that such a cell exists since $T_{d,k}$ is pure). For $n\in\BN$, denote by $B_n(v)=B_n(v,\widehat{\si})$ the ball of radius $n$ around $\widehat{\si}$ in $\lk_{T_{d,k}}(v)$, that is, the set of $(d-1)$-cells in $\lk_{T_{d,k}}(v)$ whose distance  from $\widehat{\si}$ in the associated line graph is at most $n$, together with the cells contained in them. 
	
	Recalling the inductive method for constructing $T_{d,k}$ and observing that $B_0(v)=\{\widehat{\si}\}$, it is enough to show that for every $n\in\BN$ the ball $B_{n+1}(v)$ is obtained from $B_n(v)$ by attaching to each of the $(d-2)$-cells of degree $1$ additional $(k-1)$ new $(d-1)$-cells, each using a new $0$-cell. We prove this by induction. For $n=0$, since $\widehat{\si}$ corresponds to the $d$-cell $\widehat{\si}\cup v\in T_{d,k}$, and since $T_{d,k}$ as a transitive structure on the $d$-cells, we can assume without loss of generality that $\widehat{\si}\cup v = \CT$. Hence, from the definition of $T_{d,k}$, each of the $(d-1)$-cells in the boundary of $\widehat{\si}\cup v$ is attached to $(k-1)$ additional $d$-cells, each of which uses a new $0$-cell.

	Next, assume that the assumption holds for $B_n(v)$ and observe $B_{n+1}(v)$. 
Each of the $(d-2)$-cells in $B_{n}(v)$ of degree $1$ corresponds to a $(d-1)$-cell containing $v$ in $B_n$ of degree $1$, where without loss of generality we assume that $v\cup \widehat{\si}$ is the $d$-cell around which $B_n$ is constructed, i.e. $\CT=v\cup \widehat{\si}$. Furthermore, given a $(d-2)$-cell $\si\in B_n(v)$ such that $\deg_{B_n(v)}(\si)=1$ and a corresponding $(d-1)$-cell $\tau= v\cup\si\in B_n$ we have $\deg_{B_n}(\tau)=1$, and exactly $(d-1)$ of the $(d-2)$-cells in the boundary of $\tau$ contain $v$. Using the inductive definition of the balls $B_n$, the ball $B_{n+1}$ is obtained from $B_n$ by attaching to each of the $(d-1)$-cells of degree $1$ in $B_n$ an additional $(k-1)$ distinct $d$-cells, each of which using a new $0$-cell. Hence in $\lk_{T_{d,k}}(v)$, the resulting ball $B_{n+1}(v)$ is obtained from $B_n$ by attaching to each of the $(d-2)$-cells of degree $1$ an additional $(k-1)$ distinct $(d-1)$-cells, each of which uses a new $0$-cell. 
\end{proof}

\begin{remark}
	Given a	representative of the universal element $(T_{d,k},\Gamma,\Om,\CT)$ and $\si\in T_{d,k}$, the link $\mathrm{lk}_\si(T_{d,k})$ arrives with a natural $k$-ordering and a coloring (with $(d+1-|\si|)$ colors) inherited from $(T_{d,k},\Gamma,\Om,\CT)$. 
\end{remark}


\subsection{The action of $G_{d,k}$ on $d$-cells of $(T_{d,k},\Gamma,\Om,\CT)$}

As $(T_{d,k},\Gamma,\Om,\CT)$ is an object of $\scrC_{d,k}$, by Section \ref{sec:group_action} the group $G_{d,k}$ acts on it from the left transitively. Recall the line graph $\scrG(T_{d,k})$ of $T_{d,k}$ and its graph distance $\mathrm{dist}:\scrV(T_{d,k})\times \scrV(T_{d,k})\to \BN\cup\{0\}$ (see Section \ref{sec:preliminaries}).

\begin{lemma}\label{lem:the_action_of_G_on_T_is_free}$~$
	\begin{enumerate}[(1)]
		\item For every $g\in G_{d,k}$, the unique reduced word representing $g$ is of length $m$ if and only if $\dist(g.\CT,\CT)=m$. 
		\item For every representative $(T_{d,k},\Gamma,\Om,\CT)$ of the universal element, the action of $G_{d,k}$ on the $d$-cells of $T_{d,k}$ is simply transitive.
	\end{enumerate}	
\end{lemma}

\begin{proof} $~$
\begin{enumerate}[(1)]
	\item It follows from the inductive construction of $T_{d,k}$ that $\dist(\tau,\CT)=m$ if and only if $\tau\in B_m\setminus B_{m-1}$. The result will follow once we prove the following two claims:
	\begin{enumerate}[(a)]
		\item For every $\tau\in T_{d,k}^d$ such that $\dist(\tau,\CT)=m$, there exists a unique path (called the good path) $\CT=\tau(0),\tau(1),\ldots,\tau(m)=\tau$ of $d$-cells, with $\dist(\tau(j),\CT)=j$ for $0\leq j\leq m$.
		\item Let $g\in G_{d,k}$ and assume $w=\alpha_{j_m}^{l_m} \ldots \alpha_{j_1}^{l_1}$ is  the reduced word representing $g$. Then, the path $\tau(0),\ldots,\tau(m)$, where $\tau(i)=\alpha_{j_i}^{l_i} \ldots \alpha_{j_1}^{l_1}.\CT$ is a good path. 
	\end{enumerate}
	
	We prove $(a)$ by induction on $m$. For $m=0$ the claim is immediate. Assume the statement holds for all $\tau\in T_{d,k}^d$ whose distance from $\CT$ is strictly smaller than $m$ and let $\tau\in T_{d,k}^d$ be a $d$-cell such that $\dist(\tau,\CT)=m$. Due to the  inductive construction of $T_{d,k}$, we know that $\tau\in B_m\setminus B_{m-1}$ and that $\tau$ is a neighbor of a unique $d$-cell $\tau'$ in $B_{m-1}$. In particular, any path from $\CT$ to $\tau$ must visit $\tau'$. Due to the induction assumption, there exists a unique path $\CT=\tau(0),\tau(1),\ldots,\tau(m-1)=\tau'$ such that $\dist(\CT,\tau(i))=i$ for $0\leq i\leq m-1$. 
	Consequently, the path $\CT=\tau(0),\tau(1),\ldots,\tau(m-1),\tau(m):=\tau$ is the unique path from $\CT$ to $\tau$ satisfying the required properties. 
		
	Turning to prove $(b)$, let $g\in G_{d,k}$ and assume that the path $\tau(0),\ldots,\tau(m)$ induced by the reduced word representing $g$ is not a good path, i.e. $\dist(\CT,\tau(i))\neq i$ for some $1\leq i\leq m$. Denoting by $i_0\geq 2$ the minimal such $i$ for which $\dist(\CT,\tau(i))\neq i$ (it is impossible for $i_0$ to be equal to $0$ or $1$ due to the structure of the ball $B_1$), it follows that $\dist(\CT,\tau(i_0-1))=i_0-1$, but $\dist(\CT,\tau(i_0))\neq i_0$. Furthermore, since $\tau(i_0-1)$ and $\tau(i_0)$ are neighbors, it follows that $\dist(\CT,\tau(i_0))\in \{i_0-2,i_0-1\}$. 
	
	In the first case, namely $\dist(\CT,\tau(i_0))=i_0-2$, it follows from the structure of $T_{d,k}$ that $\tau_{i_0}=\tau_{i_0-2}$. Hence, due to the definition of the action of $G_{d,k}$, it holds that $j_{i_0-1}=j_{i_0}$, which contradicts the assumption that the word is reduced. Similarly, if $\dist(\CT,\tau(i_0))=i_0-1$, we must have that $\tau_{i_0-2},\tau_{i_0-1}$ and $\tau_{i_0}$ have a common $(d-1)$-cell, which by the definition of the action, also implies that the word is not reduced. 
	
	\item The transitivity of the action of $G_{d,k}$ on the $d$-cells of $T_{d,k}$ follows from Lemma \ref{clm:action_of_G_d_k_from_the_left_is_transitive}. Thus, it remains to show that the action is simple. Assuming otherwise, one can find two distinct group elements $g_1,g_2\in G_{d,k}$ such that $g_1.\CT=g_2.\CT$. However, each of the reduced words representing $g_1$ and $g_2$ respectively induce a good path on $T_{d,k}$ which starts in $\CT$ and ends in $g_1.\CT=g_2.\CT$. Since such a path is unique, we must conclude that the paths coincide and as a result that $g_1=g_2$.
\end{enumerate}
\end{proof}

\begin{corollary}[The Cayley graph of $G_{d,k}$ and the line graph of $T_{d,k}$]\label{cor:bijection_of_G_and_T}
	Let $S=\{\alpha_i^l ~:~ i\in \ls d\rs,\,l\in [k-1]\}$. Given a representative of the universal object $(T_{d,k},\Gamma,\Om,\CT)$, there is a natural graph bijection between the line graph $\scrG(T_{d,k})$ and the left Cayley graph $\textrm{Cay}(G_{d,k};S)$, given by $g\mapsto g.\CT$. 
\end{corollary}

\begin{proof}
	Since $G_{d,k}$ acts simply transitive on the $d$-cells of $T_{d,k}$, which are exactly the vertices of the line graph $\scrG(T_{d,k})$, it follows that $\scrG(T_{d,k})$ is isomorphic to the Cayley graph of $G_{d,k}$ with respect to the generator set $S'=\{g\in G_{d,k} ~:~ g.\CT \text{ is a neighbor of }\CT\text{ and }g.\CT\neq \CT\}$. Hence, it remains to show that $S'=S$. This follows from Lemma \ref{lem:the_action_of_G_on_T_is_free}(1), as the only elements $g\in G_{d,k}$ such that $g.\CT$ and $\CT$ are distinct neighbors, that is $\dist(g.\CT,g)=1$, are represented by reduced words of length $1$ which are exactly the words $\alpha_i^l$ for $i\in \ls d\rs$ and $l\in  [k-1]$. 
\end{proof}


\subsection{The action of $G_{d,k}$ on general cells of $(T_{d,k},\Gamma,\Om,\CT)$}\label{subsec:general_left_action_in_T_d_k}

Although the group $G_{d,k}$ does not act directly from the left on lower-dimensional cells of $T_{d,k}$, one can use the action of $G_{d,k}$ on $d$-cells of $T_{d,k}$ in order to define an action of $G_{d,k}$ on pairs of cells of the form 
$(\rho,\tau)\in T_{d,k}\times T_{d,k}^d$, where $\rho\subseteq \tau$. This is done by setting $g.(\rho,\tau)$ to be the pair $(\rho',\tau')$, where $\tau'=g.\tau$, and $\rho'$ is the unique $j$-cell in $g.\tau$ whose color is the same as the color of $\rho$. 

Despite of the fact that the action of $G_{d,k}$ on $d$-cells of $T_{d,k}$ is simply transitive (see Lemma \ref{lem:the_action_of_G_on_T_is_free}), it is possible for certain group elements to stabilize a $j$-cell $\rho$ in a pair $(\rho,\tau)$, while changing the $d$-cell itself (this means in particular that both $\tau$ and $g.\tau$ contain $\rho$). We thus wish to study the subgroup
\begin{equation}\label{eq:stabilizer_of_a_lower_cell}
	L_{\rho,\tau}:=\big\{g\in G_{d,k} ~:~ g.(\rho,\tau)=(\rho,g.\tau)\big\}
\end{equation}
for $\rho\in T_{d,k}$ and $\rho\subseteq \tau \in T_{d,k}^d$.

\begin{definition}\label{def:subgroup_K}
	For $J\subseteq \ls d\rs$, define the subgroup 
	\begin{equation}
	K_J = \langle\alpha_j ~:~ j\in J\rangle \leq G_{d,k}.
	\end{equation}
	We occasionally use the notation $\widehat{J}$ to denote $\ls d\rs\setminus J$. Furthermore, for $i\in \ls d\rs$, we use the abbreviation $\widehat{i}$ for $\widehat{\{i\}}$. 
\end{definition}

\begin{lemma}\label{lem:stabilized_lemma}
	Let $(T_{d,k},\Gamma,\Om,\CT)$ be a representative of the universal object. For every $\tau\in T_{d,k}^d$ and $\rho \subseteq \tau$  
	\begin{equation}
		L_{\rho,\tau} = K_{\ls d\rs\setminus \Gamma(\rho)}.
	\end{equation} 
	In particular, $L_{\rho,\tau}$ depends only on the color of $\rho$ and not on the choice of the $d$-cell $\tau$. 
\end{lemma}

\begin{proof}
	The case $\rho=\tau$ follows from Lemma \ref{lem:the_action_of_G_on_T_is_free}. Hence we assume that $\rho\subsetneq \tau$. We first show that $K_{\ls d\rs\setminus \Gamma(\rho)}\leq L_{\rho,\tau}$. Since $K_{\ls d\rs\setminus \Gamma(\rho)}$ is generated by $(\alpha_j)_{j\in \ls d\rs \setminus \Gamma(\rho)}$ and $L_{\rho,\tau}$ is a group, it suffices to show that $(\alpha_j)_{j\in \ls d\rs \setminus \Gamma(\rho)}\subseteq L_{\rho,\tau}$. Fix some $j\in \ls d\rs \setminus \Gamma(\rho)$. Given $\tau\in T^d_{d,k}$ such that $\rho\subseteq \tau$, let $\si$ be the unique $(d-1)$-cell contained in $\tau$ of color $\ls d\rs \setminus j$. By the assumption on $j$, we know that $\Gamma(\rho) \subseteq \ls d \rs \setminus j$ and hence that $\rho\subseteq \si$. Using the definition of the action on $d$-cells, $\alpha_j.\tau = \om_{\si}(\alpha_j).\tau$ must contain $\si$ and thus also $\rho$. This implies that $\rho$ is also stabilized by the action of $\alpha_j$, i.e. $\alpha_j.(\rho,\tau)=(\rho,\alpha_j.\tau)$, as required. 
	
	Next, we show that $L_{\rho,\tau} \leq K_{\ls d\rs \setminus \Gamma(\rho)}$. Assume $g\in L_{\rho,\tau}$, then $g.(\rho,\tau)=(\rho,g.\tau)$ and thus in particular $\rho\subseteq \sigma:=g.\tau \cap \tau$. Consider the link of $\si$. Since $\tau\setminus \si, (g.\tau)\setminus \si \in \lk_X(\si)$ and since $\lk_X(\si)$ is isomorphic to the universal arboreal complex $T_{d-|\si|,k}$ (see Proposition \ref{prop:links_in_T^d_k}), it follows that there is a path $\gamma$ of $(d-|\si|)$-cells, connecting $\tau\setminus \si$ and $(g.\tau)\setminus \si$ in $\lk_X(\si)$. This path can be pulled back to $T_{d,k}$, thus creating a path of $d$-cells $\tau=\tau(0),\tau(1),\ldots,\tau(m)=g.\tau$ such that $\si\subset \tau(i)$ for every $0\leq i\leq m$. Defining $j_i\in\ls d\rs$ and $l_i\in \ls k-1\rs$ for $1\leq i\leq m$ to be the unique integers such that $\tau(i) = \alpha_{j_i}^{l_i}.\tau(i-1)$ (which must exist since $\tau(i-1)\cap\tau(i)$ is a $(d-1)$-cell) we conclude that $\alpha_{j_m}^{l_m}\ldots \alpha_{j_1}^{l_1}.\tau=g.\tau$. Recalling that $G_{d,k}$ acts freely on $T_{d,k}^d$, this implies $g=\alpha_{j_m}^{l_m}\ldots \alpha_{j_1}^{l_1}$. Finally, since for every $1\leq i\leq m$ we have $\si\subset \tau(i-1)\cap \tau(i)$ it follows that $j_i\in \ls d\rs \setminus \Gamma(\si)\subseteq \ls d\rs \setminus \Gamma(\rho)$ for every $1\leq i\leq m$. Thus $g=\alpha_{j_m}^{l_m}\ldots \alpha_{j_1}^{l_1}\in K_{\ls d\rs \setminus \Gamma(\si)}\leq K_{\ls d\rs \setminus \Gamma(\rho)}$ as required. 
\end{proof}

The last lemma allows us to generalize Corollary \ref{cor:bijection_of_G_and_T} and define a bijection between general cells of $T_{d,k}$ and certain cosets of the group $G_{d,k}$.

\begin{definition}[Cosets in $G_{d,k}$]\label{def:cosets_of_G_d_k} For $-1\leq j\leq d$ let 	
	\begin{equation}\label{eq:thm_j_cells_in_Tdk_as_costes}
		\CM^j:=\big\{K_{\widehat{J}} g ~:~ g\in G_{d,k},\, J\subseteq \ls d\rs \text{ such that }|J|=j+1\big\},
	\end{equation}
	where $K_J$ is defined as in Definition \ref{def:subgroup_K}. Also, set $\CM:=\bigcup_{j=-1}^d \CM^j$.
\end{definition}

\begin{corollary}
\label{cor:bijection_of_cells_in_T_with_cosets}
	Given a representative of the universal object $(T_{d,k},\Gamma,\Om,\CT)$, for every $-1\leq j\leq d$, there is a natural bijection $\Psi_j:\CM^j \to T_{d,k}^j$, given by 
	\begin{equation}\label{eq:defn_of_Psi}
		\Psi_j(K_{\widehat{J}}g) =
		\begin{array}{c}
		\text{the unique cell in}\\
		g.\CT\text{ whose color is } J.
		\end{array} 
	\end{equation}
	For future use, we denote by $\Psi:\CM \to T_{d,k}$ the map, whose restriction to $\CM^j$ is $\Psi_j$ for every $-1\leq j\leq d$. 
\end{corollary}

\begin{proof}
	We start by proving that $\Psi$ is well defined. Assume that for some $J,J'\subseteq \ls d\rs$ such that $|J|=|J'|$ and some $g,g'\in G_{d,k}$ we have $K_{\widehat{J}}g = K_{\widehat{J'}}g'$. Then $g\in  K_{\widehat{J'}}g' ~~\Rightarrow ~~ K_{\widehat{J'}}g\subseteq K_{\widehat{J'}}g'=K_{\widehat{J}}g ~~ \Rightarrow ~~ K_{\widehat{J'}}\subseteq K_{\widehat{J}}$. Applying the same argument in the opposite direction we conclude that $K_{\widehat{J}}=K_{\widehat{J'}}$ and hence that $J=J'$. Next, observe that $K_{\widehat{J}}g=K_{\widehat{J'}}g'=K_{\widehat{J}}g'$ and therefore that $gg'^{-1}\in K_{\widehat{J}}$. Recalling the definition of $\Psi_j$ we obtain that $\Psi_j(K_{\widehat{J}} g)$ is the unique $j$-cell $\rho$ in $g.\CT$ such that $\Gamma(\rho)=J$, and similarly that $\Psi_j(K_{\widehat{J'}}g')=\Psi_j(K_{\widehat{J}} g')$ is the unique $j$-cell $\rho'$ in $g'.\CT$ such that $\Gamma(\rho')=J$. However, since $gg'^{-1}\in K_{\widehat{J}}$, it follows from Lemma \ref{lem:stabilized_lemma} that the unique $j$-cell in $g'.\CT$ whose color is $J$ is the same as the unique $j$-cell in $(gg'^{-1}).(g'.\CT)=g.\CT$ whose color is $J$, that is $\rho=\rho'$. This completes the proof that $\Psi_j$ is well defined. 
	
	Next, we show that the map is onto. Let $\rho\in T_{d,k}^j$, and assume that $\Gamma(\rho)=J$. Since $T_{d,k}$ is pure, one can find a $d$-cell $\tau$ containing $\rho$.  By Lemma \ref{lem:the_action_of_G_on_T_is_free}, any $d$-cell of $T_{d,k}$ (and in particular $\tau$) can be written in the form $g.\CT$ for some $g\in G_{d,k}$. Hence, $\Psi_j(K_{\widehat{J}}g)=\rho$. 

	Finally, we show that the map is injective. Assume that $\rho:=\Psi_j(K_{\widehat{J}}g)=\Psi_j(K_{\widehat{J'}}g')$. It follows from the definition of the map $\Psi_j$ that $J=J'$, since otherwise the colors of the $j$-cells $\Psi_j(K_{\widehat{J}}g)$ and $\Psi_j(K_{\widehat{J'}}g')$ are not the same and in particular the cells are distinct. Furthermore, the $d$-cells $g.\CT$ and $g'.\CT$ have $\rho$ as a common $j$-cell such that $\Gamma(\rho)=J$. This implies that $gg'^{-1}\in L_{\rho,\tau}$ and therefore by Lemma \ref{lem:stabilized_lemma} that $gg'^{-1}\in K_{\ls d\rs \setminus \Gamma(\rho)}=K_{\widehat{J}}$. Consequently, $K_{\widehat{J'}}g'=K_{\widehat{J}} g'=K_{\widehat{J}} (gg'^{-1})g'=K_{\widehat{J}} g$, which proves that the map is injective. 
\end{proof}


\subsection{The right action of $G_{d,k}$ on $T_{d,k}$, and $G_{d,k}$ invariant coloring and ordering}

Let $(T_{d,k},\Gamma,\Om,\CT)$ be a representative of the universal object. Using Corollary \ref{cor:bijection_of_G_and_T} and Corollary \ref{cor:bijection_of_cells_in_T_with_cosets}, one can define the \emph{right} action of $G_{d,k}$ on the $d$-cells of $T_{d,k}$ as the right action of $G_{d,k}$ on the corresponding left Cayley graph. Formally, given a representative of the universal object $(T_{d,k},\Gamma,\Om,\CT)$, denote by $\Psi_d : G_{d,k} \to T_{d,k}^d=\scrV(T_{d,k})$ the graph bijection introduced in Corollary \ref{cor:bijection_of_G_and_T}, given by $\Psi_d(g)=g.\CT$ for every $g\in G_{d,k}$. Then, for $g\in G_{d,k}$ and $\tau\in T_{d,k}^d$ define $\tau.g$ to be $\Psi_d(\Psi_d^{-1}(\tau)g^{-1})$. 

Since the action is the right action on a left Cayley graph, it preserves the graph structure and in particular $\dist(\tau,\tau')=\dist(\tau.g,\tau'.g)$ for every $\tau,\tau'\in T_{d,k}^d$ and $g\in G_{d,k}$, where $\dist$ is the graph distance in the line graph $\scrG(T_{d,k})$. Furthermore, if we color the edges of $\textrm{Cay}(G_{d,k};S)$ by $S$, using the color $s$ for edges of the form $\{g,sg\}$, then the right action also preserves the edge color.

Unlike the left action of $G_{d,k}$ on $T_{d,k}^d$, which cannot be extended directly to the lower dimensional cells (see Subsection \ref{subsec:general_left_action_in_T_d_k}), the right action can be extended to act on all cells of $T_{d,k}$. Indeed, given any cell $\rho\in T_{d,k}$ define $\rho.g$ as follows: choose any $d$-cell $\tau$ containing $\rho$, and set $\rho.g$ to be the unique cell in $\tau.g$  whose color is  $\Gamma(\rho)$. In order to see that this is well defined, note that $(1)$ for every $d$-cell $\tau'$ in $T_{d,k}$, there is a bijection between $\{\rho ~:~ \rho\subset \tau'\}$ and subsets of $\ls d\rs$, and $(2)$ if $\tau$ and $\tau'$ are two 2 different $d$-cells containing $\rho$, i.e. $\rho \subset \tau\cap \tau'$ then $\tau.g \cap \tau'.g$ contains a subset $\rho'$ such that $\Gamma(\rho')=\Gamma(\rho)$ (since the right action preserves the structure of the edge-colored Cayley graph). Consequently, $\rho.g$ does not depend on the choice of the $d$-cell containing it. Note that the right action preserves the coloring of the cells in $T_{d,k}$ by definition, and that for every choice of $J\subseteq \ls d\rs$, it is transitive on cells of color $J$. Another way to see this is to use Corollary \ref{cor:bijection_of_cells_in_T_with_cosets}: the right action of $G_{d,k}$ on the $j$-dimensional cells of color $J$ is equivalent to the right action of $G_{d,k}$ on $M_J = G_{d,k}/K_J$. 

Due to the existence of a coloring preserving right action of $G_{d,k}$ on all the cells of $T_{d,k}$, it is natural to seek orderings which are preserved under the right group action. Formally, we call an ordering $\Om$ of $T_{d,k}$ a $G_{d,k}$-invariant ordering if for every $\si\in T_{d,k}^{d-1}$, $\tau\in \delta(\si)$ and $g\in G_{d,k}$,
\begin{equation}\label{eq:invariant_ordering}
\begin{array}{l}
	(\Om_\si(\beta).\tau).g = \Om_{\si.g}(\beta).(\tau.g),\qquad \forall \beta\in C_k. \\
\end{array}
\end{equation}
The last equation gives a simple way to define an invariant ordering for $T_{d,k}$ by defining the ordering on the $(d-1)$-cells of $\CT$ and using \eqref{eq:invariant_ordering} to define the ordering on the remaining $(d-1)$-cells of $T_{d,k}$. Formally, for $0\leq i\leq d$, denote by $\Si_i$ the unique $(d-1)$-cell of $\CT$ whose color is $\widehat{i}$ and define the ordering:
\begin{equation}\label{eq:invariant_ordering_defn}
	\begin{array}{l}
		\Omega_{\Si_i}(\alpha_i^l).\CT := \CT.\alpha_i^l,\qquad\qquad\qquad \forall i\in \ls d\rs,\, l\in\ls k-1\rs,\\
		~\\
		\Om_{\si.g}(\beta).(\tau.g):=(\Om_\si(\beta).\tau).g,\qquad \forall \beta\in C_k,\, g\in G_{d,k},\, \si\in T_{d,k}^{d-1},\, \si\subset \tau\in T_{d,k}^d.
	\end{array}
\end{equation}
This defines the ordering completely since the right action of $G_{d,k}$ is transitive on cells of the same dimension and color. One can verify that if $\tau\in T_{d,k}^d$ and $\si\subset \tau$ is the unique $(d-1)$-cell of color $\widehat{i}$ in it, then there exists a unique $g\in G_{d,k}$ such that $\tau.g = \CT$ and $\si.g=\Si_i$. Hence,
\begin{equation}
	\Om_\si(\alpha_i^l).\tau = \Om_{\Si_i.g}(\alpha_i^l).(\tau.g) = (\Omega_{\Si_i}(\alpha_i^l).\CT).g = \CT.\alpha_i^l g.
\end{equation}
Throughout the remainder of this paper we assume that any ordering $\Om$ of $T_{d,k}$ is $G_{d,k}$-invariant, i.e. it satisfies \eqref{eq:invariant_ordering_defn}.


\section{Group interpretation of $T_{d,k}$}\label{sec:group_interpretation_of_T_d_k}

In this section we describe two additional constructions of the universal object $T_{d,k}$. The first uses the group $G_{d,k}$ and the results proved in the previous section on its action on $T_{d,k}$. This hints at the construction for general subgroups of $G_{d,k}$ discussed in the next section. Using the first construction, one can introduce a second one which is based on the notion of a nerve complex (see Definition \ref{def:nerve_complex}). Since a significant portion of the proofs for both constructions is a special case (for the subgroup $H=\langle e\rangle \leq G_{d,k}$) of the more general theory (presented in the Sections \ref{sec:from_subgroup_to_elements_of_the_category}-\ref{sec:further_relations}), the proofs are postponed.

Recall the definition of $\CM$ (see Definition \ref{def:cosets_of_G_d_k}) and for $-1\leq j\leq d$, let $\Phi : \CM \to \bigcup_{j=-1}^d \binom{\CM^0}{j+1}$ be the map
	\begin{equation}\label{eq:defn_of_Phi}
		\Phi(K_{\widehat{J}}g) = \{K_{\widehat{i}}g ~:~ i\in J\},\qquad \forall J\subseteq \ls d\rs,\, g\in G_{d,k}.
	\end{equation}		

One can verify (see Lemma \ref{lem:well_defind_map_for_double_cosets}) that the map $\Phi$ is well defined and that $\Phi(K_{\widehat{J}}g)\in \binom{M^0}{j+1}$ whenever $J\subseteq \ls d\rs$ is of size $j+1$. 

\begin{definition}[The simplicial complex $\FX_{d,k}$]
	Using the set $\CM$ and the map $\Phi$, we define the simplicial complex $\FX_{d,k}$ with vertex set $\CM^0$ as follows: $\FX_{d,k}^{-1} = \{\emptyset\}$, $\FX_{d,k}^0 = \CM^0$, and for $1\leq j\leq d$
	\begin{equation}
		\FX_{d,k}^j = \{\Phi(K_{\widehat{J}} g) ~:~ K_{\widehat{J}}g\in \CM^j\}.
	\end{equation}
\end{definition}

It is not difficult to check that $\FX_{d,k}$ is indeed a $d$-dimensional simplicial complex with vertex set $\CM^0$. In addition, there is a bijection between $j$-cells of $\FX_{d,k}$ and elements in $\CM^j$, that is, the map $\Phi|_{\CM^j}:\CM^j\to \FX_{d,k}^j$ is a bijection. Combining this with Corollary \ref{cor:bijection_of_cells_in_T_with_cosets} we obtain:

\begin{theorem}[The complex $\FX_{d,k}$ is the universal object]\label{thm:j-cells_in_Tdk_as_cosets} The $d$-complex $\FX_{d,k}$ is isomorphic to $T_{d,k}$. Furthermore, for any choice of a representative of the universal object  $(T_{d,k},\Gamma,\Om,\CT)$ one can endow $\FX_{d,k}$ with a natural coloring, $k$-ordering and a root such that the map $\Psi\circ \Phi^{-1}:\FX_{d,k}\to T_{d,k}$
is a bijective morphism (in the sense of the category $\scrC_{d,k}$), where $\Psi$ is as defined in Corollary \ref{cor:bijection_of_cells_in_T_with_cosets}, and $\Phi$ is as defined in equation \eqref{eq:defn_of_Phi}.
\end{theorem}

Using the last construction of $T_{d,k}$ via $G_{d,k}$, we can also describe $T_{d,k}$ as a nerve complex.

\begin{theorem}[$T_{d,k}$ as a nerve complex]\label{thm:nerve_version_of_T_d_k}
	For $v=K_{\widehat{i}}\, g\in \CM^0$ define $\CA_v = \{g'\in G_{d,k} ~:~ K_{\widehat{i}}\, g=K_{\widehat{i}}g'\}$. Then, the nerve complex $\CN((\CA_v)_{v\in \CM^0})$	is a $d$-dimensional simplicial complex which is isomorphic to $T_{d,k}$. 	
\end{theorem}


\section{From subgroups of $G_{d,k}$ to elements of $\scrC_{d,k}$}\label{sec:from_subgroup_to_elements_of_the_category}

The goal of this section is to introduce a method for constructing elements of $\scrC_{d,k}$ using subgroups of $G_{d,k}$. Let $H\leq G_{d,k}$ be a subgroup of $G_{d,k}$. Since $G_{d,k}$ is in bijection with the $d$-cells of $T_{d,k}$, and  $G_{d,k}$ acts on $T_{d,k}$ from the right by simplicial automorphisms which preserve coloring, ordering and the gluing of cells, one can define the quotient complex $T_{d,k}\qu H$ associated with $H$ as follows: Define the $0$-cells of $T_{d,k}\qu H$ to be the orbits of the action of $H$ (from the right) on the $0$-cells of $T_{d,k}$. Next, declare a set of $0$-cells in $T_{d,k}\qu H$ to form a cell in the quotient complex, if there exist representatives of the $0$-cells in $T_{d,k}^0$ in the corresponding orbits, which form a cell in $T_{d,k}$. The cells in $T_{d,k}\qu H$ arrive with a natural multiplicity and gluing, namely, the multiplicity of a cell $\si$ in $T_{d,k}\qu H$ is the number of choices of representatives for $0$-cells along the orbits up to an $H$ equivalence, and the gluing is induced from the inclusion relation in $T_{d,k}$. As it turns out, the resulting object in $\scrC_{d,k}$ is always link-connected, that is, it belongs to $\scrC_{d,k}^{lc}$. This fact is proven in Subsection \ref{subsec:quotient_is_link_connected} and is used later on to characterize each quotient complex as a certain universal object (see Section \ref{sec:from_simplicial_complexes_to_subgroups_and_back}).


\subsection{Construction of the multicomplex from $G_{d,k}$}\label{subsec:the_construction}

We start by introducing the quotient multicomplex $T_{d,k}\qu H$, which is defined directly from the group $G_{d,k}$ and the subgroup $H$. Let $(T_{d,k},\Gamma,\Om,\CT)$ be a representative of the universal object of $\scrC_{d,k}$ and $H\leq G_{d,k}$. Recall the definition of the subgroups $K_{\widehat{J}}$ (Definition \ref{def:subgroup_K}) and of the sets $\CM^j$ (see Definition \ref{def:cosets_of_G_d_k}).

\begin{definition}[Equivalence classes of cosets of $G_{d,k}$]
	For $-1\leq j\leq d$, define an equivalence relation on cosets in $\CM^j$ by declaring $K_{\widehat{J}}g$ and $K_{\widehat{J'}}g'$ to be equivalent if $\{K_{\widehat{J}}gh ~:~ h\in H\} = \{K_{\widehat{J'}}g'h ~:~ h\in H\}$. For $K_{\widehat{J}}g\in \CM^j$ we denote by $[K_{\widehat{J}}g]_H$ the equivalence class of $K_{\widehat{J}}g$ and define
\begin{equation}
	\CM^j(H) = \{[K_{\widehat{J}}g]_H ~:~ g\in G_{d,k},\, J\subseteq \ls d\rs \text{ such that }|J|=j+1\},
\end{equation} 
Finally,  denote $\CM(H) = \bigcup_{j=-1}^d \CM^j(H)$.
\end{definition}

Note that, for $j=d$, one has $\{gh ~:~ h\in H\}=gH$, and therefore the equivalence classes $[g]_H$ are in correspondence with the left cosets $gH$, that is, one may identify $\CM^d(H)$ with $\{ gH ~:~ g\in G_{d,k}\}$. Hence $|\CM^d(H)| = [G:H]$.

\begin{lemma}\label{lem:well_defind_map_for_double_cosets}
	If $[K_{\widehat{J}} g]_H = [K_{\widehat{J'}}g']_H$, then $J=J'$. Furthermore, if $[K_{\widehat{J}} g]_H = [K_{\widehat{J}}g']_H$ and $I\subseteq J$, then $[K_{\widehat{I}}g]_H = [K_{\widehat{I}}g']_H$. In particular, $[K_{\widehat{i}}g]_H = [K_{\widehat{i}}g']_H$ for every $i\in J$. 
\end{lemma}

\begin{proof}
	Assume that $[K_{\widehat{J}} g]_H = [K_{\widehat{J'}}g']_H$. Then, $K_{\widehat{J}}g=K_{\widehat{J'}}g'h$ for some $h\in H$, and hence $g=wg'h$ for some $w\in K_{\widehat{J'}}$. Consequently, $K_{\widehat{J}}g = K_{\widehat{J'}}g' h = K_{\widehat{J'}}w^{-1}gh^{-1}h = K_{\widehat{J'}}g$, where for the last equality we used the fact that $w\in K_{\widehat{J'}}$. By multiplying both sides from the right by $g^{-1}$, we obtain $K_{\widehat{J}} = K_{\widehat{J'}}$ and therefore $J=J'$. 
	
	Assume next that $[K_{\widehat{J}}g]_H = [K_{\widehat{J}}g']_H$ and $I\subseteq J$. Then, there exists $h\in H$ such that $K_{\widehat{J}}g = K_{\widehat{J}}g'h$. As $K_{\widehat{J}}\subseteq K_{\widehat{I}}$, the same holds for $\widehat{I}$, i.e., $K_{\widehat{I}}g=K_{\widehat{I}}g'h$, which implies $[K_{\widehat{I}}g]_H = [K_{\widehat{I}}g']_H$. 
\end{proof}

Let $\Phi:\CM(H) \to \bigcup_{j=-1}^{d}\binom{\CM^0(H)}{j+1}$ be the map  
\begin{equation}\label{eq:defn_og_Phi_j}
		\Phi([K_{\widehat{J}}g]_H) = \{[K_{\widehat{i}}g]_H ~:~ i\in J\},\qquad \forall J\subseteq \ls d\rs,\, g\in G_{d,k}.
	\end{equation}
	Note that due to Lemma \ref{lem:well_defind_map_for_double_cosets} this map is indeed well defined. Also, note that, $\Phi(\CM^j(H))\subseteq \binom{\CM^0(H)}{j+1}$ and that the restriction of $\Phi$ to $\CM^0$ is the identity. 

\begin{definition}[The complex $\FX_{d,k}(H)$] For $H\leq G_{d,k}$, define $\FX^{-1}_{d,k}(H)=\{\emptyset\}$ and for $0\leq j\leq d$ let 
	\begin{equation}
		\FX^j_{d,k}(H) = \{\Phi([K_{\widehat{J}}g]_H) ~:~ [K_{\widehat{J}}g]_H\in \CM^j(H)\}.
	\end{equation}
Finally, set
	\begin{equation}
		\FX_{d,k}(H) = \bigcup_{j=-1}^d \FX^j_{d,k}(H).
	\end{equation}
\end{definition}

Note that $\FX_{d,k}^0(H) = \CM^0(H)$ since the restriction of $\Phi$ to $\CM^0(H)$ is the identity map.

\begin{lemma}[$\FX_{d,k}(H)$ has a complex structure]\label{lem:the_complex_FX_d,k(H)} Let $(T_{d,k},\Gamma,\Om,\CT)$ be a representative of the universal element and $H\leq G_{d,k}$. Then $\FX_{d,k}(H)$ is a pure $d$-complex which is colorable and $d$-lower path connected. 
\end{lemma}

\begin{proof}
	Let $\si \in \FX_{d,k}(H)$ and $\rho\subseteq \si$. Due to the definition of $\FX_{d,k}(H)$, there exists $J\subseteq \ls d\rs$ and $g\in G_{d,k}$ such that $\si = \Phi([K_{\widehat{J}}g]_H)=\{[K_{\widehat{i}}g]_H ~:~ i\in J\}$. Hence $\rho = \{[K_{\widehat{i}}g]_H ~:~ i\in J'\}$ for some $J'\subset J$, and we obtain $\rho = \Phi([K_{\widehat{J'}}g]_H)$, that is, $\rho\in \FX_{d,k}(H)$. This completes the proof that $\FX_{d,k}(H)$ is a simplicial complex with vertex set $\CM^0(H)$. 
	
	A similar argument shows that $\FX_{d,k}(H)$ is pure. Indeed, let $\si\in \FX_{d,k}(H)$. As before, there exists $J\subseteq \ls d\rs$ and $g\in G_{d,k}$ such that $\si = \Phi([K_{\widehat{J}}g]_H)=\{[K_{\widehat{i}}g]_H ~:~ i\in J\}$. Hence, $\tau = \Phi([g]_H)=\{[K_{\widehat{i}}g]_H ~:~ i\in \ls d\rs\}$ is a $d$-cell containing $\si$. 
	
	Next, we show that $\FX_{d,k}(H)$ is colorable. Define $\gamma_H:\FX_{d,k}^0(H)\to \ls d\rs$ by $\gamma_H([K_{\widehat{i}}g]_H) = i$. Note that due to Lemma \ref{lem:well_defind_map_for_double_cosets} this is a well defined map. Since any $d$-cell of $\FX_{d,k}(H)$ is of the form $\Phi([g]_H)=\{[K_{\widehat{0}}g]_H,[K_{\widehat{1}}g]_H,\ldots,[K_{\widehat{d}}g]_H\}$ for some $g\in G_{d,k}$, it follows that the $0$-cells of any $d$-cell in $\FX_{d,k}(H)$ are colored by $(d+1)$ distinct colors. Hence $\gamma_H$ is a valid coloring, and in particular $\FX_{d,k}(H)$ is colorable. 
	
	Finally, we show that $\FX_{d,k}(H)$ is $d$-lower path connected. Let $\tau,\tau'\in \FX_{d,k}^d(H)$. Then $\tau=\Phi([g]_H)$ and $\tau'=\Phi([g']_H)$ for some $g,g'\in G_{d,k}$. Writing $g'g^{-1}$ as a reduced word $\alpha_{i_m}^{l_m}\ldots\alpha_{i_2}^{l_2}\alpha_{i_1}^{l_1}$ with $i_1,\ldots,i_m\in \ls d\rs$ and $l_1,\ldots,l_m\in [ k-1]$ and abbreviating $w_r=\alpha_{i_r}^{l_r}\ldots\alpha_{i_1}^{l_1}$ for $0\leq r\leq m$, we claim that the sequence $\tau=\tau(0),\tau(1),\ldots,\tau(m)=\tau'$ with $\tau(r)=\Phi([w_rg]_H)$ is a path from $\tau$ to $\tau'$ (up to the fact that it is possible for two consecutive $d$-cells in it to be the same). We only need to check that $\tau(j)\cap \tau(j-1)$ is either a $(d-1)$-cell or that $\tau(j-1)=\tau(j)$. To this end, note that for every $g\in G_{d,k}$, every $i\in \ls d\rs$ and every $l\in [ k-1]$
	\begin{equation}
		\Phi([g]_H)\cap \Phi([\alpha_i^l g]_H) = \{[K_{\widehat{0}}g]_H,\ldots,[K_{\widehat{d}}g]_H\}\cap \{[K_{\widehat{0}}\alpha_i^lg]_H,\ldots,[K_{\widehat{d}}\alpha_i^l g]_H\}.
	\end{equation}
	Using the definition of the subgroups $K_{\widehat{r}}$ (see Definition \ref{def:subgroup_K}) we conclude that $[K_{\widehat{r}}g]_H = [K_{\widehat{r}}\alpha_i^l g]_H$ for every $r\neq i$. Hence $\Phi([g]_H)\cap \Phi([\alpha_i^l g]_H)\in \FX_{d,k}^{d-1}(H)$ if $[K_{\widehat{i}}g]_H \neq [K_{\widehat{i}}\alpha_i^l g]_H$ and $\Phi([g]_H)= \Phi([\alpha_i^l g]_H)$ if $[K_{\widehat{i}}g]_H = [K_{\widehat{i}}\alpha_i^l g]_H$, thus completing the proof. 
\end{proof}

Using the simplicial complex $\FX_{d,k}(H)$, we turn to construct an element of $\scrC_{d,k}$ whose underlying simplicial complex is $\FX_{d,k}(H)$, and is equivalent to the quotient multicomplex described at the beginning of the section. In order to describe the multicomplex it remains to define its multiplicity $\sm_{H}$ and gluing $\sg_H$ functions, as well as its ordering and root. The multiplicity associated with the complex, arises from the fact that although the map $\Phi$ is onto, it is in general not injective, that is, it is possible that $|\Phi^{-1}(\si)|>1$ for some $\si\in \FX_{d,k}(H)$ (for further discussion of this fact see Subsection \ref{subsec:the_intersection_property}).

\begin{definition}[Multiplicity]\label{defn:multiplicity_for_H}
	Define the multiplicity function $\sm_H:\FX_{d,k}(H) \to \BN$ by
	\begin{equation}
		\sm_H(\si) = |\Phi^{-1}(\si)|,\qquad \forall \si\in \FX_{d,k}(H).
	\end{equation}
	Note that by the definition of the complex $\FX_{d,k}(H)$, it is always the case that $\sm(v)=1$ for $v\in \FX_{d,k}^0(H)$. 
\end{definition}

Since the multiplicity is manifested by the map $\Phi$ and the relation between $\CM(H)$ and $\FX_{d,k}(H)$, it is more natural to use the elements of $\CM(H)$ as ``indexes'' for the multicells associated with $\FX_{d,k}(H)$ instead of the natural numbers\footnote{If one insists on working with the set $\FX_{d,k}(H)_{\sm_H}$  instead of $\CM(H)$, this can be done by fixing an arbitrary map $F:\CM(H)\to \FX_{d,k}(H)_{\sm_H}$ such that $F$ restricted to $\Phi^{-1}(\si)$ is a bijection from $\Phi^{-1}(\si)$ to $\FX_{d,k}(H)_{\sm_H}(\si)$, for every $\si\in \FX_{d,k}(H)$.}. Hence, from here onward, we use $\CM^j$ to denote the $j$-multicells of the multicomplex associated with $\FX_{d,k}(H)$, for $1\leq j\leq d$. 

\begin{definition}[Gluing]
	Define the gluing function, $\sg_H: \{(\Fa,\si)\in \CM(H)\times \FX_{d,k}(H) ~:~ \si\in\partial \Phi(\Fa)\} \to \CM(H)$ as follows: Note that for $\Fa=[K_{\widehat{J}}g]_H$ and $\si\in \partial\Phi(\Fa)$, there exists a unique $l\in J$, such that $\si = \Phi([K_{\widehat{J\setminus l}}g]_H)$. Then, set
	\begin{equation}
		\sg_H([K_{\widehat{J}}g]_H,\si) = [K_{\widehat{J\setminus l}}g]_H,
	\end{equation}
	for this unique $l\in J$. 
\end{definition}
The gluing function is well defined due to Lemma \ref{lem:well_defind_map_for_double_cosets}. Furthermore, recalling the definition of a multiboundary, our choice of gluing gives 
\begin{equation}
	\partial^m [K_{\widehat{J}}g]_H = \big\{[K_{\widehat{J\setminus l}} g]_H ~:~ l\in J\big\}.
\end{equation} 
Throughout the remainder of the paper we denote by $\widetilde{\FX}_{d,k}(H)$ the multicomplex $(\FX_{d,k}(H),\sm_H,\allowbreak\sg_H)$. Note that $\widetilde{\FX}_{d,k}(H)$ is indeed a multicomplex as it satisfies the consistency condition. Indeed, let $[K_{\widehat{J}}g]_H$ be a multicell in $\widetilde{\FX}_{d,k}(H)$, and let $[K_{\widehat{I}}g]_H,[K_{\widehat{I'}}g]_H$ be a pair of multicells contained in it (that is, $I,I'\subseteq J$) satisfying $|I|=|I'|$. If in addition the corresponding cells satisfy $\si:=\Phi([K_{\widehat{I}}g]_H)\cap \Phi([K_{\widehat{I'}}g]_H)\in \FX_{d,k}^{|I|-2}(H)$, then $\sg_H([K_{\widehat{I}}g]_H,\si)=[K_{\widehat{I\cap I'}}g]_H=\sg_H([K_{\widehat{I'}}g]_H,\si)$, that is, the gluing is consistent. 

\begin{claim}\label{clm:the_comulticells_of_a_d-1_multicell}
	For every $i\in \ls d\rs$ and $g\in G_{d,k}$, the $d$-multicells containing the $(d-1)$-multicell $[K_i g]_H$ are $\{[\alpha_i^l g]_H ~:~ l\in \ls k-1 \rs\}$. In particular, the degree of any $(d-1)$-multicell is at most $k$. 
\end{claim}

\begin{remark}
	In fact, combining Claim \ref{clm:degree_divides_k} and Lemma \ref{lem:om_H_is_a_good_ordering}, one can show that the degree of any $(d-1)$-cell divides $k$. 
\end{remark}

\begin{proof}
	Let $\Fa=[K_i g]_H\in \CM^{d-1}(H)$ be a $(d-1)$-multicell. It follows from the definition of the gluing function that $[\alpha_i^lg]_H$ for $l\in \ls k-1\rs$ are $d$-multicells containing $\Fa$. If $[g']_H$ is a $d$-multicell containing $\Fa$, then $[K_ig']_H = [K_ig]_H=\Fa$. In particular, $g'=\alpha_i^l gh_0$ for some $h_0\in H$ and $l\in \ls k-1\rs$, and hence $\{g'h ~:~ h\in H\} = \{\alpha_i^l gh_0h ~:~ h\in H\}=\{\alpha_i^l gh ~:~ h\in H\}$, that is, $[g']_H = [\alpha_i^l g]_H$. This proves that there are no other $d$-multicells containing the $(d-1)$-multicell $[K_{\widehat{i}}g]_H$. 
\end{proof}

Next, we turn to define an ordering for the multicomplex $\widetilde{\FX}_{d,k}(H)$. The ordering of $\widetilde{\FX}_{d,k}(H)$ is induced from the ordering of $(T_{d,k},\Gamma,\Om,\CT)$. 

\begin{definition}[Ordering]
	We define $\om_H=(\om_{\Fb})_{\Fb\in \CM^{d-1}(H)}$ as follows. Given a $(d-1)$-multicell $[K_i g]_H\in \CM^{d-1}(H)$ and a $d$-multicell $[\alpha_i^l g]_H\in \CM^d(H)$ containing it, define 
\begin{equation}\label{eq:ordering_rewritten}
	\om_{[K_i g]_H}(\alpha_i^m).[\alpha_i^l g]_H = [\alpha_i^{l+m}g]_H,\qquad \forall m\in\ls k-1\rs.
\end{equation}
\end{definition}

\begin{lemma}\label{lem:om_H_is_a_good_ordering}
	$\om_H$ is a valid $k$-ordering for $\widetilde{\FX}_{d,k}(H)$. 
\end{lemma}

\begin{proof}
	We start by showing that $\om_H$ is well defined. Assume first that $[K_i g]_H=[K_i g']_H$ are two representatives of the same $(d-1)$-multicell and note that in this case $g'=\alpha_i^r gh$ for some $r\in \ls k-1\rs$ and $h\in H$. In addition, the $d$-multicell $[\alpha_i^l g]_H$ represented in terms of $g'$ is given by $[\alpha_i^l g]_H = [\alpha_i^l \alpha_i^{-r}g'h^{-1}]_H = [\alpha_i^{l-r} g']_H$. Thus it is enough to show that $\om_{[K_ig]_H}(\alpha_i^m).[\alpha_i^l g]_H = \om_{[K_ig']_H}(\alpha_i^m).[\alpha_i^{l-r} g']_H$. This is indeed the case, since 
	\begin{equation}
	\begin{aligned}
		\om_{[K_ig]_H}(\alpha_i^m).[\alpha_i^l g]_H &= [\alpha_i^{l+m}g]_H = [\alpha_i^{l+m}\alpha_i^{-r}g'h]_H  \\
		&= [\alpha_i^{l+m-r}g']_H = \om_{[K_ig']_H}(\alpha_i^m).[\alpha_i^{l-r} g']_H.
	\end{aligned}
	\end{equation}
	
	It remains to show that $\om_H$ is indeed an ordering, that is, for every $(d-1)$-multicell $[K_ig]_H$, the map $\om_{[K_i g]_H}:C_k\to \CS_{\delta([K_ig]_H)}$ is a transitive homomorphism. The fact that $\om_{[K_i g]_H}$ is a homomorphism follows from the fact that it is defined using the action of $G_{d,k}$. As for the fact that $\om_{[K_ig]_H}$ is transitive, it follows from by observing that for any $(d-1)$-multicell $[K_{\widehat{i}}g]_H$, one can obtain any $d$-multicell of the form $[\alpha_i^m g]_H$ when starting from the $d$-multicells $gH$ by applying $\om_{[K_ig]_H}(\alpha_i^m)$, and noting that those are all the $d$-multicells containing $[K_{\widehat{i}}g]_H$ (see Claim \ref{clm:the_comulticells_of_a_d-1_multicell}).
\end{proof}

Finally, we define the root to be the $d$-multicell $[e]_H$. Combining all of the above we conclude that $(\widetilde{\FX}_{d,k}(H),\gamma_H,\om_H,[e]_H)$ is an object in the category $\scrC_{d,k}$. 


\subsection{The multicomplexes $\widetilde{\FX}_{d,k}$ as quotients of the universal object}\label{subsec:proof_of_thm_from_groups_to_SC}

Having completed the construction of the object $(\widetilde{\FX}_{d,k}(H),\gamma_H,\om_H,[e]_H)\in\scrC_{d,k}$ associated with a subgroup $H$, we turn to discuss the unique morphism in the sense of the category $\scrC_{d,k}$ from the universal object to it. This morphism is in fact the promised quotient map from $T_{d,k}$ to $T_{d,k}\qu H$. In particular, we prove that the multicomplex associated with the subgroup $H=\langle e\rangle$ is simply the universal object itself (thus proving Theorem \ref{thm:j-cells_in_Tdk_as_cosets}). 

\begin{theorem}[The quotient map] \label{thm:from_subgroups_to_multicomplexes} Let $(T_{d,k},\Gamma,\Om,\CT)$ be a representative of the universal element of $\scrC_{d,k}$ and $H\leq G_{d,k}$. Then the unique morphism from $(T_{d,k},\Gamma,\Om,\CT)$ to $(\widetilde{\FX}_{d,k}(H),\gamma_H,\om_H,\allowbreak[e]_H)$ is given by the map $\si \to [\Psi^{-1}(\si)]_H$, where $\Psi$ is the bijection defined in Corollary \ref{cor:bijection_of_cells_in_T_with_cosets}. 
\end{theorem}

Note that in the case $H=\langle e \rangle$ the resulting morphism from 
$(T_{d,k},\Gamma,\Om,\CT)$ to $(\widetilde{\FX}_{d,k}(e),\gamma_e,\om_e,e)$ is simply the map $\Psi^{-1}$, which by Corollary \ref{cor:bijection_of_cells_in_T_with_cosets} is a bijection. Hence, $(T_{d,k},\Gamma,\Om,\CT)$ and $(\widetilde{\FX}_{d,k}(e),\gamma_e,\om_e,e)$ are isomorphic in the category $\scrC_{d,k}$, which completes the proof of Theorem \ref{thm:j-cells_in_Tdk_as_cosets}.

\begin{proof}
	Denote by $\Pi_H:T_{d,k}\to \CM(H)$ the map $\Pi_H(\si)=[\Psi^{-1}(\si)]_H$ and note that $\Pi_H$ is the composition of the map $\Psi^{-1}$ with the quotient map by $H$ induced from the equivalence relation (see Figure \ref{fig:main_proof}). 
\begin{figure}[h]
	\begin{center}
	\includegraphics[scale=0.5]{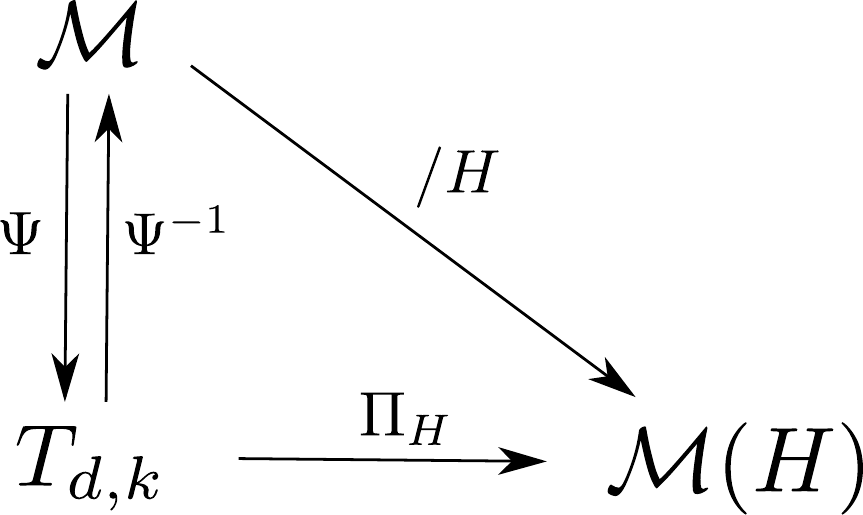}
	\caption{The map $\Pi_H$ is a composition of the map $\Psi^{-1}$ and the quotient map by $H$.\label{fig:main_proof}}
	\end{center}
\end{figure}
In particular, this implies that $\Pi_H$ is well defined. Indeed, for every $\si\in T_{d,k}$, there exists a unique coset $K_{\widehat{J}}g$ such that $\Psi(K_{\widehat{J}}g)=\si$. Furthermore, if $K_{\widehat{J}}g=K_{\widehat{J'}}g'$, then by definition $[K_{\widehat{J}}g]_H=[K_{\widehat{J'}}g']_H$ which proves that the map is well defined. 

Next, we show that $\Pi_H$ is a simplicial multimap. Since $T_{d,k}$ is a complex, and in particular has no multiplicity in the cells, it suffices to show that $\Pi_H(\si) = \{\Pi_H(v_0),\ldots,\Pi_H(v_j)\}\in \CM(H)$ for every $\si=\{v_0,\ldots,v_j\}\in T_{d,k}$. Let $\si$ be as above and assume that $\Psi^{-1}(\si)=K_{\widehat{J}}g$ with $J=\{i_0,\ldots,i_j\}$ and $g\in G_{d,k}$. Due to the definition of $\Psi$ (see Corollary \ref{cor:bijection_of_cells_in_T_with_cosets}) the color of $\si$ is $J$ and hence, without loss of generality, we can assume that the color of $v_r$ is $i_r$ for $0\leq r\leq j$. Using the definition of $\Psi$ once more, together with the fact that $v_r$ is the unique $0$-cell in $\si$ of color $i_r$, we obtain that $\Psi^{-1}(v_r)=K_{\widehat{i_r}}g$ for every $0\leq r\leq j$, and hence $\Pi_H(v_r)=[K_{\widehat{i_r}}g]_H$. Recalling that the $0$-cells of the multicell $\Pi_H(\si)=[K_{\widehat{J}}g]_H$ are $\{[K_{\widehat{i_0}}g]_H,\ldots,[K_{\widehat{i_j}}g]_H\}$ by definition (see the definition of $\Phi$), this completes the proof that $\Pi_H$ is a simplicial multimap. 

It remains to show that $\Pi_H$ preserves the coloring, the root and the ordering. Starting with the coloring, the previous argument implies that for every $\si\in T_{d,k}$ such that $\Psi^{-1}(\si)=K_{\widehat{J}}g$, with $J=\{i_0,\ldots,i_j\}$, the color of $\si$ is $J$. However, the color of the multicell $[K_{\widehat{J}}g]_H$ is defined to be $J$ as well (see Lemma \ref{lem:the_complex_FX_d,k(H)} for the definition of $\gamma_H$). The map $\Pi_H$ preserves the root since $\Psi(e)=\CT$, which implies that  $\Pi_H(\CT)=[e]_H$. Finally, turning to deal with the ordering, let $\si\in T_{d,k}^{d-1}$ and $\tau\in T_{d,k}^d$ such that $\si\subset \tau$. Assume further that $\Psi^{-1}(\si)=K_i g$, which, by the definition of the map $\Psi$, implies that $\Psi^{-1}(\tau)=\alpha_i^lg$ for some $l\in \ls k-1\rs$. Since the action of $G_{d,k}$ on $\CT$ from the left and right is the same (as it corresponds to the unit element in $G_{d,k}$) it follows that $\si=\Si_i.g$ and $\tau=\CT.\alpha_i^lg$, where $\Si_i$ is the unique $(d-1)$-cell of $\CT$ of color $\widehat{i}$. Hence, using the assumption that the ordering $\Om$ is invariant (see \eqref{eq:invariant_ordering_defn}) for every $m\in \ls k-1\rs$
\begin{equation}
\begin{aligned}
	\Pi_H(\Om_\si(\alpha_i^m).\tau) &= \Pi_H(\Om_{\Si_i.g}(\alpha_i^m).(\CT.\alpha_i^lg)) = 
	\Pi_H([\Om_{\Si_i}(\alpha_i^m).\CT].\alpha_i^lg) \\
	&= \Pi_H(\CT.\alpha_i^m\alpha_i^lg) = \Pi_H(\alpha_i^{l+m}g.\CT)=[\alpha_i^{l+m}g]_H \\
	& = \om_{K_igH}(\alpha_i^m).[\alpha_i^lg]_H = \om_{\Pi_H(\si)}(\alpha_i^m).\Pi_H(\tau)
\end{aligned}
\end{equation}
as required. 
\end{proof}


\subsection{Link-connectedness of quotient complexes}\label{subsec:quotient_is_link_connected}

\begin{proposition}\label{prop:quotients_are_link_connected}
	For every $H\leq G_{d,k}$, the multicomplex $\widetilde{\FX}_{d,k}(H)$ is link-connected. 
\end{proposition}

\begin{proof}
	Due to Proposition \ref{prop:equivalent_defn_of_lc} it suffices to show that the link of each $j$-multicell is $(d-j-1)$-lower path connected. Let $\Fa=[K_{\widehat{J}}g]_H$ be a $j$-multicell of $\widetilde{\FX}_{d,k}(H)$ with $-1\leq j\leq d-2$. The link $\widetilde{Y}:=\lk_{\widetilde{\FX}_{d,k}(H)}(\Fa)$ is a $(d-j-1)$-dimensional multicomplex, whose $i$-multicells (for $-1\leq i\leq d-j-1$) are in correspondence with elements $[K_{\widehat{I}}g']_H\in \CM^{i+j+1}(H)$ for $J\subseteq I$ satisfying $[K_{\widehat{J}}g']_H=[K_{\widehat{J}}g]_H$. Let $\Fb'$ and $\Fb''$ be two $(d-j-1)$-multicells in $\widetilde{Y}$ and denote by $[g']_H$ and $[g'']_H$ the corresponding $d$-multicells in $\widetilde{\FX}_{d,k}(H)$. Since both $[g']_H$ and $[g'']_H$ contain $\Fa$, we have $[K_{\widehat{J}}g']_H=[K_{\widehat{J}}g]_H=[K_{\widehat{J}}g'']_H$. Observing the corresponding $d$-cells in $T_{d,k}$, by Theorem \ref{thm:from_subgroups_to_multicomplexes}, one can find $h',h''\in H$ such that the $d$-cells $\Psi^{-1}(g'h')$ and $\Psi^{-1}(g''h'')$ contain $K_{\widehat{J}}g$. As $K_{\widehat{J}}g$ corresponds to a $j$-cell of $T_{d,k}$, and since the link of a $j$-cell in $T_{d,k}$ is isomorphic to $T_{d-j-1,k}$ (see Proposition \ref{prop:links_in_T^d_k}), we conclude that: $(a)$ the link $\lk_{T_{d,k}}(\Psi^{-1}(K_{\widehat{J}}g))$ is $(d-j-1)$-lower path connected, $(b)$	$g'h'$ and $g''h''$ have a corresponding $(d-j-1)$-cells in the link  $\lk_{T_{d,k}}(K_{\widehat{J}}g)$. Combining $(a)$ and $(b)$, one can find a path of $(d-j-1)$-cells in $\lk_{T_{d,k}}(K_{\widehat{J}}g)$, which in turn yields a path of $d$-cells in $T_{d,k}$ from $g'h'$ to $g''h''$, all of whose $d$-cells contain the $j$-cell $K_{\widehat{J}}g$. Projecting this path to $\widetilde{\FX}_{d,k}(H)$, one obtains a path of $d$-multicells from $[g'h']_H=[g']_H$ to $[g''h'']_H=[g'']_H$, all of which contain the $j$-multicell $[K_{\widehat{J}}g]_H$. Finally, projecting this path to the link $\widetilde{Y}$, we conclude that there is a path of $(d-j-1)$-multicells from $\Fb'$ to $\Fb''$ in the link, thus completing the proof. 
\end{proof}


\subsection{Links of $T_{d,k}$ described by the group $G_{d,k}$}

Combining Proposition \ref{prop:links_in_T^d_k} and  Theorem \ref{thm:j-cells_in_Tdk_as_cosets} we can give a group description of the links in $T_{d,k}$ using the group $G_{d,k}$. 

Let $(T_{d,k},\Gamma,\Om,\CT)$ be a representative of the universal object, and fix a cell  $\si\in T_{d,k}^j$ of color $\Gamma(\si)=J$. Assume first that $\Phi^{-1}(\si) = K_{\widehat{J}}e$. Then, the set of cells in $T_{d,k}$ containing $\si$ is in correspondence with the cosets $K_{\widehat{I}}g'\in \CM$ such that $J\subseteq I$ and $K_{\widehat{J}}=K_{\widehat{J}}g'$, or equivalently, $J\subseteq I$ and $g'\in K_{\widehat{J}}$. Furthermore, all the cell in the link of $\si$ are colored by the colors $\ls d\rs _J := \ls d\rs \setminus J$, and the color of the cell corresponding to $K_{\widehat{I}}g'$ is $\ls d\rs_J \setminus I = \ls d\rs_J \setminus (I\setminus J)$. 

Let $G_{d,k}^J = \ast_{i\in \ls d\rs_J} K_i \leq G_{d,k}$ be the subgroup of generated by $(\alpha_i)_{i\in \ls d\rs_J}$. Note that 
\begin{itemize}
	\item $G_{d,k}^J$ is isomorphic to $G_{d-j-1,k}$.
	\item By Theorem \ref{thm:j-cells_in_Tdk_as_cosets}, $G_{d-j-1,k}$ gives rise to the simplicial complex $T_{d-j-1,k}$ using the map $\Phi$. 
	\item By Proposition \ref{prop:links_in_T^d_k} The link of $\si$ is isomorphic to $T_{d-j-1,k}$.
\end{itemize} 

Hence we can also use the group $G_{d,k}$ in order to describe the link of $\si$, namely, $K_{\widehat{I}}g'$ with $J\subseteq I$ and $g'\in K_{\widehat{J}}\leq G_{d,k}$ corresponds to the coset $K_{\ls d\rs_J \setminus (I\setminus J)}g'$ in $G_{d,k}^J$, which in turn corresponds to a cell in the link of $\si$. 

Finally, in the general case, i.e. $\Phi^{-1}(\si)=K_{\widehat{J}}g$ for a general $g\in G_{d,k}$, we can give a natural structure, by an appropriate conjugation by $g$.


\section{From simplicial complexes to subgroups and back}\label{sec:from_simplicial_complexes_to_subgroups_and_back}

In this section we discuss how to associate a subgroup of $G_{d,k}$ to objects of $\scrC_{d,k}$ and study the relation between the original object and the quotient object obtained from the associated subgroup.


\subsection{Classification of objects in $\scrC_{d,k}^{lc}$}

Recall that we have an action from the left of the group $G_{d,k}$ on the $d$-multicells of every object in the category $\scrC_{d,k}$ (see Section \ref{sec:group_action}). Note that this action is not in the category $\scrC_{d,k}$. 

\begin{definition}[The associated subgroup]
	For $(\widetilde{X},\gamma,\om,\Fa_0)\in \scrC_{d,k}$ define 
	\begin{equation}
		\CH((\widetilde{X},\gamma,\om,\Fa_0)) = \{g\in G_{d,k} ~:~ g.\Fa_0 = \Fa_0\}.
	\end{equation}
	Since $\CH((\widetilde{X},\gamma,\om,\Fa_0))$ is the stabilizer of $\Fa_0$, it is a subgroup of $G_{d,k}$, that we will call the subgroup associated with $(\widetilde{X},\gamma,\om,\Fa_0)$.
\end{definition}

\begin{theorem}[Classification of objects in $\scrC_{d,k}^{lc}$]\label{thm:main_theorem} Let $(T_{d,k},\Gamma,\Om,\CT)$ be a representative of the universal object. 
\begin{enumerate}[(1)]
	\item For any $H\leq G_{d,k}$, the subgroup associated with the multicomplex $(\widetilde{\FX}_{d,k}(H),\allowbreak\gamma_H,\om_H,[e]_H)$ is $H$, that is, 
	\begin{equation}\label{eq:functor_1}
		\CH((\widetilde{\FX}_{d,k}(H),\gamma_H,\om_H,[e]_H))=H.
	\end{equation}
	\item For every $(\widetilde{X},\gamma,\om,\Fa_0)\in \scrC_{d,k}^{lc}$, the object of  $\scrC_{d,k}$ associated with $\CH((\widetilde{X},\gamma,\om,\Fa_0))$ is isomorphic to $(\widetilde{X},\gamma,\om,\Fa_0)$, that is 
	\begin{equation}
		(\widetilde{X},\gamma,\om,\Fa_0) \cong (\widetilde{\FX}_{d,k}(H),\gamma_H,\om_H,[e]_H),
	\end{equation}
	where $H=\CH((\widetilde{X},\gamma,\om,\Fa_0))$.
\end{enumerate}
In particular, we obtain a correspondence between subgroups of $G_{d,k}$ and the link-connected objects in $\scrC_{d,k}$. 
\end{theorem}

\begin{proof}$~$
\begin{enumerate}[(1)]
	\item The main identity needed for the proof of \eqref{eq:functor_1} is
	\begin{equation}\label{eq:action_of_g_on_cosets}
		g.[e]_H = [g]_H,\qquad \forall g\in G_{d,k},
	\end{equation}
	where the action of $G_{d,k}$ is the one defined by the group action on the $d$-multicells of $\widetilde{\FX}_{d,k}(H)$. Indeed, if \eqref{eq:action_of_g_on_cosets} holds, then
	\begin{equation}
	\begin{aligned}
		\CH((\widetilde{\FX}_{d,k}(H),\gamma_H,\om_H,[e]_H))
		&=\{g\in G_{d,k} ~:~ g.[e]_H) = [e]_H\}\\
		& =\{g\in G_{d,k} ~:~ [g]_H = [e]_H\}=H
	\end{aligned}
	\end{equation} 
thus completing the proof. 
	
	We turn to prove \eqref{eq:action_of_g_on_cosets}. Since $g$ can be written as a word in $\alpha_0,\ldots,\alpha_d$ it is enough to show that $\alpha_i.[g]_H = [\alpha_ig]_H$ for any $g\in G_{d,k}$ and $i\in \ls d\rs$. The last equality is now a consequence of the definition of the action and the ordering. Indeed, for any $i\in\ls d\rs$ and $g\in G_{d,k}$
	\begin{equation}
		\alpha_i.[g]_H = \om_{[K_i g]_H}(\alpha_i).[g]_H = [\alpha_ig]_H
	\end{equation}
	as required.

	\item Assume $\widetilde{X}=(X,\sm,\sg)$. We construct the morphism from $(\widetilde{\FX}_{d,k}(H),\gamma_H,\om_H,[e]_H)$ to $(\widetilde{X},\gamma,\om,\allowbreak\Fa_0)$ in several steps: (i) a map between $0$-cells, $(ii)$ extension to a simplicial map from $\FX_{d,k}(H)$ to $X$, $(iii)$ extension to a simplicial multimap from $\widetilde{\FX}_{d,k}(H)$ to $\widetilde{X}$, and finally, $(iv)$ showing that the resulting simplicial multimap is a bijective morphism, whenever $\widetilde{X}$ is link-connected.
	
	$(i)$ Recall that the $0$-cells have no multiplicity in any complex and define $\varphi:\FX_{d,k}^0(H) \to X^0$ by 
	\begin{equation}
		\varphi ([K_{\widehat{i}}g]_H) = \begin{array}{c}
		\text{the unique }0\text{-cell of}\\
		g.\Fa_0\text{ whose color is }i
		\end{array}.
	\end{equation}
	
	First, we show that $\varphi$ is well defined. Indeed, if $[K_{\widehat{i}}g]_H = [K_{\widehat{i}}g']_H$, then $g'=wgh$ for some $w\in K_{\widehat{i}}$ and $h\in H$ and therefore $g'.\Fa_0 = wgh.\Fa_0 = wg.\Fa_0 =w.(g.\Fa_0)$, where for the second inequality we used the fact that any element of $H$ stabilizes $\Fa_0$. Recalling the definition of the action, for any element of $w'\in K_{\widehat{i}}$ and any $\Fa\in X_\sm$ it holds that $w'.\Fa$ and $\Fa$ contain the same $0$-cell whose color is $i$. In particular $g.\Fa_0$ and $g'.\Fa_0=w.(g.\Fa_0)$ contain the same $0$-cell whose color is $i$, implying that $\varphi([K_{\widehat{i}}g]_H)=\varphi([K_{\widehat{i}}g']_H)$. 
	
	$(ii)$ Next, we extend the definition of $\varphi$ to all cells in $\FX_{d,k}(H)$ by $\varphi(\{\si_0,\ldots,\si_j\})=\{\varphi(\si_0),\ldots,\allowbreak\varphi(\si_j)\}$ for every $\si=\{\si_0,\ldots,\si_j\}\in T_{d,k}^j$. We claim that $\varphi$ is a simplicial map. In order to prove this, one only need to verify that $\varphi(\{[K_{\widehat{i_0}}g]_H,\ldots,[K_{\widehat{i_j}}g]_H\})$ is a $j$-cells in $X$ for every choice of $g\in G_{d,k}$ and distinct $i_0,\ldots,i_j\in \ls d\rs$. Using the definition of $\varphi$ for $0$-cells, one can verify that $\varphi(\{[K_{\widehat{i_0}}g]_H,\ldots,[K_{\widehat{i_j}}g]_H\})$ is the set of $0$-cells of $g.\Fa_0$ whose colors are $i_0,i_1,\ldots,i_j$. Using the forgetful map $\iota$ from multicells to cells defined in Section \ref{sec:multicomplexes}, since the colors are distinct and since $\iota(g.\Fa_0)$ is a $d$-cell of $X$ (which is closed under inclusion), we conclude that $\varphi(\{[K_{\widehat{i_0}}g]_H,\ldots,[K_{\widehat{i_j}}g]_H\})$ is a $j$-cell in $X$. 
	
	$(iii)$ A final extension of $\varphi$ to a map $\widetilde{\varphi}:\widetilde{X}\to \widetilde{\FX}_{d,k}(H)$ is defined as follows. For $0\leq j\leq d$, $J\subset \ls d\rs$ satisfying $|J|=j+1$ and $g\in G_{d,k}$ define
	\begin{equation}
		\widetilde{\varphi}([K_{\widehat{J}}g]_H) = \begin{array}{c}
		\text{the unique }j\text{-multicell contained}\\
		\text{in }g.\Fa_0\text{ whose color is }J
		\end{array}.
	\end{equation}
	Using a similar argument to the one above, we obtain that $\widetilde{\varphi}$ is  an extension of the map $\varphi$. Furthermore, this map is well defined due to the definition of $H$ and the fact that the action of elements in $K_{\widehat{J}}$ from the left on $d$-multicell stabilizes the $(|J|-1)$-cell of color $J$. 
	
	$(iv)$ Let us now show that $\widetilde{\varphi}$ is a bijection of multicomplexes. Indeed, since the action of $G_{d,k}$ on the $d$-multicells of $\widetilde{X}$ is transitive (see Claim \ref{clm:action_of_G_d_k_from_the_left_is_transitive}), it follows that the map is onto. As for injectivity, we separate the  proof to the case of $d$-multicells and to lower dimensional multicells. Starting with the former, note that if $\widetilde{\varphi}([g]_H)=\widetilde{\varphi}([g']_H)$, then $g.\Fa_0 = g'.\Fa_0$ and hence $g^{-1}g'.\Fa_0=\Fa_0$. Recalling the definition of the subgroup $H$, we obtain that $g^{-1}g'\in H$, or equivalently $[g]_H=[g']_H$. Turning to deal with general multicells, if $\widetilde{\varphi}([K_{\widehat{J}}g]_H) = \widetilde{\varphi}([K_{\widehat{J'}}g']_H)$, then $J=J'$, because otherwise the colors or dimension of the multicells are different and in particular there is no equality. Note that the equality $\widetilde{\varphi}([K_{\widehat{J}}g]_H) = \widetilde{\varphi}([K_{\widehat{J}}g']_H)$  implies that the $d$-multicells $g.\Fa_0$ and $g'.\Fa_0$ contains the same $(|J|-1)$-multicell of color $J$, denoted $\Fb$. Therefore, in the link $\lk_{\widetilde{X}}(\Fb)$, there are $(d-\dim\Fb-1)$-multicells $\Fc$ and $\Fc'$ which correspond to the $d$-multicells $g.\Fa_0$ and $g'.\Fa_0$ respectively. Recalling that $\widetilde{X}$ is link-connected and using Proposition \ref{prop:equivalent_defn_of_lc}, we conclude that there exists a path of $(d-\dim\Fb-1)$-multicells in $\lk_{\widetilde{X}}(\Fb)$ from $\Fc$ to $\Fc'$ such that each pair of consecutive multicells along the path contains a common $(d-\dim\Fb-2)$-multicell. Lifting this path to $\widetilde{X}$, one obtain a path of $d$-multicells from $g.\Fa_0$ to $g'.\Fa_0$ such that each consecutive pair contains a common $(d-1)$-multicell, and each of the $d$-multicells contains the multicell $\Fb$. Recalling the definition of the action of $G_{d,k}$ (see Section \ref{sec:group_action}) this implies the existence of $w\in K_{\widehat{J}}$ such that $wg.\Fa_0=w.(g.\Fa_0)=g'.\Fa_0$. Hence, by the same argument used for $d$-multicells $[wg]_H=[g']_H$, and therefore $[K_{\widehat{J}}g]_H = [K_{\widehat{J}}g']_H$. 
	
	Finally, we note that $\widetilde{\varphi}$ preserves the coloring ordering and root. The coloring is preserved due to the definition of $\widetilde{\varphi}$ and the choice of coloring $\gamma_H$ for cells in $\FX_{d,k}(H)$. Similarly, the root is preserved, due to the definition of $\varphi$. As for the ordering, due to the definition of $\om_H$, see \eqref{eq:ordering_rewritten}, one needs to show that for every $i\in \ls d\rs$, $g\in G_{d,k}$ and $l,m\in \ls k-1\rs$
	\begin{equation}
		\widetilde{\varphi}([\alpha_i^{l+m}g]_H)= \om_{\widetilde{\varphi}([K_{\widehat{i}}g]_H)}(\alpha_i^l).\widetilde{\varphi}([\alpha_i^m g]_H).
	\end{equation}
	
	Starting with the left hand side, $\widetilde{\varphi}([\alpha_i^{l+m}g]_H)$ is the unique $d$-multicell $\alpha_i^{l+m}g.\Fa_0$. As for the right hand side, denoting by $\Fb$ the unique $(d-1)$-multicell contained in $\alpha_i^mg.\Fa_0$ whose color is $\ls d\rs \setminus i$, this can be rewritten as $\om_{\Fb}(\alpha_i^l).(\alpha_i^mg.\Fa_0)$ which by the definition of the action on $(\widetilde{X},\gamma,\om,\Fa_0)$ equals to $\alpha_i^{l+m}g.\Fa_0$ as well. 
\end{enumerate}
\end{proof}

\subsection{Some implications of Theorem \ref{thm:main_theorem}}

\begin{proposition}\label{prop:the_line_graph_of_widetilde_FX_d_k_H}
	For every subgroup $H\leq G_{d,k}$, the line graph of $\widetilde{\FX}_{d,k}(H)$ is isomorphic to the Schreier graph $\textrm{Scr}(G/H;S)$, where $S=\{\alpha_i^l ~:~ i\in \ls d\rs,\,l\in [k-1]\}$.
\end{proposition}

\begin{proof}
	Since $G_{d,k}$ acts transitively on the $d$-multicells of $(\widetilde{\FX}_{d,k}(H),\gamma_H,\om_H,[e]_H)$, and since the stabilizer of $[e]_H$ is the subgroup $H$, it follows that $\scrG(\widetilde{\FX}_{d,k}(H))$ is isomorphic to the Schreier graph of $G_{d,k}/H$ with respect to the generator set $S'=\{g\in G_{d,k} ~:~ g.[e]_H \text{ is a neighbor of }[e]_H\}$. Hence, it remains to show that $S'=S$. First, note that $S\subseteq S'$, since the definition of the left action on $\widetilde{\FX}_{d,k}(H)$ guarantees that the $d$-multicells $[e]_H$ and $\alpha_i^l.[e]_H$ contain the same $(d-1)$-multicell of color $\widehat{i}$, that is, $\alpha_i^l.[e]_H$ and $[e]_H$ are neighboring $d$-cells in the line graph. As for the other direction, assume that $g\in S'$. Then, $g.[e]_H$ is a neighbor of $[e]_H$ in the line graph, that is, they contain a common $(d-1)$-multicell. Recalling that the $(d-1)$-multicells contained in $[e]_H$ are $([K_ie]_H)_{i\in \ls d\rs}$ and using Claim \ref{clm:the_comulticells_of_a_d-1_multicell}, we conclude that $[g]_H=[\alpha_i^l]_H$ for some $l\in \ls k-1\rs$, thus proving that $S'\subseteq S$. 
\end{proof}

Combining proposition \ref{prop:the_line_graph_of_widetilde_FX_d_k_H} and Theorem \ref{thm:main_theorem} we obtain

\begin{corollary}$~$
	\begin{enumerate}[(1)]
		\item For every $d,k$ and $H\leq G_{d,k}$ there exists a $d$-multicomplex, whose line graph is isomorphic to the Schreier graph $\textrm{Sch}(G_{d,k}/H;S)$, where $S=\{\alpha_i^l ~:~ i\in \ls d\rs,\,l\in [k-1]\}$. 	
		\item Every pure, link-connected $d$-multicomplex $\widetilde{X}$, such that $\deg(\Fb)$ divides $k$ for every $(d-1)$-multicell $\Fb$ is isomorphic to $\widetilde{\FX}_{d,k}(H)$ for some $H\leq G_{d,k}$. 
	\end{enumerate}
\end{corollary}

\begin{remark}
	Let $(\widetilde{X},\gamma,\om,\Fa)\in \scrC_{d,k}$ and denote by $H$ the corresponding subgroup of $G_{d,k}$. It follows from the work of \cite[Corollary 4.1]{Lee16} that the line graph of $\widetilde{X}$ is transitive if and only if $H$ is length-transitive (see \cite[Definition 4.3]{Lee16}). It is interesting to ask, whether one can find a condition on $H$ that guarantees transitivity of the line graphs of the skeletons of $\widetilde{X}$. Similarly, let $H_1,H_2$ be two subgroups of $G_{d,k}$. It is shown in \cite[Theorem 4.1]{Lee16} that the line graphs of $\widetilde{\FX}_{d,k}(H_1)$ and $\widetilde{\FX}_{d,k}(H_2)$ are isomorphic if and only if $H_1$ and $H_2$ are length-isomorphic (see \cite[Definition 4.2]{Lee16}). It is interesting to ask whether one can find conditions on $H_1$ and $H_2$ that guarantee isomorphism of the line graphs of the skeletons of $\widetilde{\FX}_{d,k}(H_1)$ and $\widetilde{\FX}_{d,k}(H_2)$. 
\end{remark}

Using Theorem \ref{thm:main_theorem} we may also deduce Proposition \ref{prop:uniqueness_of_T_d_k}. Indeed, the assumption that $T_{d,k}$ is pure $d$-dimensional, $k$-regular and has the non-backtracking property described in Proposition \ref{prop:uniqueness_of_T_d_k} determines the line graph of $T_{d,k}$ completely. Hence, by Theorem \ref{thm:main_theorem} there exists a unique object in $\scrC_{d,k}^{lc}$ up to isomorphism whose line graph is the line graph of $T_{d,k}$, which must be $T_{d,k}$ itself. 

Another corollary of Theorem \ref{thm:main_theorem} is the following:

\begin{corollary}\label{cor:covering_two_complexes}
	For every pair of finite objects $(X_1,\gamma_1,\om_1,\Fa_1)$ and $(X_2,\gamma_2,\om_2,\Fa_2)$ in $\scrC_{d,k}$, there exists a finite $(Y,\gamma,\om,\Fa)\in \scrC_{d,k}^{lc}$ with surjective morphisms $\pi_i:(Y,\gamma,\om,\Fa)\to (X_i,\gamma_i,\om_i,\Fa_i)$, $i=1,2$.
\end{corollary}

\begin{proof}
	If $H_1$ and $H_2$ are subgroups of $G_{d,k}$ associated with $(X_1,\gamma_1,\om_1,\Fa_1)$ and $(X_2,\gamma_2,\om_2,\Fa_2)$ respectively, then take $(Y,\gamma,\om,\Fa) = (\widetilde{\FX}_{d,k}(H_3),\gamma_{H_3},\om_{H_3},[e]_{H_3})$, where $H_3=H_1\cap H_2$. 
\end{proof}

Corollary \ref{cor:covering_two_complexes} is a kind of analogue of Leighton's graph covering theorem (see \cite{Lei82}) asserting that two finite graphs with the same universal cover have a common finite cover. It should be noted however, that in our case $T_{d,k}$ is usually not the universal cover of neither $X_1$ nor $X_2$. Moreover, in general the fundamental groups $\pi_1(X_1)$ and $\pi_1(X_2)$ can be very different and, in particular, not commensurable to each other (i.e., do not have isomorphic finite index subgroups) as is the case for graphs. Therefore, Corollary \ref{cor:covering_two_complexes} is not totally expected. In fact we do not know how to prove it by direct combinatorial methods.


\subsection{The link-connected cover}

\begin{theorem}\label{thm:H_universality}
	For every $(\widetilde{X},\gamma,\om,\Fa_0)\in \scrC_{d,k}$ there exists a unique (up to isomorphism) object $(\widetilde{Y},\widehat{\gamma},\widehat{\om},\widehat{\Fa}_0)\in \scrC_{d,k}$ with an epimorphism $\pi:(\widetilde{Y},\widehat{\gamma},\widehat{\om},\widehat{\Fa}_0)\to (\widetilde{X},\gamma,\om,\Fa_0)$ satisfying the following properties
	\begin{enumerate}[(1)]
		\item $\widetilde{Y}$ is link-connected.
		\item For every $(\widetilde{Z},\overline{\gamma},\overline{\om},\overline{\Fa}_0)\in \scrC_{d,k}^{lc}$ and every epimorphism $\varphi:(\widetilde{Z},\overline{\gamma},\overline{\om},\overline{\Fa}_0)\to (\widetilde{X},\gamma,\om,\Fa_0)$, there exists an epimorphism $\psi:(\widetilde{Z},\overline{\gamma},\overline{\om},\overline{\Fa}_0)\to (\widetilde{Y},\widehat{\gamma},\widehat{\om},\widehat{\Fa}_0)$ such that $\pi \circ \psi = \varphi$. 
	\end{enumerate}
	Furthermore, the map $\pi$ induces an isomorphism between the line graphs of $\widetilde{X}$ and $\widetilde{Y}$, and the object $(\widetilde{Y},\widehat{\gamma},\widehat{\om},\widehat{\Fa}_0)$ is given by $(\widetilde{\FX}_{d,k}(H),\gamma_H,\om_H,[e]_H)$.
\end{theorem}

Given $(\widetilde{X},\gamma,\om,\Fa_0)\in \scrC_{d,k}$ we call the unique object satisfying Theorem \ref{thm:H_universality} the \emph{link-connected cover of $(\widetilde{X},\gamma,\om,\Fa_0)$} and denote it by $\textrm{LCC}((\widetilde{X},\gamma,\om,\Fa_0))$. 

\begin{proof}
Let $(\widetilde{X},\gamma,\om,\Fa_0)\in \scrC_{d,k}$, and denote $H=\CH((\widetilde{X},\gamma,\om,\Fa_0))$. By repeating the argument in Theorem \ref{thm:main_theorem}$(2)$, one can show that the simplicial multimap $\pi:\widetilde{\FX}_{d,k}(H)\to \widetilde{X}$ 
	\begin{equation}
		\pi([K_{\widehat{J}}g]_H) = \begin{array}{c}
		\text{the unique }j\text{-multicell contained}\\
		\text{in }g.\Fa_0\text{ whose color is }J
		\end{array}.
	\end{equation}
	is a surjective morphism $\scrC_{d,k}$ which is injective at the level of $d$-cells. Furthermore, by Proposition \ref{prop:quotients_are_link_connected}, $\widetilde{\FX}_{d,k}(H)$ is link-connected. 
	
	Next, we show that $(\widetilde{\FX}_{d,k}(H),\gamma_H,\om_H,[e]_H)$ satisfies property $(2)$. Assume that $(\widetilde{Z},\overline{\gamma},\overline{\om},\overline{\Fa}_0)\in \scrC_{d,k}^{lc}$ and $\varphi:(\widetilde{Z},\overline{\gamma},\overline{\om},\overline{\Fa}_0)\to (\widetilde{X},\gamma,\om,\Fa_0)$ is an epimorphism in the category. Since $(\widetilde{Z},\overline{\gamma},\overline{\om},\overline{\Fa}_0)\in \scrC_{d,k}^{lc}$, it follows from Theorem \ref{thm:main_theorem} that $(\widetilde{Z},\overline{\gamma},\overline{\om},\overline{\Fa}_0)$ is isomorphic to $(\widetilde{\FX}_{d,k}(H'),\gamma_{H'},\allowbreak\om_{H'},[e]_{H'})$, where $H'=\CH((\widetilde{Z},\overline{\gamma},\overline{\om},\overline{\Fa}_0))$. Since there is an epimorphism from $(\widetilde{Z},\overline{\gamma},\overline{\om},\overline{\Fa}_0)$ onto $(\widetilde{X},\gamma,\om,\Fa_0)$, there is also an epimorphism from $(\widetilde{\FX}_{d,k}(H'),\gamma_{H'},\om_{H'},\allowbreak[e]_{H'})$ onto $(\widetilde{X},\gamma,\om,\Fa_0)$. Furthermore, since the line graph of $(\widetilde{X},\gamma,\om,\Fa_0)$ is isomorphic to $\mathrm{Sch}(G_{d,k}/H;S)$, where $S=\{\alpha_i^l ~:~ i\in \ls d-1\rs,\, l\in [k-1]\}$, we must have that $H'\leq H$. Consequently, there exists an epimorphism from $(\widetilde{\FX}_{d,k}(H'),\gamma_{H'},{H'},[e]_{H'})$ onto $(\widetilde{\FX}_{d,k}(H),\gamma_H,\om_H,[e]_H)$, given by $\psi([K_{\widehat{J}}g]_{H'})=[K_{\widehat{J}}g]_H$. 
	
	The uniqueness of the object in the theorem up to isomorphism follows from property $(2)$. 
\end{proof}

Let $(\widetilde{X},\gamma,\om,\Fa_0)\in \scrC_{d,k}$ and denote $H=\CH((\widetilde{X},\gamma,\om,\Fa_0))$. Due to the last theorem, the link-connected cover of $(\widetilde{X},\gamma,\om,\Fa_0)$, denoted $\textrm{LCC}((\widetilde{X},\gamma,\om,\Fa_0))$ is given by $(\widetilde{\FX}_{d,k}(H),\gamma_H,\om_H,[e]_H)$. As the explicit construction of the subgroup $H$ associated with an object $(\widetilde{X},\gamma,\om,\Fa_0)$ might be somewhat involved, let us point out a direct combinatorial way to construct $LCC((\widetilde{X},\gamma,\om,\Fa_0))$ out of $(\widetilde{X},\gamma,\om,\Fa_0)$. 

\paragraph{The algorithm.} Let $(\widetilde{X},\gamma,\om,\Fa_0)\in\scrC_{d,k}$ with $\widetilde{X}=(X,\sm,\sg)$. \\
	\indent	$\bullet$ Run over $j$ from $d-2$ to $-1$ \\
	\indent\indent$\bullet$ Run over $\Fb\in X_\sm^j$\\
	\indent\indent\indent$\bullet$ If $\lk_{\widetilde{X}}(\Fb)$ is not connected\\
	\indent\indent\indent\indent$\bullet$ split $\Fb$ into $l$ new copies $\Fb_1,\ldots,\Fb_l$, where $l$ is the number of connected\\
	\indent\indent\indent\indent\quad component of $\lk_{\widetilde{X}}(\Fb)$ and change the multiplicity function $\sm$ accordingly. \\
	\indent\indent\indent\indent$\bullet$ change the gluing function $\sg$, so that each of the multicells in $\widetilde{X}$, associated with\\
	\indent\indent\indent\indent\quad the $i$-th connected component of $\lk_{\widetilde{X}}(\Fb)$, are glued to $\Fb_i$ instead of $\Fb$,\\
	\indent\indent\indent\indent\quad for every $1\leq i\leq l$. \\
	
	Note that by following the above algorithm, one obtain a new multicomplex in which the line graph is not changed and the link of each multicell is connected, that is, we recover the link-connected cover of $(\widetilde{X},\gamma,\om,\Fa_0)$.


\section{Further relations between subgroups and multicomplexes}\label{sec:further_relations}

In this section we wish to study further relations between subgroups of $G_{d,k}$ and elements of $\scrC_{d,k}$. More precisely, we study the following questions:
\begin{enumerate}
	\item When does $\FX_{d,k}(H)$ have a natural representation as a nerve complex?
	\item When is the multicomplex associated with a subgroup $H$ actually a complex?
	\item For which subgroups $H$ is the multicomplex upper $k$-regular? 
	\item When is the $(d-1)$-skeleton of $\widetilde{\FX}_{d,k}$ complete?
\end{enumerate}


\subsection{Representation of $\FX_{d,k}(H)$ as a nerve complex}\label{subsec:representation_as_nerve}

\begin{theorem}\label{thm:representation_as_a_nerve_complex}
	Let $H\leq G_{d,k}$. For $v=[K_{\widehat{i}}g]_H \in \FX_{d,k}^0(X)$, define $\CA_v = \{g'\in G_{d,k} ~:~ [K_{\widehat{i}}g]_H = [K_{\widehat{i}}g']_H\}$ and let $\CA = (\CA_v)_{v\in \FX_{d,k}^0(X)}$. Then 
	\begin{equation}
		\FX_{d,k}(H) = \CN(\CA),
	\end{equation} 
	where $\CN(\CA)$ is the nerve complex defined in Definition \ref{def:nerve_complex}, i.e., the base complex of $\widetilde{\FX}_{d,k}$ is equal to the nerve complex of the system $\CA$. 
\end{theorem}

\begin{remark}
	Combining Theorem \ref{thm:from_subgroups_to_multicomplexes} and Theorem \ref{thm:representation_as_a_nerve_complex} with $H=\langle e\rangle$, we obtain Theorem \ref{thm:nerve_version_of_T_d_k}.
\end{remark}

\begin{proof}
	Observe that the $0$-cells of both complexes are $\FX^0_{d,k}(H)=\CM^0(H)$ by definition. Given $\si\in \FX_{d,k}(H)$, one can find $g\in G_{d,k}$ and $i_0,\ldots,i_j\in \ls d\rs$ such that $\si = \{[K_{\widehat{i_0}}g]_H,\ldots,[K_{\widehat{i_j}}g]_H\}$  and therefore 
	\begin{equation}
		g\in \bigcap_{r=0}^j \CA_{v_r}, 
	\end{equation}
	where we denote $v_r=[K_{\widehat{i_r}}g]_H$ for $0\leq r\leq j$. In particular $\bigcap_{r=0}^j \CA_{v_r}\neq \emptyset$, which implies $\si=\{v_0,\ldots,v_j\}\in \CN(\CA)$. 
	
	Turning to prove inclusion in the other direction, assume that $\si=\{v_0,\ldots,v_j\}$, with $v_r = [K_{\widehat{i_r}}g_r]_H$ for $0\leq r\leq j$ is a cell in $\CN(\CA)$. Then, by the definition of the nerve complex $\emptyset \neq \bigcap_{r=0}^j \CA_{v_r}$, which implies that one can find $g\in G_{d,k}$ such that $[K_{\widehat{i_r}}g]_H=[K_{\widehat{i_r}}g_r]_H$, or equivalently $v_r=[K_{\widehat{i_r}}g]_H$ for every $0\leq r\leq j$. Denoting $J=\{i_0,\ldots,i_j\}$, the multicell $[K_{\widehat{J}}g]_H$ in $\widetilde{\FX}_{d,k}(H)$, has $\{[K_{\widehat{i_0}}g]_H,\ldots,[K_{\widehat{i_j}}g]_H\}=\{v_0,\ldots,v_j\}=\si$ as a corresponding cell in $\FX_{d,k}(H)$, which gives $\si\in \FX_{d,k}(H)$.  
\end{proof}


\subsection{The intersection property of subsgroups}\label{subsec:the_intersection_property}

Let $H\leq G_{d,k}$ and $(\widetilde{\FX}_{d,k}(H),\gamma_H,\om_H,[e]_H)\in \scrC_{d,k}$ the corresponding object in the category, with $\widetilde{\FX}_{d,k}(H)=(\FX_{d,k}(H),\sm_H,\sg_H)$.
Recall that a multicell in $\widetilde{\FX}_{d,k}(H)$ is of the form $[K_{\widehat{J}}g]_H$ for some $J=\{i_0,\ldots,i_j\}\subseteq \ls d\rs$ and $g\in G_{d,k}$, and that the corresponding cell in $\FX_{d,k}(H)$ is $\{[K_{\widehat{i_0}}g]_H,\ldots,[K_{\widehat{i_j}}g]_H\}$. In particular, it was shown in Lemma \ref{lem:well_defind_map_for_double_cosets} that $[K_{\widehat{J}}g]_H = [K_{\widehat{J}}g']_H$ implies $[K_{\widehat{i_r}}g]_H = [K_{\widehat{i_r}}g']_H$ for every $0\leq r\leq j$. The other direction does not always hold, i.e., it is possible that $[K_{\widehat{i_r}}g]_H = [K_{\widehat{i_r}}g']_H$ for every $0\leq r\leq j$, but $[K_{\widehat{J}}g]_H \neq [K_{\widehat{J}}g']_H$. This is in fact the source of multiplicity in the multicomplex, namely, the multiplicity of a cell $\{[K_{\widehat{i_0}}g]_H,\ldots,[K_{\widehat{i_j}}g]_H\}$ equals to the number of equivalence class $[K_{\widehat{J}}g']_H$ with $J=\{i_0,\ldots,i_j\}$ whose corresponding cell is $\{[K_{\widehat{i_0}}g]_H,\ldots,[K_{\widehat{i_j}}g]_H\}$. 

\begin{corollary} \label{cor:what_does_it_mean_to_be_a_complex}
	$(\widetilde{\FX}_{d,k}(H),\gamma_H,\om_H,[e]_H)\in \scrC_{d,k}$ is a complex, namely, $\sm_H\equiv 1$, if and only if for every choice of $J=\{i_0,\ldots,i_j\}$ and $g,g'\in G_{d,k}$
\begin{equation}
	[K_{\widehat{J}}g]_H = [K_{\widehat{J}}g']_H \qquad \Leftrightarrow \qquad 
	[K_{\widehat{i_r}}g]_H = [K_{\widehat{i_r}}g']_H,\,\,\forall 0\leq r\leq j.
\end{equation}
\end{corollary}

\begin{definition}
	We say that $H \leq G_{d,k}$ satisfies the intersection property if for every $0\leq j\leq d$, every $J=\{i_0,\ldots,i_j\}\subseteq \ls d\rs$ and $g\in G_{d,k}$ 
	\begin{equation}\label{eq:intersection_property}
		\bigcap_{r=0}^j \CA_{[K_{\widehat{i_r}}g]_H} =  \CA_{[K_{\widehat{J}}g]_H},
	\end{equation}
	where for a multicell $[K_{\widehat{J}}g]_H$ we denote 
	\begin{equation}
		\CA_{[K_{\widehat{J}}g]_H} = \{g'\in G_{d,k} ~:~ [K_{\widehat{J}}g]_H = [K_{\widehat{j}}g']_H\}.
	\end{equation}
	Note that for $0$-cells $\CA_{[K_{\widehat{i}}g]_H}$ for the set used to represent the $d$-complex $\FX_{d,k}(H)$ as a nerve complex. 
\end{definition}

\begin{proposition}\label{prop:condition_for_being_a_complex}
	Let $H \leq G_{d,k}$. Then, $\widetilde{\FX}_{d,k}(H)$ is a simplicial complex if and only if $H$ satisfies the intersection property. In particular, if $H\triangleleft G_{d,k}$, then $\widetilde{\FX}_{d,k}(H)$ is a simplicial complex if and only if for every choice of $j$ and $J$ as above $\bigcap_{r=0}^j \CA_{[K_{\widehat{i_r}}]_H}= \CA_{[K_{\widehat{J}}]_H}$, i.e., one need to check \eqref{eq:intersection_property} only for $g=e$. 
\end{proposition}

\begin{proof}

		First, assume that $\widetilde{\FX}_{d,k}(H)$ is a complex and note that by Lemma \ref{lem:well_defind_map_for_double_cosets}, $\CA_{[K_{\widehat{J}}g]_H} \subseteq \bigcap_{r=0}^j \CA_{[K_{\widehat{i_r}}g]_H}$ for every $J=\{i_0,\ldots,i_j\}\subset \ls d\rs$ and $g\in G_{d,k}$. Thus, it suffices to prove inclusion in the other direction. To this end, let $J=\{i_0,\ldots,i_j\}\subset \ls d\rs$ and $g\in G_{d,k}$, and assume that $g' \in  \bigcap_{r=0}^j \CA_{[K_{\widehat{i_r}}g]_H}$. Then  $[K_{\widehat{i_r}}g']_H=[K_{\widehat{i_r}}g]_H$ for $0\leq r\leq j$ and hence, by Corollary \ref{cor:what_does_it_mean_to_be_a_complex}, $[K_{\widehat{J}}g]_H = [K_{\widehat{J}}g']_H$, and in particular $g'\in \CA_{[K_{\widehat{J}}g]_H}$. 
		
	Next, assume that $H$ satisfies the intersection property and let $g,g'\in G_{d,k}$ and $J=\{i_0,\ldots,i_j\}\subset \ls d\rs$ such that $[K_{\widehat{i_r}}g]_H = [K_{\widehat{i_r}}g']_H$ for every $0\leq r\leq j$, that is $g'\in \CA_{[K_{\widehat{i_r}}g]_H}$ for every $0\leq r\leq j$. The intersection property implies that $g'\in \CA_{[K_{\widehat{J}}g]_H}$, i.e., $[K_{\widehat{J}}g]_H = [K_{\widehat{J}}g']_H$,	and hence, by Corollary \ref{cor:what_does_it_mean_to_be_a_complex}, that $\widetilde{\FX}_{d,k}(H)$ is a complex. 
	
	Finally, assume that $H\triangleleft G_{d,k}$ and that $H$ satisfies the intersection property for every $J=\{i_0,\ldots,i_j\}\subset \ls d\rs$ with respect to neutral element $e$. Since $H$ is a normal subgroup, there is a natural action of $G_{d,k}$ from the right on the cosets (which are the group elements of $G_{d,k}/H$) and hence for an arbitrary element $g\in G_{d,k}$ we have 
	\begin{equation}
		[K_{\widehat{J}}g]_H = [K_{\widehat{J}}]_H.g^{-1} = \bigcap_{r=0}^j [K_{\widehat{i_r}}]_H.g^{-1} = \bigcap_{r=0}^j [K_{\widehat{i_r}}g]_H,
	\end{equation}
	which proves the intersection property and hence that $\widetilde{\FX}_{d,k}(H)$ is a complex. 
\end{proof}


\subsection{Regularity of the multicomplex}
By Claim \ref{clm:degree_divides_k}, the degree of every $(d-1)$-multicell in an object of $\scrC_{d,k}$ always divides $k$. We turn to discuss which conditions on $H$ guarantee that these degrees are exactly $k$.

\begin{lemma}
	Given $H\leq G_{d,k}$, the multicomplex $(\widetilde{\FX}_{d,k}(H),\gamma_H,\om_H,[e]_H)$ is upper regular of degree $k$ if and only if 
	\begin{equation}\label{eq:condition_for_regularity}
		\langle \alpha_i\rangle \cap gHg^{-1} = \{e\},\qquad \forall i\in \ls d\rs,\,g\in G_{d,k}.
	\end{equation}
	In particular, if $H$ is a normal subgroup of $G_{d,k}$, the complex is upper $k$-regular if and only if $\langle \alpha_i \rangle \cap H = \{e\}$ for every $i\in \ls d\rs$, which happens if and only if for every $i\in \ls d\rs$, the order of the image of $\alpha_i$ in $G_{d,k}/H$ is exactly $k$.
\end{lemma}

\begin{proof}
	Let $H\leq G_{d,k}$. Recall that a $(d-1)$-multicell of $\widetilde{\FX}_{d,k}(H)$ is of the form $[K_{i}g]_H$ for some $i\in \ls d\rs$ and $g\in G_{d,k}$, and that by Claim \ref{clm:the_comulticells_of_a_d-1_multicell}, the $d$-multicells containing it are $\{[\alpha_i^lg]_H\}_{l\in \ls k-1\rs}$. Consequently, the degree of $[K_{\widehat{i}}g]_H$ is at most $k$, and is precisely $k$ if and only if 
\begin{equation}
	[\alpha_i^lg]_H\neq [\alpha_i^r g]_H,\qquad \forall r,l\in \ls k-1 \rs \text{ such that }r\neq l.
\end{equation}	
	However, for distinct $r,l\in \ls k-1 \rs$ we have $[\alpha_i^lg]_H = [\alpha_i^rg]_H$ if and only if $\alpha_i^{l-r}\in gHg^{-1}$. 
\end{proof}


\subsection{Multicomplexes with full skeleton}

The skeleton is complete if and only if for every $0\leq j\leq d$, every $J=\{i_0,\ldots,i_j\}\subset \ls d\rs$ and every $g_0,\ldots,g_j \in G_{d,k}$, one can find $g\in G_{d,k}$ such that 
\begin{equation}
	[K_{\widehat{i_r}}g_r]_H = [K_{\widehat{i_r}}g]_H,\qquad \forall 0\leq r\leq j. 
\end{equation}

The proof of this fact is left to the reader.


\section{Examples}\label{sec:examples}

\subsection{The subgroup $M_{d,k}$}

For a word $w=\alpha_{i_m}^{l_m}\ldots \alpha_{i_1}^{l_1}$ representing an element in $G_{d,k}$ and $i\in \ls d\rs$ define $\theta_i(w)$ to be the sum of the exponents of $\alpha_i$ in $w$ modulo $k$, namely, $\theta_i(w) = \sum_{j=1}^{m} l_j\ind_{i_j=i} \text{ mod }k$. Since $\alpha_i^k$ is the trivial group element for every $i\in \ls d\rs$, it follows that $\theta_i(w)=\theta_i(w')$ for any pair of words $w,w'$ representing the same group element. Thus we can define $\theta_i(g)$, for $g\in G_{d,k}$ and $i\in \ls d\rs$, using any word representing it. Also, note that $\theta_i$ for $i\in \ls d\rs$ is a homomorphism from $G_{d,k}$ to $\BZ/k\BZ$. 

\begin{proposition}
Let 
\begin{equation}
	M_{d,k} = \{g\in G_{d,k} ~:~ \theta_i(g)=0,\,\forall i\in\ls d\rs\}.
\end{equation}
Then, $M=M_{d,k}$ is a normal subgroup of $G_{d,k}$ of index $k^{d+1}$, satisfying $G_{d,k}/M_{d,k}\cong (\BZ/k\BZ)^{d+1}$. In addition, $\widetilde{\FX}_{d,k}(M_{d,k})$ is isomorphic the complete $(d+1)$-partite $d$-complex all of whose parts have size $k$. 
\end{proposition}

\begin{proof}
	Since $M_{d,k}=\bigcap_{i\in \ls d\rs} \ker \theta_i$, the group $M_{d,k}$ is normal. Next, we show that for every $J=\{i_0,\ldots,i_j\}\subset \ls d\rs$ and $g,g'\in G_{d,k}$ 
	\begin{equation}\label{eq:the_subgroup_M}
		[K_{\widehat{J}}g]_M = [K_{\widehat{J}}g']_M \qquad \Leftrightarrow \qquad 
		\theta_i(g)=\theta_i(g'),\, \forall i\in J.
	\end{equation}
	Indeed, assume that $[K_{\widehat{J}}g]_M=[K_{\widehat{J}}g']_M$ for some $g,g'\in G_{d,k}$. Then there exist $w\in K_{\widehat{J}}$ and $h\in M$ such that $g=wg'h$. Since $\theta_i(h'h'')=\theta_i(h')+\theta_i(h'') \text{ mod }k$ for every $i\in \ls d\rs$ and $h',h''\in M$, it follows that for every $i\in J$
	\begin{equation}
		\theta_i(g) = \theta_i(wg'h)=\theta_i(w)+\theta_i(g') + \theta_i(h) = \theta_i(g') \quad\text{ mod }k,
	\end{equation}
	where for the last equality we used the definition of $M$ and the fact that $\theta_i(w)=0$ for any $w\in K_{\widehat{J}}\subset K_{\widehat{i}}$.
	
	In the other direction, assume that $\theta_i(g)=\theta_i(g')$ for every $i\in J$, then $\theta_i(g'g^{-1})=0$ for every $i\in J$. Denoting $w=\prod_{i\in \widehat{J}} \alpha_i^{-\theta_i(g'g^{-1})}\in K_{\widehat{J}}$, it follows that $g^{-1}wg'\in M$, that is $g'=w^{-1}gh$ for some $h\in M$, which implies $[K_{\widehat{J}}g]_M = [K_{\widehat{J}}g']_M$. 

	Combining \eqref{eq:the_subgroup_M} and Proposition \ref{prop:condition_for_being_a_complex} we conclude that $\widetilde{\FX}_{d,k}(M)$ is a complex, i.e., has no multiplicity. 
	
Finally, we turn to study the structure of the complex. The $0$-cells of the complex are given by $[K_{\widehat{i}}g]_M$ for some $i\in \ls d\rs$ and $g\in G_{d,k}$. Hence, using \eqref{eq:the_subgroup_M}, there are exactly $k$ distinct $0$-cells of each color. Since the complex $\FX_{d,k}(M)$ is pure and colorable, in order to show that it is the complete $(d+1)$-partite $d$-complex, it is enough to show that the $d$-cells of $\FX_{d,k}(M)$ are all the possible $d$-cells. To this end, choose one $0$-cells from each color. Due to \eqref{eq:the_subgroup_M}, we can assume without loss of generality that the $0$-cell of color $i$ is of the form $[K_{\widehat{i}}\alpha_i^{r_i}]_M$ for some $r_i\in \ls k-1\rs$. It now follows from \eqref{eq:defn_og_Phi_j} that the $d$-cell $[\alpha_d^{r_d}\ldots\alpha_1^{r_1}\alpha_0^{r_0}]_M$ is a $d$-cell with the chosen $0$-cells. 
\end{proof}


\subsection{Coxeter complexes}

Let $W$ be a group and $S=\{s_i ~:~ i\in I\}\subset W$ a set of generators of $W$, all of its elements are of order $2$. For each pair $(s,t)\in S\times S$, let $m_{st}$ denote the order of $st$. The matrix $M=(m_{st})_{s,t\in S}$ is an $S\times S$ symmetric matrix with entries in $\BN\cup\{\infty\}$, with $1$'s on the diagonal and entries which are strictly bigger than $1$ off the diagonal. A group $W$ is called a Coxeter group, if there exists $S$ as above such that $\langle S |R \rangle$ is a presentation of $W$, where $R=\{(st)^{m_{st}} ~:~ s,t\in S\}$. In this case, we call $M$ the Coxeter matrix of $(W,S)$. 

Coxeter groups have been extensively studied, see for example \cite{BB06,Dav08} and the references therein. Some of their basic properties are summarized in the following lemma.

\begin{lemma}\label{lem:coxeter_lem}[\cite{Dav08} Theorem 4.1.6] Let $\langle W,S=(s_i)_{i\in I}\rangle$ be a Coxeter group.
\begin{itemize}
	\item For every $J\subseteq I$, the group $W_J = \langle s_i ~:~ i\in J\rangle$ is a Coxeter group. 
	\item For every $J,J'\subseteq I$, we have $W_J \cap W_{J'} = W_{J\cap J'}$. 
	\item For every $J,J'\subseteq I$, we have $\langle W_J,W_{J'}\rangle = W_{J\cup J'}$. 
\end{itemize}
\end{lemma}

\paragraph{The Coxeter complex.} To every Coxeter group $\langle W,S\rangle$ one can associate a simplicial complex $X(W,S)$, called the Coxeter complex, via the following procedure: The $0$-cells of the complex are defined to be the cosets $W_{\widehat{i}}w$ for $i\in I$ and $w\in W$, where $\widehat{i}=I\setminus \{i\}$. For every coset $W_{\widehat{J}}w$, with $J=\{i_0,\ldots,i_j\}\subseteq I$ add to the complex the cell $\{W_{\widehat{i_0}}w,\ldots,W_{\widehat{i_j}}w\}$. 

\begin{remark}
	In certain books (see e.g. \cite{Dav08}) the Coxeter complex is defined somewhat differently, but the outcome is simply the barycentric subdivision of the above complex. 
\end{remark}

\begin{proposition}
	Let $(W,S)$ be a Coxeter group. Let $\varphi$ be a one to one map from $\{\alpha_0,\ldots,\alpha_{|S|-1}\}$ onto $S$, hence inducing an epimorphism $\widetilde{\varphi}$ from $G_{|S|-1,2}$ onto $W$, and let $H=\ker \widetilde{\varphi}$. Then 
	\begin{equation}
		\widetilde{\FX}_{|S|-1,2}(H) = \FX_{|S|-1,2}(H) = X(W,S). 
	\end{equation}
	that is, Coxeter complexes are quotients of $T_{|S|-1,2}$. 
\end{proposition}

\begin{proof}
	This follows from the fact that $G_{|S|-1,2}/H = W$ together with Proposition \ref{prop:condition_for_being_a_complex} and Lemma \ref{lem:coxeter_lem}. 
\end{proof}


\subsection{The flag complexes $S(d,q)$}

Let $d\geq 3$ and $q\in \BN$ a prime power. Let $\BF=\BF_q$ be the field with $q$-elements and $V=\BF^d$. The flag complex $S(d,q)$ is defined to be the $(d-2)$-dimensional complex whose $0$-cells are non-trivial proper subspaces of $V$ and $\{W_0,\ldots,W_i\}$ is an $i$-cell of $S(d,q)$ if and only if, up to reordering of the vertices $W_0\subsetneq W_1 \subsetneq  \ldots \subsetneq W_i$. Note that $S(d,q)$ comes with a natural coloring $\gamma$ of its $0$-cells by $(d-1)$ colors, namely, $\gamma(W)=\dim(W)$ for $W\in S(d,q)^0$. Note that due to the definition of the cells in $S(d,q)$, the coloring $\gamma$ is indeed a coloring in the sense of Section \ref{sec:multicomplexes}, which makes $S(d,q)$ into a colorable complex. Furthermore, note that the degree of any $(d-3)$-cell $\si$ in $S(d,q)$ is $q+1$, the number of ways ``to complete'' the sequence of subspaces defining $\si$, to a full sequence of subspaces. As a result we can consider $S(d,q)$ with the coloring $\gamma$ and some ordering $\om$ to be an element of the category $\scrC_{d-2,q+1}$. 

Unfortunately, we do not know of any natural choice for the ordering of $S(d,q)$, which makes it difficult to give a simple description of the subgroup $H_{\text{flag}}\leq G_{d,k}$ associated with it. Note that by Theorem \ref{thm:from_subgroups_to_multicomplexes} and the action of $GL_d(\BF_q)$ on the $(d-2)$-cells we have
we have
\begin{equation}
	PGL_d(\BF_q)/B \cong S(d,q)^{d-2} \cong G_{d-2,q+1}/H_{\text{flag}},
\end{equation}
where $B$ is the Borel subgroup of $PGL_d(\BF_q)$. In particular, the index of $H_{\text{flag}}$ which equals the number of $(d-2)$-cells in $S(d,q)$, i.e., the number of complete flags in $\BF_q^d$, is $\prod_{i=2}^d \frac{q^i-1}{q-1}$.


\subsection{Bruhat-Tits buildings of type $\widetilde{A}_d(q)$}

Let $F$ be a local non-Archimedean field with residue class $q$. Let $\FG = PGL_{d+1}(F)$ and let $B(F)=B_d(F)$ the Bruhat-Tits building associated with it. This is a building of type $\widetilde{A}_d(q)$. The building $B(F)$ can be defined in several different ways. The following description shows that the link of every vertex of $B(F)$ is the complex $S(d+1,q)$ described in the previous subsection and in particular that $B(F)$ is $(q+1)$-regular. Let $\CO$ be the ring of integers in $F$. For every $F$-basis $\{\beta_0,\ldots,\beta_d\}$ of $F^{d+1}$ consider the $\CO$-lattice $L=\CO\beta_0+\ldots +\CO \beta_d$. We say that two such lattices $L_1$ and $L_2$ are equivalent if there exists $0\neq \mu\in F$ such that $L_1 = \mu L_2$. The set of equivalence classes of lattices is the vertex set of $B(F)$. Put an edge between $[L_1]$ and $[L_2]$ if there exist representatives $L_1'\in [L_1]$ and $L_2'\in [L_2]$ such that $\pi L_1'\subset L_2'\subset L_1'$, where $\pi$ is a uniformizer of $\CO$ (e.g. for $F=\BQ_p$ and $\CO = \BZ_p$ the $p$-adic integers, $\pi$ can be taken to be $p$). The simplicial complex $B(F)$ is then defined to be the clique complex with respect to the above adjacency relation. 

It is known, see \cite{Bro98}, that $B(F)$ is a locally finite, contractible $(d-1)$-dimensional simplicial complex, all of its $(d-2)$-cells are of degree $q+1$. Furthermore, the links of its vertices are all isomorphic to $S(d+1,q)$ and so it is link-connected. Thus by Theorem \ref{thm:main_theorem} is isomorphic to $T_{d-1,q+1}/H$ for a suitable $H$. 

This quotient can be used to answer in the negative a question asked in \cite{PR12}. In order to describe the question some additional definitions are needed. For a $d$-dimensional simplicial complex $X$, let $\Om^{d-1}(X;\BR)$ be the vector space of $(d-1)$-forms
\begin{equation}
	\Om^{d-1}(X;\BR) = \{f:X^{d-1}_\pm \to \BR ~:~ f(\overline{\si}) = -f(\si),\quad \forall \si\in X^{d-1}_\pm\},
\end{equation} 
where $X^{d-1}_\pm$ is the set of oriented $(d-1)$-cells and $\overline{\si}$ is the oriented cell $\si$ with the opposite orientation to $\si$. 

Define $\Delta^+$ to be the linear operator on $\Om^{d-1}(X;\BR)$ given by 
\begin{equation}
	\Delta^+ f(\si) = \deg (\si) f(\si) - \sum_{\si' \sim\si} f(\si'),
\end{equation}
where two oriented $(d-1)$-cells $\si,\si'$ are called neighbors (denoted $\si\sim\si'$) if their union is a $d$-cell on which they induce the opposite orientation. 

\begin{remark}$~$
\begin{itemize}
	\item $\Delta^+$ defined above is the same operator as the one obtained via the composition of the appropriate boundary and coboundary operators, see \cite{PR12}.
	\item If $X$ is a graph, i.e., $d=1$, this recovers the standard graph Laplacian. 
\end{itemize}
\end{remark}

Finally, we define the spectral gap of $\Delta^+$ to be the minimum of the spectrum on the non-trivial eigenfunctions, that is 
\begin{equation}
	\lam(X) = \min \textrm{Spec}(\Delta^+|_{(B^{d-1})^\bot}),
\end{equation}
where $B^{d-1}$ denotes the space of $(d-1)$ $\BR$-coboundaries. We refer the reader to \cite[Section 3]{PR12} for additional details and further discussion on the spectral gap. 

Turning back to the question: If $X$ and $Y$ are graphs and $\pi:X\to Y$ is a covering then $\lam(X)\geq \lam(Y)$. If $X$ is connected, this fact can be seen from the interpretation of $\lam(X)$ as $\limsup_{n\to\infty} k- \sqrt[n]{\text{number of closed paths from }v_0 \text{ to iteself}}$, where $v_0$ is any vertex in the graph. Indeed, any closed path in $X$ from $v_0$ to itself corresponds to a closed path in $Y$, so $\lam(X)$ is as least as big as $\lam(Y)$. 

In \cite{PR12} is was asked whether the same holds for general simplicial complexes, i.e., if $X$ and $Y$ are $d$-dimensional simplicial complexes and $\pi:X\to Y$ is a covering map, does $\lam(X)\geq \lam(Y)$?

We can now see that this is not the case. The spectrum of $\Delta^+$ of $T_{d,k}$ was calculated in \cite{PR12}, from which one can see that
\begin{equation}
	\lam(T_{d,k}) = \begin{cases}
		0 & 2\leq k\leq d\\
		k+d-1-2\sqrt{d(k-1)} & k\geq d+1
	\end{cases}.
\end{equation}
On the other hand, for the building $B(F)=B_2(F)$ of $PGL_3(F)$ and of its finite Ramanujan quotients, it was shown in \cite{GP14} that 
\begin{equation}
	\lam(B(F))= q+1-2\sqrt{q+1}. 
\end{equation}
Thus for $d=2$ and $k=q+1$ we obtain that $\lam(B(F))>\lam(T_{2,k})$, which implies a negative answer to the question above. 

\begin{remark}
	The counter example $\pi:T_{2,q+1} \to B_2(F)$ can be replaced by a finite counter example: Indeed, using a deep result of Lafforgue \cite{Laf02}, it was shown in \cite{LSV05} that when $F$ is of positive characteristic, $B_2(F)$ has finite quotients which are Ramanujan complexes. For such a complex $Y$, the spectrum of the upper Laplacian is contained in the spectrum of the upper Laplacian of $B_2(F)$ (see also \cite{Fir16}). Hence, $\lam(Y)\geq q+1-2\sqrt{q+1}$. On the other hand, by a general result in the same spirit as \cite[Corollary 3.9]{PR12}  $T_{2,q+1}$ can be approximated by finite quotients $X_i$ of it, whose spectrum converges to that of $T_{2,q+1}$. Moreover, one can choose $X_i$ so that for each one of them there is a surjective morphism $\pi_i:X_i \to Y$. Indeed, let $H$ be the finite index subgroup of $G=G_{2,q+1}$ associated with $Y$, and then take $H_i$ to be a decreasing chain of finite index subgroups of $H$ with $\bigcap_i H_i = \{e\}$. Such a chain exists since $G$,  and hence also $H$, is residually finite. Then $X_i = \widetilde{\FX}_{d,k}(H_i)$ will have this property. Therefore, for large enough $i$, $\lam(X_i)\leq \lam(T_{2,q+1})+\varepsilon = q+2-\sqrt{2q}+\varepsilon < 1+q-2\sqrt{q}\leq \lam(Y)$. 
\end{remark}


\section{Random multicomplexes} \label{sec:random_multicomplexes}

The correspondence established in Sections \ref{sec:from_subgroup_to_elements_of_the_category} and \ref{sec:from_simplicial_complexes_to_subgroups_and_back} between the link-connected objects in $\scrC_{d,k}$ with $n$ distinct $d$-multicells and the subgroups of $G_{d,k}$ of index $n$ enables us to present a convenient model for random elements of $\scrC_{d,k}$. 

The subgroups of index $n$ are in $1$ to $(n-1)!$ correspondence with the transitive action of the group $G_{d,k}$ on the set $[n]$ of $n$ elements. This is a general fact that holds for every group $G$. Indeed, for every such transitive action of $G$, let $H$ be the stabilizer of $1\in [n]$, which is an index $n$ subgroup since the action is transitive. Conversely, if $H$ is an $n$-index subgroup of $G$, then the action of $G$ gives a transitive action on a set with $n$ elements - the left cosets of $H$ in $G$. Now, every bijection from this set of cosets of $H$ to $[n]$ which sends $H$ to $1$ gives rise to a transitive action of $G$ on $[n]$ with $H$ being the stabilizer of $1$. There are $(n-1)!$ such bijections. The reader is referred to \cite{LS03} for an extensive use of this argument. 

As $G_{d,k}$ is a free product of $(d+1)$ cyclic groups $K_i = \langle \alpha_i | \alpha_i^k=e\rangle$, a homomorphism $\varphi:G_{d,k}\to \CS_{[n]}$ is completely determined by the images of $(\alpha_i)_{i\in \ls d\rs}$ . Each such $\varphi(\alpha_i)$ should be a permutation of order $k$, namely, a product of disjoint cycles of lengths dividing $k$. Conversely, every choice of $(d+1)$ such permutations $\beta_0,\ldots,\beta_d$ determines a unique homomorphism from $G_{d,k}$ to $\CS_{[n]}$. These homomorphisms are not necessarily transitive but they are so with high probability, when chosen independently and uniformly at random, as long as $d\geq 2$ or $d=1$ and $k\geq 3$. This can be proved directly by estimating the probability of a proper subset $A\subsetneq [n]$ to be invariant under such permutation, and summing over all possible $A$'s. 

So, all together a typical choice of such $\beta_i$'s leads to a $k$-regular multicomplex. 

We plan to come to the study of this random model of multicomplexes in a followup paper. 

\begin{remark}
	As mentioned before, the results of this paper are true also for $k=\infty$, in which case we get a random model for all $d$-dimensional colorable, link-connected multicomplexes by choosing $(d+1)$ random permutations of $\CS_{[n]}$. 
\end{remark}

\begin{appendix}

\section{Regular graphs as Schreier graphs}\label{sec:regular_graphs_as_schreier_graphs}
\label{sec:appendix}

The concept of a Cayley graph is a very important one for graph theory and group theory alike.
Let $G$ be a group and let $S$ be a set of generators such that the identity element $e$ is not in $S$. Assume, moreover, that $S$ is symmetric: $s \in S \Leftrightarrow s^{-1} \in S$. Then $Y:=\text{Cay}(G;S)$, the Cayley graph of $G$ with respect to $S$, is defined by $V(Y)=G$ and $E(Y) = \{\{g,sg\}:g \in G, s \in S\}$. Note that $Y$ does not contain loops or multiple edges. Cayley graphs are $|S|$-regular graphs, but the overwhelming majority of regular graphs are not Cayley graphs. 

It is natural to seek an algebraic construction for regular graphs that can generate all or almost all regular graphs. This is the motivation for the study of Schreier graphs, which are a vast generalization of Cayley graphs. There are several competing definitions of a Schreier graph in the literature, each with its own advantages and disadvantages.

Let $G$ be a group, $H \leq G$ a subgroup and $S$ a set of generators as above. One way to define the Schreier graph $X := \text{Sch}(G/H;S)$ is to take $V = G/H$ and $E = \{\{gH,sgH \}:g \in G,s \in S\}$. Note that $E$ is a set, rather than a multiset, and so $X$ is a graph (although it may contain loops). The disadvantage is that $X$ may not be regular. In general, all that can be said about the vertex degrees is that they are between $1$ and $|S|$.

Another possibility is to take $E$ to be a multiset, with one edge between $gH$ and $sgH$ for every $s\in S$. Using this definition, $X$ is a $2|S|$-regular multigraph, and in fact every regular multigraph of even degree is a Schreier graph in this sense \cite{Gr77}. The problem is that this definition does not allow one to generate regular graphs and multigraphs of odd degree.

Consistent with our approach in this paper, we prefer what we believe is the most natural definition, i.e., we define $X$ as the quotient of the Cayley graph $Y = \text{Cay}(G;S)$ by the action of the subgroup $H$. The group $G$ acts on $Y$ from the right by multiplication, and moreover this action preserves the graph structure, that is, $\{y_1,y_2\}.g := \{y_1.g^{-1},y_2.g^{-1}\}$ is an edge in $Y$ if and only if $\{y_1,y_2\}$ is an edge. We define now $X = \text{Sch}(G/H;S)$ to be the multigraph whose vertices $V(X)$ are the orbits of vertices of $Y$, namely $G/H$, and whose  multiedges, $E(X)$, are the orbits of $Y$'s edges with respect to the action of $H$, where an orbit $[\{y_1,y_2\}]_H:=\{\{y_1,y_2\}.h:h \in H\}$ connects the vertices $y_1 H$ and $y_2 H$. Thus, $E(X)$ is a multiset which may contain loops and multiple edges. The natural projection from $G$ to $G/H$ induces a surjective graph homomorphism from $Y$ to $X$.

Let us now define $2$-factors and perfect matchings in a multigraph, in a way which is perhaps not the most common in the literature: A $2$-factor (respectively, a perfect matching) in a multigraph $\mathcal{G} = \langle V,E \rangle$ is a multisubset $F \subseteq E$ of edges such that every vertex $v \in V$ either belongs to a unique loop in $F$ or to exactly two edges in $F$ (respectively, a unique edge in $F$).

\begin{proposition} \label{prop:schreier}
A connected $k$-regular multigraph $X$ is a Schreier graph if and only if its edges form a disjoint union of perfect matchings and $2$-factors. In fact, it is a union of $m$ perfect matchings and $f$ $2$-factors with $m+2f = k$.
\end{proposition}
\begin{proof}
We first show that every graph $X = \langle V,E \rangle$ whose edge set is the disjoint union of perfect matchings and $2$-factors is a Schreier graph. The proof is essentially the same as the proof of Proposition 1.1.

Let $E = \bigcup_i F_i$, where $\{F_i\}_i$ are pairwise disjoint perfect matchings and $2$-factors. By choosing a cyclic ordering for each cycle in such a $2$-factor, we can associate each $F_i$ with a permutation $\sigma_i \in \CS_{V}$. If $F_i$ is a perfect matching, then $\sigma_i$ is an involution.

Let $S = \{\sigma_i\}_i$, let $G_0$ be the subgroup of $\mathbb{S}_{V}$ generated by $S$, and set $H_0$ to be the stabilizer of some vertex $v_0$. Then one can check that $X \cong \text{Sch}(G_0/H_0;S)$. 

Conversely, we show that the edges of a Schreier graph form a disjoint union of perfect matchings and $2$-factors.

Let $G$ be a group, let $H \leq G$, and let $S$ be a symmetric set of generators not containing the identity element $e$. Let $S_2 = \{s \in S: s^2 = e\}$, and let $S' = S \setminus S_2$. We color the edges of the Cayley graph $Y = \text{Cay}(G;S)$ as follows: There is a color $c_s$ for each element $s \in S_2$, a color $c_{s'}=c_{s'^{-1}}$ for each pair of elements $\{s',s'^{-1}\} \subseteq S'$, and we color an edge $\{g,rg\}$ using the color $c_s$ if $r \in \{s,s^{-1}\}$. Note that every $y \in V(Y)$ is adjacent to one edge of color $c_s$ if $s \in S_2$, and to two edges of color $c_s$ if $s \in S'$. Let $F_s$ be the set of edges of color $c_s$.

Consider now the Schreier graph $X = \text{Sch}(G/H;S)$. Note first that the action of $G$ on the edges of $Y$ preserves the coloring. Thus, the coloring of the edges of $Y$ induces a coloring of the edges of $X$. Consequently, the images $\bar{F_s}$ of $F_s$ are disjoint from each other for different colors.
We will show that if $s \in S_2$, then $\bar{F_s}$ is a perfect matching, while for $s \in S'$ it is a $2$-factor. As these sets are disjoint, this will complete the proof.

Note that if an edge of $X$ of color $c_s$ is adjacent to a vertex $gH$, then it must connect $gH$ to either $sgH$ or to $s^{-1}gH$.
Now, for $s \in S_2$ and a given vertex $gH$ of $X$, there is a unique edge attached to it of color $c_s$, namely $\{gH,sgH\}$. If $gH = sgH$ (i.e., $g^{-1}sg \in H$), then this is a loop, but in either case $\bar{F_s}$ is a perfect matching.

Finally, fix $s \in S'$ and a vertex $gH$ in $X$.The image of $F_s$ gives rise to potentially two edges coming out of $gH$, $\{gH,sgH\}$ and $\{gH,s^{-1}gH\}$. Now, assume first that $gH = sgH$. In this case we also have $gH = s^{-1}gH$, and so $F_s$ induces a loop around $gH$. Note that this is a single loop, rather than a double loop, since here $g^{-1}s^{-1}g \in H$ and therefore $[\{g,sg\}]_H = [\{g,s^{-1}g\}]_H$.
If $sgH = s^{-1}gH$, but they are different from $gH$, then $F_s$ induces a cycle of length $2$ between $gH$ and $sgH$. Note that in this case the two edges do not coincide, because $[\{g,sg\}]_H = [\{g,s^{-1}g\}]_H$ implies that either $s = s^{-1}$ or $gH = sgH$.
If the three vertices $gH,sgH$, and $s^{-1}gH$ are all different, then the degree of $gH$ in $\bar{F_s}$ is clearly $2$. Hence, we deduce that $\bar{F_s}$ is a $2$-factor as required.

\end{proof}

We now observe that most, but not all, regular graphs are Schreier graphs.

\begin{proposition}\label{prop:bipartite_graphs}$~$
\begin{enumerate}
\item Let $G = \langle U \amalg V,E\rangle$ be a $k$-regular bipartite multigraph. Then $|U|=|V|$ and $E$ is the disjoint union of $k$ perfect matchings, and hence a Schreier graph.
\item For every even $k \geq 2$, every connected $k$-regular graph is a Schreier graph.
\item For $k \geq 3$, with probability $1-o(1)$ a random $k$-regular graph on $n$ vertices chosen uniformly at random is a Schreier graph. Here $k$ is fixed and $n$ tends to infinity.
\end{enumerate}
\end{proposition}

\begin{proof}$~$
\begin{enumerate}
\item
The fact that $|U|=|V|$ follows from a double count of the edges: The number of edges is equal to the sum of the degrees of vertices in $U$, so $|E|=k \cdot |V|$. By the same argument, $|E| = k \cdot |U|$, and so $|U|=|V|$.

For the second part of the proposition, it is sufficient to show that every $k$-regular bipartite graph $G = \langle U \cup V,E\rangle$ contains a perfect matching $M$. Since the remaining graph, whose edge set is $E \setminus M$, is $(k-1)$-regular, the result follows by induction on $k$.

We recall Hall's marriage theorem, which states that a bipartite graph $H = \langle U \amalg V,E\rangle$ contains a matching of cardinality $|U|$ if and only if for every $X \subseteq U$ we have $|N_H(X)| \geq |X|$, where $N_H(X)$ denotes the neighborhood of $X$ in $H$. Hall's theorem also holds for multigraphs.

Since $G$ is $k$-regular, for every set $X \subseteq U$ the number of edges adjacent to $X$ is $k\cdot|X|$. On the other hand, this is clearly a lower bound on the number of edges adjacent to $N(X)$, and so we have $k\cdot |X| \leq k \cdot |N_G(X)|$, which implies that $G$ has a matching of size $|U|$. Since $|U|=|V|$, this is a perfect matching.
\item This was proved in \cite{Gr77} with a different definition on Schreier graphs, but in fact, their proof relies on Petersen's theorem \cite{Pet91}, which states that every $k$-regular graph, for even $k\geq 2$, is an edge-disjoint union of $2$-factors. Therefore, by Proposition \ref{prop:schreier}, it is also Schreier according to our definition. 

\item When $k$ is even, all $k$-regular graphs are Schreier graphs as observed above. When $k$ is odd, the number $n$ of vertices in the graph must be even. It was shown by Wormald and Robinson (\cite{WR92}, Theorem 3) that in this case when $n$ tends to infinity, with probability $1-o(1)$ the edge set of a random $k$-regular graph on $n$ vertices has a decomposition into perfect matchings, so by Proposition \ref{prop:schreier} it is a Schreier graph.
\end{enumerate}
\end{proof}

\begin{proposition}\label{prop:not_a_Schrier_graph}
For odd $k \geq 3$, there exists a connected $k$-regular graph that is not a Schreier graph. 
\end{proposition}
\begin{proof}
This follows from a construction of a $k$-regular graph $G$ that contains neither a perfect matching nor a $2$-factor.

Let $G'$ be the complete balanced $k$-regular tree of depth $3$. That is, the root $r$ has $k$ children $x_1, ... ,x_k$, every node that is not the root or a leaf has $(k-1)$ children, and there is a path of length $3$ from each leaf to the root. Note that the degree of every vertex in $G'$ is $k$, except for the leaves, whose degree is $1$. For $1 \leq i\leq k$, let $T_i$ denote the set of vertices whose ancestor is $x_i$, and let $L_i := \{l^1_i, ... ,l^{(k-1)^2}_i\}$ be the $(k-1)^2$ leaves in $T_i$.  

We obtain a $k$-regular graph $G$ from $G'$ by adding the edges of a $(k-1)$-regular graph on the vertex set $L_i$ for every $1 \leq i \leq k$. Here we are relying on the basic fact that there exists an $r$-regular graph on $n$ vertices if and only if $n>r$ and $rn$ is even (See Figure \ref{fig:not_Sch} for an illustration of the case $k=3$). 

\begin{figure}[h]
	\begin{center}
	\includegraphics[scale=0.8]{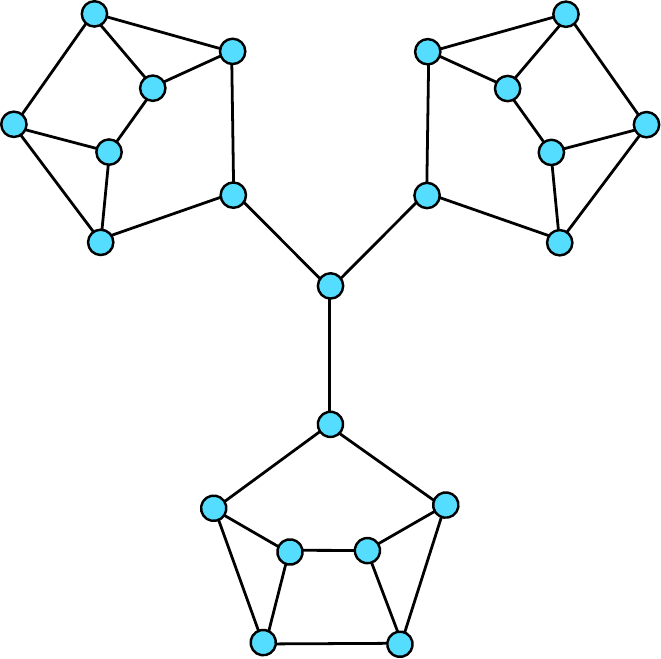}
	\caption{An illustration of a graph which is $3$-regular, but is not a Schreier graph. \label{fig:not_Sch}}
	\end{center}
\end{figure}

Now, $G$ has the following two properties:
\begin{itemize}
\item For every $1 \leq i \leq k$, the only edge from $T_i$ to $V(G)\setminus T_i$ is $\{r,x_i\}$.
\item The cardinality of $T_i$ is $1+(k-1)+(k-1)^2$, which is an odd number.
\end{itemize}
The first property implies that $G$ does not contain a $2$-factor. Indeed, any $2$-factor must have a cycle that contains the root, and the first property implies that the root does not belong to any cycle.

The second property implies that $G$ does not contain a perfect matching: Assume to the contrary that $M$ is a perfect matching in $G$. Let $x_{i_0}$ be the vertex that is matched to the root $r$ in $M$, and let $j_0 \ne i_0$. Then the vertices in $T_{j_0}$ must be matched to each other, since the only edge from $T_{j_0}$to $V(G)\setminus T_{j_0}$ is $\{r,x_{j_0}\}$. However, the cardinality of this vertex set is odd, and as in our graphs there are no loops, they cannot be matched to each other.
\end{proof}

\end{appendix}


\bibliography{Biblio}

\begin{thebibliography}{LSV05}

\bibitem[BB05]{BB06}
Anders Bj{\"o}rner and Francesco Brenti.
\newblock {\em Combinatorics of {C}oxeter groups}, volume 231 of {\em Graduate
  Texts in Mathematics}.
\newblock Springer, New York, 2005.

\bibitem[Bro98]{Bro98}
Kenneth~S. Brown.
\newblock {\em Buildings}.
\newblock Springer Monographs in Mathematics. Springer-Verlag, New York, 1998.
\newblock Reprint of the 1989 original.

\bibitem[Dav08]{Dav08}
Michael~W. Davis.
\newblock {\em The geometry and topology of {C}oxeter groups}, volume~32 of
  {\em London Mathematical Society Monographs Series}.
\newblock Princeton University Press, Princeton, NJ, 2008.

\bibitem[Fir16]{Fir16}
Uriya~A First.
\newblock The ramanujan property for simplicial complexes.
\newblock {\em arXiv preprint arXiv:1605.02664}, 2016.

\bibitem[GP14]{GP14}
Konstantin Golubev and Ori Parzanchevski.
\newblock Spectrum and combinatorics of ramanujan triangle complexes.
\newblock {\em arXiv preprint arXiv:1406.6666}, 2014.

\bibitem[Gro77]{Gr77}
Jonathan~L Gross.
\newblock Every connected regular graph of even degree is a schreier coset
  graph.
\newblock {\em Journal of Combinatorial Theory, Series B}, 22(3):227--232,
  1977.

\bibitem[Kue15]{Kue15}
Thilo Kuessner.
\newblock Multicomplexes, bounded cohomology and additivity of simplicial
  volume.
\newblock {\em Bull. Korean Math. Soc.}, 52(6):1855--1899, 2015.

\bibitem[Laf02]{Laf02}
Vincent Lafforgue.
\newblock {$K$}-th\'eorie bivariante pour les alg\`ebres de {B}anach et
  conjecture de {B}aum-{C}onnes.
\newblock {\em Invent. Math.}, 149(1):1--95, 2002.

\bibitem[Lee16]{Lee16}
Paul-Henry Leemann.
\newblock Schreir graphs: transitivity and coverings.
\newblock {\em Internat. J. Algebra Comput.}, 26(1):69--93, 2016.

\bibitem[Lei82]{Lei82}
Frank~Thomson Leighton.
\newblock Finite common coverings of graphs.
\newblock {\em J. Combin. Theory Ser. B}, 33(3):231--238, 1982.

\bibitem[LS03]{LS03}
Alexander Lubotzky and Dan Segal.
\newblock {\em Subgroup growth}, volume 212 of {\em Progress in Mathematics}.
\newblock Birkh\"auser Verlag, Basel, 2003.

\bibitem[LSV05]{LSV05}
Alexander Lubotzky, Beth Samuels, and Uzi Vishne.
\newblock Ramanujan complexes of type {$A_d$}.
\newblock {\em Israel J. Math.}, 149:267--299, 2005.

\bibitem[Pet91]{Pet91}
Julius Petersen.
\newblock Die {T}heorie der regul\"aren graphs.
\newblock {\em Acta Math.}, 15(1):193--220, 1891.

\bibitem[PR16]{PR12}
Ori Parzanchevski and Ron Rosenthal.
\newblock Simplicial complexes: Spectrum, homology and random walks.
\newblock {\em Random Structures \& Algorithms}, pages 1--37, 2016.

\bibitem[Ros14]{Ro14}
Ron Rosenthal.
\newblock Simplicial branching random walks and their applications.
\newblock {\em arXiv preprint arXiv:1412.5406}, 2014.

\bibitem[RW94]{WR92}
Robert~W. Robinson and Nicholas~C. Wormald.
\newblock Almost all regular graphs are hamiltonian.
\newblock {\em Random Structures \& Algorithms}, 5(2):363--374, 1994.

\end{thebibliography}
\bibliographystyle{alpha}

\medskip{}

$~$\\
Institute of Mathematics, Hebrew University\\
Jerusalem 91904, \\
Israel.\\
E-mail: alex.lubotzky@mail.huji.ac.il

\bigskip{}
$~$\\
Institute of Theoretical Studies,\\
ETH Z{\"u}rich, 
CH-8092 Z{\"u}rich,\\ 
Switzerland.\\
E-mail: zluria@gmail.com

\bigskip{}
$~$\\
Departement Mathematik,\\
ETH Z{\"u}rich, CH-8092 Z{\"u}rich, \\
Switzerland.\\
E-mail: ron.rosenthal@math.ethz.ch

\end{document}